\documentclass[11pt,leqno]{article}
\usepackage{amsmath,amsthm,amsfonts,amssymb}

\newcommand{\N}{{\mathbb N}}
\newcommand{\Z}{{\mathbb Z}}
\newcommand{\R}{{\mathbb R}}
\newcommand{\C}{{\mathbb C}}
\newcommand{\E}{{\mathbb E}}
\newcommand{\eps}{{\varepsilon}}
\newcommand{\tauetabar}{{(\underline{\tau},\underline{\eta})}}
\newcommand{\utau}{\underline{\tau}}
\newcommand{\ueta}{\underline{\eta}}

\newcommand{\LL}{{\mathbb L}}

\newcommand{\bA}{{\mathbb A}}

\newcommand\cA{{\cal  A}}
\newcommand\cB{{\cal  B}}

\newcommand\cH{{\cal  H}}
\newcommand\cV{{\cal  V}}
\newcommand\cW{{\cal  W}}

\newcommand\cI{{\cal  I}}
\newcommand\cR{{\cal  R}}
\newcommand\cG{{\cal  G}}

\newcommand\cE{{\cal  E}}
\newcommand\cF{{\cal  F}}

\newcommand\cO{{\cal O}}

\newcommand\cU{{\mathcal U}}

\newcommand\cL{{\mathcal L}}

\newcommand{\opeg}{{\text{\rm Op}^{\varepsilon,\gamma}}}
\newcommand{\bfS}{{\bf S}}

\newcommand{\tA}{\tilde{A}}

\newcommand{\X}{{\xi+\frac{k\, \beta}{\eps}}}

\newcommand\adots{\mathinner{\mkern2mu\raise1pt\hbox{.}
\mkern3mu\raise4pt\hbox{.}\mkern1mu\raise7pt\hbox{.}}}

\newtheorem{theo}{Theorem}[section]
\newtheorem{theorem}{Theorem}[section]
\newtheorem{prop}[theo]{Proposition}
\newtheorem{proposition}[theo]{Proposition}
\newtheorem{cor}[theo]{Corollary}
\newtheorem{lem}[theo]{Lemma}
\newtheorem{defn}[theo]{Definition}
\newtheorem{ass}[theo]{Assumption}

\newtheorem{rem}[theo]{Remark}
\newtheorem{remark}[theo]{Remark}

\newtheorem{nota}[theo]{Notations}
\newtheorem{assumption}[theo]{Assumption}
\newtheorem{definition}[theo]{Definition}

\numberwithin{equation}{section}

 \title{Nonlinear geometric optics for reflecting uniformly stable pulses}

\author{\sc \small
Jean-Francois Coulombel\thanks{CNRS and Universit\'e de Nantes, Laboratoire de math\'ematiques Jean
Leray (UMR CNRS 6629), 2 rue de la Houssini\`ere, BP 92208, 44322 Nantes Cedex 3, France. Email:
{\tt jean-francois.coulombel@univ-nantes.fr}. Research of J.-F. C. was supported by the French Agence
Nationale de la Recherche, contract ANR-08-JCJC-0132-01.},
Mark Williams\thanks{University of North Carolina, Mathematics Department, CB 3250, Phillips Hall,
Chapel Hill, NC 27599. USA. Email: {\tt williams@email.unc.edu}. Research of M.W. was partially supported
by NSF grants number DMS-0701201 and DMS-1001616.}}

\begin{document}

\maketitle

\begin{abstract}
We provide a justification with rigorous error estimates showing that the leading term in weakly nonlinear geometric
optics expansions of highly oscillatory reflecting pulses is close to the uniquely determined exact solution for small
wavelengths $\eps$. Pulses reflecting off fixed noncharacteristic boundaries are considered under the assumption
that the underlying  boundary problem is uniformly spectrally stable in the sense of Kreiss. There are two respects
in which these results make rigorous the formal treatment of pulses in Majda and Artola \cite{majdaartola}, and Hunter,
Majda and Rosales \cite{hmr}. First, we give a rigorous construction of leading pulse profiles in problems where pulses
traveling with many distinct group velocities are, unavoidably, present; and second, we provide a rigorous error analysis
which yields a rate of convergence of approximate to exact solutions as $\eps\to 0$. Unlike wavetrains, interacting
pulses do not produce resonances that affect leading order profiles. However, our error analysis  shows the importance
of estimating pulse interactions in the construction and estimation of correctors. Our results apply to a general class of
systems that includes quasilinear problems like the compressible Euler equations; moreover, the same methods yield
a stability result for uniformly stable Euler shocks perturbed by highly oscillatory pulses.
\end{abstract}

\tableofcontents

\section{Introduction}
\label{intro}

\emph{\quad}We study highly oscillatory pulse solutions for a general class of hyperbolic equations that includes
quasilinear systems like the compressible Euler equations. Our main objective is to construct leading order weakly
nonlinear geometric optics expansions of the solutions (which are valuable because, for example, they exhibit
important qualitative properties), and to rigorously justify such expansions, that is, to show that they are close in
a precise sense to true exact solutions.

A single pulse colliding with a fixed noncharacteristic boundary in an $N\times N$ hyperbolic system will generally
give rise to a family of reflected pulses traveling with several distinct group velocities. We study this situation when
the underlying  boundary problem is assumed to be uniformly spectrally stable in the sense of Kreiss. A formal
treatment of this problem was given in Majda-Artola \cite{majdaartola}, building on an earlier treatment of nonlinear
geometric optics for pulses in free space in Hunter-Majda-Rosales \cite{hmr}. In the papers \cite{majdaartola,hmr},
systems of nonlinear equations for leading order profiles were derived, but their solvability was not discussed.
Morever, the questions of the existence of exact solutions on a fixed time interval independent of the wavelength
of oscillations (or pulse width) $\eps$, and of the relation between exact and approximate solutions, were not
studied there. In this paper, we give a rigorous construction of leading pulse profiles in problems where pulses
traveling with many distinct group velocities are, unavoidably, present. In addition, we construct exact solutions
on a fixed time interval independent of $\eps$, and  provide a rigorous error analysis which yields a rate of
convergence of approximate to exact solutions as $\eps\to 0$.

Rigorous treatments of the short-time propagation of a single pulse in free space were given in Alterman-Rauch
\cite{ar2} and Gu\`es-Rauch \cite{gr}\footnote{The paper \cite{gr} considered ``fronts" as well as pulses.}.
The methods (e.g., conormal estimates in \cite{ar2,gr}, high-order approximate solutions in \cite{gr})
used in the constructions of exact solutions and in the error analyses of these papers do not readily extend to
problems involving many pulses with distinct group velocities. The method we use here to construct exact solutions
and justify leading term expansions involves replacing the original system \eqref{a1} with an associated singular
system \eqref{a3} involving coefficients of order $\frac{1}{\eps}$ and a new unknown $U_\eps(x,\theta_0)$.\footnote{The singular system approach was used in \cite{altermanrauch} in their study of a single pulse on diffractive time scales.}
Exact solutions $U_\eps$ to the singular system yield exact solutions to the original system by a substitution
\begin{align*}
u_\eps(x)=U_\eps\left(x,\frac{\phi_0(x')}{\eps}\right),
\end{align*}
where $\phi_0(x')=x'\cdot\beta$ is the ``boundary phase" as in \eqref{a1}. Both the singular system and the system
of profile equations satisfied by the leading profile $\cU^0(x,\theta_0,\xi_d)$ are solved by Picard iteration.

The error analysis is based on ``simultaneous Picard iteration", a method first used in the study of geometric optics
for wavetrains in free space in \cite{jmr}. The idea  is to show that for every $n$, the $n$-th profile iterate $\cU^{0,n}
(x,\theta_0,\frac{x_d}{\eps})$ converges as $\eps\to 0$ in an appropriate sense to the $n$-th exact iterate $U^n_\eps
(x,\theta_0)$, and to conclude therefrom that $\cU^0(x,\theta_0,\frac{x_d}{\eps})$ is close to $U_\eps(x,\theta_0)$ for
$\eps$ small. Unlike wavetrains, interacting pulses do not produce resonances that affect leading order profiles.
However, our error analysis shows the importance of estimating pulse interactions in the construction and estimation
of correctors. Another key tool in the error analysis, discussed further in section \ref{errorintro}, is the machinery of
moment-zero approximations developed in section \ref{mz}.   Our use of these approximations was inspired by the ``low-frequency cutoff" argument of
\cite{altermanrauch}.

The main novelties of this paper are:

1)\;We give a rigorous treatment of pulses reflecting off boundaries;  earlier rigorous work concerned pulses in free space.

2)\;We provide methods for handling many pulses traveling with distinct group velocities;  in particular, we show that although pulse interactions do not produce new pulses at leading order, pulse interactions must be estimated in the construction of correctors and in the error analysis.   We distinguish in the estimates between ``transversal" and ``nontransversal" pulse interactions.

3)\;In contrast to the treatment of uniformly stable reflecting wavetrains in \cite{CGW1}, we are able here to give a rate of convergence of approximate to exact solutions as wavelength $\eps\to 0$.\footnote{In the case of wavetrains there was an ``arithmetic obstacle" to obtaining a rate of convergence in the error analysis; namely, the generation of many noncharacteristic, but ``almost characteristic", phases by nonlinear interactions.  Because pulses interact weakly, that obstacle is absent in the problem studied here.}

\subsection{Exact solutions and singular systems}\label{exact1}

\emph{\quad}In order to study geometric optics for nonlinear problems with highly oscillatory solutions it is important
first to settle the question of whether exact solutions exist on a fixed time interval independent of the wavelength
($\eps$ in the notation below). A powerful method for studying this problem, introduced in \cite{jmr} for initial value
problems and extended to boundary problems in \cite{williams3}, is to replace the original system with an associated
singular system.

On $\overline{\R}^{d+1}_+ = \{x=(x',x_d)=(t,y,x_d)=(t,x''):x_d\geq 0\}$, consider the $N\times N$ quasilinear
hyperbolic boundary problem:
\begin{align} \label{a1}
\begin{split}
& \sum^d_{j=0}A_j(v_\eps) \, \partial_{x_j}v_\eps=f(v_\eps)\\
&b(v_\eps)|_{x_d=0}=g_0+\eps \, G\left(x',\frac{x'\cdot\beta}{\eps}\right)\\
&v_\eps = u_0 \text{ in } t < 0,
\end{split}
\end{align}
where $x_0=t$ is  time, $G(x',\theta_0)\in C^\infty (\R^d \times \R^1,\R^p)$ decays to zero as $|\theta_0|\to \infty$,
with supp $G\subset \{x_0\geq 0\}$, and the boundary frequency $\beta \in \R^d \setminus \{ 0\}$\footnote{Wavetrains
instead of pulses are obtained by taking $G(x',\theta_0)$ to be periodic in $\theta_0$.}. Here the coefficients $A_j \in
C^\infty(\R^N,\R^{N^2})$, $f\in C^\infty(\R^N,\R^N)$, and $b\in C^\infty(\R^N,\R^p)$.

Looking for $v_\eps$ as a perturbation $v_\eps=u_0+\eps u_\eps$ of a constant state $u_0$ such that $f(u_0)=0$,
$b(u_0)=g_0$, we obtain for $u_\eps$ the system (with slightly different $A_j$'s)
\begin{align}\label{a2}
\begin{split}
&(a)\;P(\eps u_\eps,\partial_x) u_\eps:=\sum^d_{j=0}A_j(\eps u_\eps) \, \partial_{x_j}u_\eps
=\mathcal{F}(\eps u_\eps)u_\eps\text{ on }x_d\geq 0\\
&(b)\;B(\eps u_\eps)u_\eps|_{x_d=0}=G\left(x',\frac{x'\cdot \beta}{\eps}\right)\\
&(c)\;u_\eps = 0 \text{ in } t<0,
\end{split}
\end{align}
where $B(v)$ is a $C^\infty \;p\times N$ real matrix defined by
\begin{align*}
b(u_0+\eps u_\eps)=b(u_0)+B(\eps u_\eps)\eps u_\eps
\end{align*}
and $\cF$ is defined similarly. We assume that the boundary $\{x_d=0\}$ is noncharacteristic, that is, $A_d(0)$ is
invertible. The other key assumptions, explained in section \ref{assumptions}, are that $P(v,\partial_x)$ is hyperbolic
with characteristics of constant multiplicity for $v$ in a neighborhood of the origin (Assumption \ref{assumption1})
and that $(P(0,\partial_x),B(0))$ is uniformly stable in the sense of Kreiss (Assumption \ref{a7}).

For any fixed $\eps_0>0$, the standard theory of hyperbolic boundary problems (see e.g., \cite{CP,kreiss}) yields
solutions of \eqref{a2} on a fixed time interval $[0,T_{\eps_0}]$ independent of $\eps\geq \eps_0$. However, since
Sobolev norms of the boundary data blow up as $\eps\to 0$, the standard theory  yields solutions $u_\eps$ of \eqref{a2}
only on time intervals $[0,T_\eps]$ that shrink to zero as $\eps\to 0$. In section \ref{exact}, exact (and necessarily
unique) solutions to \eqref{a2} of the form $u_\eps(x)=U_\eps(x,\frac{x'\cdot\beta}{\eps})$ are constructed on a time
interval independent of $\eps\in (0,\eps_0]$ for $\eps_0$ sufficiently small, where $U_\eps(x,\theta_0)$   satisfies
 the \emph{singular system} derived by substituting $U_\eps(x,\frac{x'\cdot\beta}{\eps})$ into \eqref{a2}:
\begin{align}\label{a3}
\begin{split}
&\;\sum^d_{j=0} A_j(\eps U_\eps) \, \partial_{x_j}U_\eps+\dfrac{1}{\eps} \, \sum^{d-1}_{j=0}
A_j(\eps U_\eps) \, \beta_j \, \partial_{\theta_0}U_\eps=\cF(\eps U_\eps)U_\eps,\\
&\;B(\eps U_\eps)(U_\eps)|_{x_d=0}=G(x',\theta_0),\\
&\;U_\eps=0 \text{ in } t<0.
\end{split}
\end{align}

As explained in \cite{williams3}, the study of singular systems is greatly complicated by the presence of a boundary.
Even if one assumes that the matrices $A_j$ are symmetric (as we do not here), there is no way to obtain an $L^2$
estimate uniform in $\eps$ by a simple integration by parts because of the boundary terms that arise\footnote{The
class of symmetric problems with maximal strictly dissipative boundary conditions provides an exception to this
statement, but that class is too restrictive for some important applications; for example, the boundary problem that
arises in the study of multi-D shocks does not lie in this class.}. The blow-up examples of \cite{williams2} show that,
at least in the wavetrain case, for certain boundary frequencies $\beta$ it is impossible  to estimate solutions of
\eqref{a3} uniformly with respect to $\eps$ in $C(x_d,H^s(x',\theta_0))$ norms, or indeed in \emph{any} norm that
dominates the $L^\infty$ norm\footnote{The problem occurs only for $\beta$ in the glancing set (Definition \ref{def1}),
as the examples of \cite{williams2} together with the results of \cite{williams3} show.}. We do not know if analogous
blow-up examples exist in the pulse case, but it is clear that the proofs of this paper do not apply when $\beta$ lies
in the glancing set (Definition \ref{def1}).


In \cite{CGW2} a class of singular pseudodifferential operators, acting on functions $U(x',\theta_0)$ decaying in
$\theta_0$ and having the form
\begin{align}\label{a4a}
p_s(D_{x',\theta_0})U:=\int_{\R^d \times \R} {\rm e}^{i x'\xi'+i\theta_0 k} \,
p \left( \eps V(x',\theta_0),\xi'+\frac{k\, \beta}{\eps},\gamma \right) \, \widehat{U}(\xi',k) \, {\rm d}\xi' \, {\rm d}k,
\; \; \gamma \geq 1,
\end{align}
was introduced to deal with these difficulties. Observe that after multiplication by $A_d^{-1}(\eps U_\eps)$ and
setting $\tA_j:=A_d^{-1}A_j$, $F:=A_d^{-1}\cF$, \eqref{a3} becomes
\begin{align}\label{a4b}
\begin{split}
&\partial_{x_d}U_\eps +\sum^{d-1}_{j=0} \tA_j(\eps U_\eps)
\left(\partial_{x_j}+\dfrac{\beta_j \partial_{\theta_0}}{\eps}\right) U_\eps \\
&\qquad \equiv \partial_{x_d}U_\eps +\mathbb{A} \left( \eps U_\eps,\partial_{x'}+\frac{\beta \partial_{\theta_0}}{\eps}
\right) U_\eps=F(\eps U_\eps)U_\eps,\\
&B(\eps U_\eps)(U_\eps)|_{x_d=0}=G(x',\theta_0),\\
&U_\eps=0 \text{ in } t<0,
\end{split}
\end{align}
where $\mathbb{A}\left(\eps U_\eps,\partial_{x'}+\frac{\beta \partial_{\theta_0}}{\eps}\right)$ is a (differential) operator
that can be expressed in the form \eqref{a4a}. Kreiss-type symmetrizers $r_s(D_{x',\theta_0})$ in the singular calculus
can be constructed for the system \eqref{a4b} as in \cite{williams3} under the assumptions given below. With these one
can prove $L^2(x_d,H^s(x',\theta_0))$ estimates uniform in $\eps$ for the linearization of \eqref{a4b}. The main difference
with \cite{williams3} is that we use here a singular pseudodifferential calculus that is specially constructed  for pulses in \cite{CGW2}.    In the pulse case $\theta_0$ lies
in an unbounded set and the exact profile $U_\eps(x,\theta_0)$ has continuous Fourier spectrum.   The analysis of \cite{williams3} relied on a singular calculus for wavetrains.  In that case $\theta_0$ lies in
$S^1$ and $U_\eps(x,\theta_0)$ has discrete Fourier spectrum, a fact that was used in several places for proving symbolic calculus rules in \cite{williams3}.  The results of the pulse calculus needed here are  recalled in Appendix \ref{append}.

To progress beyond $L^2(x_d,H^s(x',\theta_0))$ estimates and control $L^\infty$ norms, the boundary frequency $\beta$ must be restricted to lie in the complement
of the glancing set (Definition \ref{def1}). With this extra assumption we are able to use the pulse calculus  to
block-diagonalize the operator $\mathbb{A}\left(\eps U_\eps,\partial_{x'}+\frac{\beta \partial_{\theta_0}}{\eps}\right)$ and
thereby prove estimates uniform with respect to $\eps$ in the spaces
\begin{align}\label{a5}
E^s_T=C(x_d,H^s_T(x',\theta_0))\cap L^2(x_d,H^{s+1}_T(x',\theta_0)).
\end{align}
These spaces are algebras and are contained in $L^\infty$ for $s>\frac{d+1}{2}$. For large enough $s$, as determined
by the requirements of the calculus, existence of solutions to \eqref{a4} in $E^s_T$ on a time interval $[0,T]$ independent
of $\eps\in (0,\eps_0]$ follows by Picard iteration (see Theorem \ref{a12}).

\subsection{Assumptions and main results}\label{assumptions}

\emph{\quad}Before continuing with an overview of the strategies for constructing profiles and for showing that
approximate solutions are close to exact solutions, we now give a precise statement of our assumptions and
main results.

We make the following hyperbolicity assumption on the system \eqref{a2}:

\begin{ass}
\label{assumption1}
The matrix $A_0=I$. For an open neighborhood $\cU$ of  $0\in\R^N$, there exists an integer $q \ge 1$, some
real functions $\lambda_1,\dots,\lambda_q$ that are $C^\infty$ on $\cU \;\times$ $\R^d\setminus\{0\}$ and
homogeneous of degree $1$ and analytic in $\xi$, and there exist some positive integers $\nu_1,\dots,\nu_q$
such that:
\begin{equation*}
\det \Big[ \tau \, I+\sum_{j=1}^d \xi_j \, A_j(u) \Big] =\prod_{k=1}^q \big( \tau+\lambda_k(u,\xi) \big)^{\nu_k}
\end{equation*}
for $u \in \cU$ and $\xi=(\xi_1,\dots,\xi_d)\in\R^d\setminus\{0\}$. Moreover the eigenvalues $\lambda_1(u,\xi),
\dots, \lambda_q(u,\xi)$ are semi-simple (their algebraic multiplicity equals their geometric multiplicity) and
satisfy $\lambda_1(u,\xi)<\dots<\lambda_q(u,\xi)$ for all $u \in \cU$, $\xi \in \R^d \setminus \{ 0\}$.
\end{ass}

We restrict our analysis to noncharacteristic boundaries and therefore make the following:

\begin{ass}
\label{assumption2}
For $u\in \cU$ the matrix $A_d(u)$ is invertible and the matrix $B(u)$ has maximal rank, its rank $p$ being
equal to the number of positive eigenvalues of $A_d(u)$ (counted with their multiplicity).
\end{ass}

In the normal modes analysis for the linearization of \eqref{a2} at $0 \in \cU$, one first performs a Laplace
transform in the time variable $t$ and a Fourier transform in the tangential space variables $y$. We let
$\tau-i\, \gamma \in \C$ and $\eta \in \R^{d-1}$ denote the dual variables of $t$ and $y$. We introduce
the symbol
\begin{equation*}
{\mathcal A}(\zeta):= -i \, A_d^{-1}(0) \left( (\tau-i\gamma) \, I +\sum_{j=1}^{d-1} \eta_j \, A_j(0) \right)
\, ,\quad \zeta:=(\tau-i\gamma,\eta) \in \C \times \R^{d-1} \, .
\end{equation*}
For future use, we also define the following sets of frequencies:
\begin{align*}
& \Xi := \Big\{ (\tau-i\gamma,\eta) \in \C \times \R^{d-1} \setminus (0,0) : \gamma \ge 0 \Big\} \, ,
& \Sigma := \Big\{ \zeta \in \Xi : \tau^2 +\gamma^2 +|\eta|^2 =1 \Big\} \, ,\\
& \Xi_0 := \Big\{ (\tau,\eta) \in \R \times \R^{d-1} \setminus (0,0) \Big\} = \Xi \cap \{ \gamma = 0 \} \, ,
& \Sigma_0 := \Sigma \cap \Xi_0 \, .
\end{align*}
Henceforth we suppress the $u$ in $\lambda_k(u,\xi)$ when it is evaluated at $u=0$ and write
$\lambda_k(0,\xi)=\lambda_k(\xi)$. Two key objects in our analysis are the hyperbolic region and the
glancing set that are defined as follows:

\begin{definition}
\label{def1}
\begin{itemize}
 \item The hyperbolic region ${\mathcal H}$ is the set of all $(\tau,\eta) \in \Xi_0$ such that the matrix
       ${\mathcal A}(\tau,\eta)$ is diagonalizable with purely imaginary eigenvalues.

 \item Let $G$ denote the set of all $(\tau,\xi) \in \R \times \R^d$ such that $\xi \neq 0$ and there exists
       an integer $k \in \{1,\dots,q\}$ satisfying:
\begin{equation*}
\tau + \lambda_k(\xi) = \dfrac{\partial \lambda_k}{\partial \xi_d} (\xi) = 0 \, .
\end{equation*}
If $\pi (G)$ denotes the projection of $G$ on the $d$ first coordinates (in other words $\pi (\tau,\xi) =
(\tau,\xi_1,\dots,\xi_{d-1})$ for all $(\tau,\xi)$), the glancing set ${\mathcal G}$ is ${\mathcal G} :=
\pi (G) \subset \Xi_0$.
\end{itemize}
\end{definition}

\noindent We recall the following result that is due to Kreiss \cite{kreiss} in the strictly hyperbolic case
(when all integers $\nu_j$ in Assumption \ref{assumption1} equal $1$) and to M\'etivier \cite{metivier}
in our more general framework:

\begin{prop}[\cite{kreiss,metivier}]
\label{thm1}
Let Assumptions \ref{assumption1} and \ref{assumption2} be satisfied. Then for all $\zeta \in \Xi \setminus
\Xi_0$, the matrix ${\mathcal A}(\zeta)$ has no purely imaginary eigenvalue and its stable subspace $\E^s (\zeta)$
has dimension $p$.\footnote{The stable subspace is the direct sum of the generalized eigenspaces associated to eigenvalues with negative real part.} Furthermore, $\E^s$ defines an analytic vector bundle over $\Xi \setminus \Xi_0$ that can
be extended as a continuous vector bundle over $\Xi$.
\end{prop}

\noindent For all $(\tau,\eta) \in \Xi_0$, we let $\E^s(\tau,\eta)$ denote the continuous extension of $\E^s$ to the
point $(\tau,\eta)$. The analysis in \cite{metivier} shows that away from the glancing set ${\mathcal G} \subset
\Xi_0$, $\E^s(\zeta)$ depends analytically on $\zeta$, and the hyperbolic region ${\mathcal H}$ does not contain
any glancing point.

Next we define the hyperbolic operator
\begin{align*}
L(\partial_x):=\partial_t+\sum_{j=1}^dA_j(0)\partial_{x_j}
\end{align*}
and recall the definition of uniform stability \cite{kreiss,CP}:

\begin{defn}\label{a6}
The problem \eqref{a2} is said to be \emph{uniformly stable} at $u=0$ if the linearized operators
$(L(\partial_x),B(0))$ at $u=0$ are such that
\begin{align*}
B(0):\E^s(\tau-i\gamma,\eta) \to \C^p \text{ is an isomorphism for all } (\tau-i\gamma,\eta) \in \Sigma.
\end{align*}
\end{defn}

\begin{ass}\label{a7}
The problem \eqref{a2} is \emph{uniformly stable} at $u=0$.
\end{ass}

It is clear that uniform stability at $u=0$ implies uniform stability at nearby states (and therefore at all $u \in \cU$
up to restricting $\cU$). Thus, there is a slight redundancy in Assumptions \ref{assumption2} and \ref{a7} as far
as the rank of $B(0)$ is concerned.
\bigskip

\textbf{Boundary and interior phases.}
We consider a planar real phase $\phi_0$ defined on the boundary:
\begin{equation}
\label{phasebord}
\quad \phi_0(t,y) :=\utau \, t +\ueta \cdot y \, ,\quad \tauetabar \in \Xi_0 \, .
\end{equation}
As follows from earlier works (e.g. \cite{majdaartola}), oscillations on the boundary associated with the phase
$\phi_0$ give rise to oscillations in the interior associated with some planar phases $\phi_m$. These phases
are characteristic for the hyperbolic operator $L(\partial_x)$ and their trace on the boundary equals $\phi_0$.
For now we make the following:

\begin{assumption}
\label{a8}
The phase $\phi_0$ defined by \eqref{phasebord} satisfies $\tauetabar \in {\mathcal H}$.
\end{assumption}

\noindent Thanks to Assumption \ref{a8}, we know that the matrix ${\mathcal A} \tauetabar$ is
diagonalizable with purely imaginary eigenvalues. These eigenvalues are denoted $i\, {\omega}_1,
\dots,i\, {\omega}_M$, where the ${\omega}_m$'s are real and pairwise distinct. The
${\omega}_m$'s are the roots (and all the roots are real) of the dispersion relation:
\begin{equation*}
\det \Big[ \underline{\tau} \, I+\sum_{j=1}^{d-1} \underline{\eta}_j \, A_j(0) +\omega \, A_d(0) \Big] = 0 \, .
\end{equation*}
To each root ${\omega}_m$ there corresponds a unique integer $k_m \in \{ 1,\dots,q\}$ such that
$\underline{\tau} + \lambda_{k_m} (\underline{\eta},{\omega}_m)=0$. We can then define the following
real\footnote{If $\tauetabar$ does not belong to the hyperbolic region ${\mathcal H}$, some of the phases
$\phi_m$ may be complex, see e.g. \cite{williams1,williams2,lescarret,marcou,hernandez}. Moreover, glancing
phases introduce a new scale $\sqrt{\eps}$ as well as boundary layers. } phases and their associated group
velocities:
\begin{equation}
\label{phases}
\forall \, m =1,\dots,M \, ,\quad \phi_m (x):= \phi_0(t,y)+\omega_m \, x_d \, ,\quad
{\bf v}_m := \nabla \lambda_{k_m} (\underline{\eta},\omega_m) \, .
\end{equation}
Let us observe that each group velocity ${\bf v}_m$ is either incoming or outgoing with respect to the space
domain $\R^d_+$: the last coordinate of ${\bf v}_m$ is nonzero. This property holds because $\tauetabar$
does not belong to the glancing set ${\mathcal G}$. We can therefore adopt the following classification:

\begin{definition}
\label{def2}
The phase $\phi_m$ is said to be incoming if the group velocity ${\bf v}_m$ is incoming (that is, $\partial_{\xi_d}
\lambda_{k_m} (\underline{\beta},{\omega}_m)>0$), and outgoing if the group velocity ${\bf v}_m$ is outgoing
($\partial_{\xi_d} \lambda_{k_m} (\underline{\beta},{\omega}_m) <0$).
\end{definition}

\noindent In all that follows, we let ${\mathcal I}$ denote the set of indices $m \in \{ 1,\dots,M\}$ such that
$\phi_m$ is an incoming phase, and ${\mathcal{O}}$ denote the set of indices $m \in \{ 1,\dots,M\}$ such
that $\phi_m$ is an outgoing phase.   If $p\geq 1$, then $\cI$ is nonempty, while if $p\leq N-1$, $\cO$ is
nonempty (this follows from Lemma \ref{lem1} below).
\bigskip

\textbf{Main results. }We  will use the notation:
\begin{align*}
\begin{split}
&L(\tau,\xi) := \tau \, I +\sum_{j=1}^d \xi_j \, A_j(0) \, ,\\
&\beta=(\utau,\ueta), \;x'=(t,y),\; \phi_0(x')=\beta\cdot x'.
\end{split}
\end{align*}
For each phase $\phi_m$, ${\rm d}\phi_m$ denotes the differential of the function $\phi_m$ with respect to its
argument $x=(t,y,x_d)$. It follows from Assumption \ref{assumption1} that the eigenspace of ${\mathcal A} (\beta)$
associated with the eigenvalue $i\, \underline{\omega}_m$ coincides with the kernel of $L({\rm d}\phi_m)$ and has
dimension $\nu_{k_m}$. The following well-known lemma, whose proof is recalled in \cite{jfcog}, gives a useful
decomposition of $\E^s$ in the hyperbolic region.

\begin{lem}
\label{lem1}
The stable subspace $\E^s \tauetabar$ admits the decomposition:
\begin{equation}
\label{decomposition1}
\E^s \tauetabar = \oplus_{m \in {\mathcal I}} \, \text{\rm Ker } L({\rm d}\phi_m) \, ,
\end{equation}
and each vector space in the decomposition \eqref{decomposition1} admits a basis of real vectors.
\end{lem}

The next Lemma, also proved in \cite{jfcog}, gives a useful decomposition of $\C^N$ and introduces projectors
needed later for formulating and solving the profile equations.

\begin{lem}
\label{lem2}
The space $\C^N$ admits the decomposition:
\begin{equation}
\label{decomposition2}
\C^N = \oplus_{m=1}^M \, \text{\rm Ker } L({\rm d} \phi_m)
\end{equation}
and each vector space in \eqref{decomposition2} admits a basis of real vectors. If we let $P_1,\dots,P_M$
denote the projectors associated with the decomposition \eqref{decomposition2}, then there holds
$\text{\rm Im } A_d^{-1}(0) \, L({\rm d} \phi_m) = \text{\rm Ker } P_m$ for all $m=1,\dots,M$.
\end{lem}

For each $m\in\{1,\dots,M\}$ we let
\begin{align*}
r_{m,k}, \;k=1,\dots,\nu_{k_m}
\end{align*}
denote a basis of $\ker L({\rm d}\phi_m)$ consisting of real vectors. In section \ref{sect3}, we construct an
approximate solution $u^a_\eps$ of \eqref{a2} of the form
\begin{align}\label{a10}
u^a_\eps(x) =\sum_{m\in\cI} \sum^{\nu_{k_m}}_{k=1} \sigma_{m,k} \left(x,\frac{\phi_m}{\eps}\right) \, r_{m,k},
\end{align}
where the $\sigma_{m,k}(x,\theta_m)$ are $C^1$ functions decaying to zero as $|\theta_m|\to \infty$, which
describe the propagation of pulses with group velocity ${\bf v}_m$ (see Proposition \ref{c37}). Observe that
if one plugs the expression \eqref{a10} of $u^a_\eps$ into $P(\eps u_\eps,\partial_x)u_\eps$, the terms of
order $1/\eps$ vanish, leaving an $O(1)$ error, \emph{regardless} of how the $\sigma_{m,k}$ are chosen.
The interior profile equations satisfied by these functions are solvability conditions that permit this $O(1)$ error
to be (at least partially) removed by a corrector that is sublinear (in fact bounded) as $|\theta_m|\to\infty$.
Additional conditions on the profiles come, of course, from the boundary conditions.

For use in the remainder of the introduction and later, we collect some notation here.

\begin{nota}\label{a11}
(a)\; Let $\Omega:=\overline{\R}^{d+1}_+\times\R^1$, $\Omega_T:=\Omega\cap\{-\infty<t<T\}$,
$b\Omega:=\R^d\times\R^1$, $b\Omega_T:=b\Omega\cap \{-\infty<t<T\}$,  and set $\omega_T
:=\overline{\R}^{d+1}_+\cap\{-\infty<t<T\}$.

(b)\; For $s\geq 0$ let $H^s\equiv H^s(b\Omega)$, the standard Sobolev space with norm $\langle
V(x',\theta_0)\rangle_s$.

(c)\; $L^2H^s\equiv L^2(x_d,H^s(b\Omega))$ with $|U(x,\theta_0)|_{L^2H^s}\equiv|U|_{0,s}$.

(d)\; $CH^s\equiv C(x_d,H^s(b\Omega))$ with $|U(x,\theta_0)|_{CH^s}\equiv\sup_{x_d\geq 0}|U(.,x_d,.)|_{H^s}
\equiv |U|_{\infty,s}$ (note that $CH^s\subset L^\infty H^s$).

(e)\; $C^{0,M}(\mathbb{R}^{d+1}_+ \times \mathbb{R}) \equiv \{ V(x',x_d,\theta_0) \in C \left(
\mathbb{R}_+,C^M_b(\mathbb{R}^d \times \mathbb{R},\mathbb{R}^N)\right) \}$ where $C^M_b$
denotes the space of $M$ times differentiable functions with derivatives up to the order $M$ bounded.

(f)\; Similarly, $H^s_T\equiv H^s(b\Omega_T)$ with norm $\langle V\rangle_{s,T}$ and $L^2H^s_T \equiv
L^2(x_d,H^s_T)$, $CH^s_T\equiv C(x_d,H^s_T)$ have norms $|U|_{0,s,T}$, $|U|_{\infty,s,T}$ respectively.

(g)When the domains of $x_d$ and $(x',\theta_0)$ are clear, we sometimes use the self-explanatory notation
$C(x_d,H^s(x',\theta_0))$ or $L^2(x_d,H^s(x',\theta_0))$.

(h)\; For $r\geq 0$, $[r]$ is the smallest integer $\geq r$.

(i)\; $M_0:=3\, d+5$,
%
\end{nota}

The main result of section \ref{exact} is the following theorem, which gives the existence of exact solutions to
the singular system \eqref{a3}, or equivalently \eqref{a4b}, and the original system \eqref{a2} on a time interval
independent of the wavelength $\eps$:

\begin{theo}\label{a12}
Under Assumptions \ref{assumption1}, \ref{assumption2}, \ref{a7}, \ref{a8}, consider the quasilinear boundary
problem \eqref{a2}, where $G(x',\theta_0)\in H^{s+1}(b\Omega)$, $s\geq [M_0+\frac{d+1}{2}]$, satisfies
\begin{equation*}
\text{\rm Supp } G \subset \{ t \geq 0\} \, .
\end{equation*}
There exist $\eps_0>0$, $T_0>0$ independent of $\eps \in (0,\eps_0]$, and a unique $U_\eps(x,\theta_0) \in
CH^s_{T_0}\cap L^2H^{s+1}_{T_0}$ satisfying the singular problem \eqref{a4b}, so that
\begin{equation*}
u_\eps(x) := U_\eps \left( x,\frac{x'\cdot\beta}{\eps} \right) \, ,
\end{equation*}
is the unique $C^1$ solution of \eqref{a2} on $\omega_{T_0}$.
\end{theo}

\begin{rem}\label{a15}
\textup{The regularity requirement $s\geq [M_0+\frac{d+1}{2}]$ in the above theorem is needed
in order to apply the singular pseudodifferential calculus introduced in \cite{CGW2}.}
\end{rem}

We can now state the  main result of this paper. This theorem is actually a corollary of the result for singular
systems given in Theorem \ref{e1}.

\begin{theo}\label{a16}
Under the same assumptions as in Theorem \ref{a12}, there exists $T_0>0$ and functions 
$\sigma_{m,k}(x,\theta_m)\in C^1(\Omega_{T_0})$  satisfying the leading order profile equations \eqref{c20}
and defining an approximate solution $u^a_\eps$ as in \eqref{a10} such that
\begin{align*}
\lim_{\eps\to 0} \;u_\eps-u^a_\eps =0\text{ in } L^\infty(\omega_{T_0}),
\end{align*}
where $u_\eps\in C^1(\omega_{T_0})$ is the unique exact solution of \eqref{a2}. In fact we obtain the rate of convergence
\begin{align*}
|u_\eps-u^a_\eps|_{L^\infty(\omega_{T_0})}\leq C\, \eps^{\frac{1}{2M_1+5}},\text{ where }M_1:=\left[\frac{d}{2}+3\right].
\end{align*}
\end{theo}

Theorem \ref{a16} can be recast in a form where the pulses originate in initial data at $t=0$ and reflect off the
boundary $\{x_d=0\}$. This requires a discussion similar to that given in section 3.2 of \cite{CGW1} to justify the
reduction of the initial boundary value problem with data prescribed at $t=0$ to a forward boundary problem (with
data identically zero in $t<0$), so we omit that discussion here.

\subsection{Profile equations}

In $d$ space variables $(x'',x_d)$, consider the quasilinear problem equivalent to \eqref{a2}
\begin{align*}
\begin{split}
&\partial_d u_\eps+\sum_{j=0}^{d-1}\tA_j(\eps u_\eps)\partial_ju_\eps=F(\eps u_\eps)u_\eps\text{ in }x_d\geq 0\\
&B(\eps u_\eps)u_\eps= G(x',\theta_0)|_{\theta_0=\frac{\phi_0}{\eps}}\text{ on }x_d=0\\
&u_\eps=0\text{ in }t<0,
\end{split}
\end{align*}
where $G(x',\theta_0)$ decays to $0$ like $\langle\theta_0\rangle^{-k}$ (for some $k\geq 2$ to be specified later)
as $|\theta_0|\to\infty$. For ease of exposition we will begin by considering the $3\times 3$ case, which contains
all the main difficulties. In section \ref{extension}, we describe the changes needed to treat the general case.
We define the boundary phase $\phi_0 :=\beta\cdot x'$, the real eigenvalues $\omega_m$ of $\cA(\beta)$, and
the phases $\phi_m:=\phi_0+\omega_m \, x_d$ as in \eqref{phases}, where $\beta\in \cH$. We assume that the
eigenvalues $\omega_m$ are pairwise distinct. For the sake of clarity, we also assume that $\omega_1$ and
$\omega_3$ are incoming (or causal) and $\omega_2$ is outgoing. (The same kind of arguments would apply
if two of the phases were outgoing and only one was incoming.) The corresponding right and left eigenvectors
of the real matrix $-i \, \cA (\beta)$ are denoted $r_j$ and $l_j$, $j=1,2,3$.

Below we frequently suppress $\eps$-dependence in the notation. For functions $\cU(x,\theta_0,\xi_d)$ and
$\cV(x,\theta_0,\xi_d)$, define
\begin{align*}
\begin{split}
&\tilde\cL(\partial_{\theta_0},\partial_{\xi_d}):=\partial_{\xi_d} +\sum^{d-1}_{j=0} \beta_j \, \tA_j(0) \, \partial_{\theta_0}
=\partial_{\xi_d} +\tA(\beta) \, \partial_{\theta_0} \text{ and } \tilde L(\partial):=\partial_d
+\sum^{d-1}_{j=0} \tA_j(0) \, \partial_j \, ,\\
&M(\cU,\partial_{\theta_0} \cV) :=\sum^{d-1}_{j=0}\beta_j \, ({\rm d} \tA_j(0) \cdot \cU) \, \partial_{\theta_0}\cV \, .
\end{split}
\end{align*}
Formally looking for a corrected approximate solution of the form
\begin{equation*}
u^c_\eps(x)=\big[ \cU^0(x,\theta_0,\xi_d)+\eps \, \cU^1(x,\theta_0,\xi_d) \big]
|_{\theta_0=\frac{\phi_0}{\eps},\, \xi_d=\frac{x_d}{\eps}} \, ,
\end{equation*}
we obtain interior profile equations
\begin{align}\label{13}
\begin{split}
&\eps^{-1}:\quad \tilde\cL(\partial_{\theta_0},\partial_{\xi_d}) \, \cU^0=0 \, ,\\
&\eps^0:\;\;\quad \tilde \cL(\partial_{\theta_0},\partial_{\xi_d}) \, \cU^1+\tilde L(\partial) \, \cU^0
+M(\cU^0,\partial_{\theta_0}\cU^0)=F(0) \, \cU^0 \, ,
\end{split}
\end{align}
and the boundary equation
\begin{equation}\label{14}
\eps^0:B(0) \, \cU^0|_{x_d=0,\xi_d=0}=G(x',\theta_0) \, .
\end{equation}

Consider the first equation in \eqref{13}. A function $\cU^0(x,\theta_0,\xi_d)$, taking values in $\R^3$ and assumed
to be $C^1$ for the moment, can always be written
\begin{align*}
\cU^0=\tilde\sigma_1(x,\theta_0,\xi_d) \, r_1+\tilde\sigma_2(x,\theta_0,\xi_d) \, r_2
+\tilde\sigma_3(x,\theta_0,\xi_d) \, r_3 \, .
\end{align*}
Using the matrix $[r_1\;r_2\;r_3]$ to diagonalize $\tA (\beta)$, we find that the scalar $\tilde\sigma_i$ must satisfy
\begin{align*}
(\partial_{\xi_d}-\omega_i \, \partial_{\theta_0}) \, \tilde\sigma_i =0,\;i=1,2,3, \text{ in }
\{ (x,\theta_0,\xi_d):\theta_0\in\R, \xi_d\geq 0 \}.
\end{align*}
This implies that the $\tilde\sigma_i$'s have the form
\begin{align*}
\tilde\sigma_i(x,\theta_0,\xi_d)=\sigma_i(x,\theta_0+\omega_i\xi_d)\text{ for some }\sigma_i(x,\theta_i).
\end{align*}
Using \eqref{14}, we find
\begin{align}\label{17}
B(0) \, \left( \sum_{i=1,3}\sigma_i(x',0,\theta_0) \, r_i \right)=G(x',\theta_0)-B(0)(\sigma_2(x',0,\theta_0) \, r_2),
\end{align}

\begin{rem}\label{18}
\textup{1. We expect the $\sigma_i(x,\theta_i)$ to decay polynomially to $0$ as $|\theta_i|\to\infty$. To prove this
we must formulate and solve profile equations for the $\sigma_i$'s. For this we use an approach inspired by the
formal constructions in \cite{hmr} and \cite{majdarosales}.}

\textup{2. Instead of $\sigma_i(x,\theta_i)$ we shall sometimes write $\sigma_i(x,\theta)$ with the understanding
that $\theta$ is a placeholder for $\theta_0+\omega_i \, \xi_d$ when it appears as an argument of $\sigma_i$.}
\end{rem}

To get transport equations for the $\sigma_i$'s, we consider \eqref{13} ($\eps^0$):
\begin{align}\label{19}
\tilde \cL(\partial_{\theta_0},\partial_{\xi_d}) \, \cU^1=-\left(\tilde L(\partial) \, \cU^0+M(\cU^0,\partial_{\theta_0}\cU^0)
\right) +F(0)\cU^0:=\cF(x,\theta_0,\xi_d).
\end{align}
The corrector $\cU^1$ can be written as
\begin{align*}
\cU^1=t_1(x,\theta_0,\xi_d)r_1+t_2(x,\theta_0,\xi_d)r_2+t_3(x,\theta_0,\xi_d)r_3.
\end{align*}
Diagonalizing again we find that the $t_i$'s must satisfy
\begin{align}\label{21}
(\partial_{\xi_d}-\omega_i\partial_{\theta_0})t_i(x,\theta_0,\xi_d)=l_i\cdot \cF:=\cF_i(x,\theta_0,\xi_d),\;i=1,2,3.
\end{align}
The general solution to \eqref{21} is
\begin{align}\label{22}
t_i(x,\theta_0,\xi_d)=\tau^*_i(x,\theta_0+\omega_i\xi_d)+\int^{\xi_d}_0 \cF_i(x,\theta_0+\omega_i(\xi_d-s),s) \, {\rm d}s,\;
\end{align}
where $\tau^*_i$ is arbitrary. This can be rewritten
\begin{align}\label{23}
\begin{split}
&t_i(x,\theta_0,\xi_d) \\
&\qquad =\tau^*_i(x,\theta_0+\omega_i\xi_d) +\int^\infty_0 \cF_i(x,\theta_0+\omega_i(\xi_d-s),s) \, {\rm d}s
+\int^{\xi_d}_\infty \cF_i(x,\theta_0+\omega_j(\xi_d-s),s) \, {\rm d}s \\
&\qquad =\tau_i(x,\theta_0+\omega_i\xi_d) +\int^{\xi_d}_\infty \cF_i(x,\theta_0+\omega_i(\xi_d-s),s) \, {\rm d}s \, ,
\end{split}
\end{align}
provided the integrals in \eqref{23} exist.

We will need the following modification of a classical lemma due to Lax \cite{lax}. We refer to \cite[Lemma 2.11]{CGW1}
for the proof.

\begin{prop}\label{23a}
Let $W(x,\theta_0,\xi_d)=\sum^3_{i=1}w_i(x,\theta_0,\xi_d)r_i$ be any $C^1$ function.  Then
\begin{equation*}
\tilde L(\partial)W=\sum^3_{i=1}(X_{\phi_i}w_i) \, r_i+ \sum^3_{i=1} \big( \sum_{k\neq i} V^i_k w_k \big) \, r_i \, ,
\end{equation*}
where $X_{\phi_i}$ is the characteristic vector field\footnote{This vector field is a scalar multiple of $\partial_t +\nabla
\lambda_{k_i}(\ueta,\omega_i) \cdot \nabla_{x''}$ that describes propagation at the group velocity ${\bf v}_i$; see
\eqref{phases}.}
\begin{equation*}
X_{\phi_i} :=\partial_{x_d}+\sum^{d-1}_{j=0}-\partial_{\xi_j}\omega_i(\beta)\partial_{x_j} \, ,
\end{equation*}
and $V^i_k$ for $k\neq i$ is the tangential vector field
\begin{equation*}
V^i_k :=\sum^{d-1}_{l=0}(l_i \, \tilde A_l(0) \, r_k) \, \partial_{x_l}.
\end{equation*}
\end{prop}


We see from \eqref{19} that $\cF_i(x,\theta_0,\xi_d)$ has the form
\begin{align}\label{25}
\cF_i(x,\theta_0,\xi_d)=-X_{\phi_i}\tilde{\sigma}_i -\sum_k c^i_k \, \tilde{\sigma}_k \, \partial_{\theta_0} \tilde{\sigma}_k
-\sum_{l\neq m} d^i_{l,m} \, \tilde{\sigma}_l \, \partial_{\theta_0} \tilde{\sigma}_m
+\sum_k e^i_k \, \tilde{\sigma}_k -\sum_{k\neq i} V^i_k \, \tilde{\sigma}_k \, ,
\end{align}
where we recall $\tilde{\sigma}_p (x,\theta_0,\xi_d)=\sigma_p(x,\theta_0+\omega_p\theta_d)$. The coefficients
in \eqref{25} are defined by\footnote{We refer to section \ref{extension} for the general case.}
\begin{equation*}
c_k^i :=l_i \, \sum^{d-1}_{j=0}\beta_j \, ({\rm d} \tA_j(0) \cdot r_k) \, r_k \, ,\quad
d_{l,m}^i := l_i \, \sum^{d-1}_{j=0}\beta_j \, ({\rm d} \tA_j(0) \cdot r_l) \, r_m \, ,\quad
e_k^i := l_i \, F(0) \, r_k \, .
\end{equation*}
Thus, we compute
\begin{align}\label{26}
\begin{split}
&\cF_i(x,\theta_0+\omega_i(\xi_d-s),s) \\
&=-(X_{\phi_i} \sigma_i +c^i_i \, \sigma_i \, \partial_\theta \sigma_i -e^i_i \, \sigma_i) (x,\theta_0+\omega_i\xi_d) \\
&\quad -\sum_{k\neq i} c^i_k \, \sigma_k(x,\theta_0+\omega_i\xi_d+s(\omega_k-\omega_i)) \,
\partial_\theta \sigma_k (x,\theta_0+\omega_i\xi_d+s(\omega_k-\omega_i)) \\
&\quad -\sum_{m\neq i}d^i_{i,m} \, \sigma_i(x,\theta_0+\omega_i\xi_d) \, \partial_\theta \sigma_m
(x,\theta_0+\omega_i\xi_d+s(\omega_m-\omega_i)) \\
&\quad -\sum_{l\neq i} d^i_{l,i} \, \sigma_l(x,\theta_0+\omega_i\xi_d+s(\omega_l-\omega_i)) \,
\partial_\theta \sigma_i (x,\theta_0+\omega_i\xi_d) \\
&\quad -\sum_{l\neq m,l\neq i,m\neq i} d^i_{l,m} \, \sigma_l(x,\theta_0+\omega_i\xi_d+s(\omega_l-\omega_i)) \,
\partial_\theta \sigma_m (x,\theta_0+\omega_i\xi_d+s(\omega_m-\omega_i)) \\
&\quad +\sum_{k\neq i} (e^i_k -V_k^i) \, \sigma_k (x,\theta_0+\omega_i\xi_d+s(\omega_k-\omega_i)) \, .
\end{split}
\end{align}

We look for functions $\sigma_i(x,\theta)$ that decay at least at the rate $\langle \theta\rangle^{-2}$. So we
assume now and verify later that they have this property. Then the integral
\begin{align}\label{27}
\int^{\xi_d}_0 \cF_i(x,\theta_0+\omega_i(\xi_d-s),s) \, {\rm d}s
\end{align}
is sublinear in $(\theta_0,\xi_d)$ (a condition that must be satisfied by $\cU^1$ if $\eps \, \cU^1$ is to make
sense as a corrector) if and only if the sum of the first three terms on the right in \eqref{26} is $0$. In that
case the integral \eqref{27} is actually bounded, since the remaining terms in \eqref{26} have good decay
in $s$. This sublinearity condition gives the profile equations for the $\sigma_i$'s:
\begin{align}\label{28}
\begin{split}
&X_{\phi_i} \sigma_i +c^i_i \, \sigma_i \, \partial_{\theta_i}\sigma_i -e^i_i \, \sigma_i=0,\;i=1,2,3\\
&(\sigma_i(x',0,\theta_0), i =1,3)=\cB \, \left( G(x',\theta_0),\sigma_2(x',0,\theta_0) \right),\\
&\sigma_i=0\text{ in }t<0.
\end{split}
\end{align}
where $\cB$ is a well-determined linear function of its arguments whose existence is given by Lemma \ref{lem1}
and the uniform stability assumption (the matrix $[B(0) \, r_1 \, \, B(0) \, r_3]$ in \eqref{17} is invertible). As expected
from the general rule of thumb, pulses of different families do not interact at the leading order, meaning that the
evolution equations for the amplitudes $\sigma_i$'s are decoupled. In Proposition \ref{c37}, we show that system
\eqref{28} is uniquely solvable on some time interval $[0,T_1]$, that $\sigma_2=0$ and that $\sigma_i$, $i=1,3$,
decay at the rate $\langle\theta\rangle^{-k}$ for some $k\geq 2$ to be determined.

\begin{rem}\label{29}
\textup{The equations \eqref{28} and our assumption that the $\sigma_i$ decay at least at the rate
$\langle \theta \rangle^{-2}$ imply that the integrals in \eqref{23} all do exist. This argument is made
more precise below.}
\end{rem}

Next we introduce an averaging operator ${\bf E}$ and a solution operator ${\bf R}_\infty$ that will be useful in the
error analysis of the next paragraph. Motivated by the form of $\cF_i$ in \eqref{25}, we make the following definition.

\begin{defn}[Type $\cF$ functions]\label{30}
Suppose
\begin{align}\label{30aa}
F(x,\theta_0,\xi_d)=\sum_{i=1}^3 F_i(x,\theta_0,\xi_d) \, r_i,
\end{align}
where each $F_i$ has the form
\begin{align}\label{30a}
\begin{split}
&F_i(x,\theta_0,\xi_d)=\sum_{k=1}^3 f^i_k(x,\theta_0+\omega_k \, \xi_d)
+\sum_{l \le m=1}^3 g^i_{l,m}(x,\theta_0+\omega_l \, \xi_d) \, h^i_{l,m}(x,\theta_0+\omega_m \, \xi_d),\\
\end{split}
\end{align}
where the functions  $f^i_k(x,\theta)$, $g^i_{l,m}(x,\theta)$, $h^i_{l,m}(x,\theta)$ are real-valued, $C^1$, and decay
along with their first order partials at the rate $O(\langle\theta\rangle^{-2})$ uniformly with respect to $x$. We then
say that $F$ is of \emph{type $\cF$}. For such $F$ define
\begin{align*}
{\bf E} F(x,\theta_0,\xi_d):=\sum_{j=1}^3 \left( \lim_{T\to\infty} \, \dfrac{1}{T} \, \int^T_0
l_j\cdot F(x,\theta_0+\omega_j \, (\xi_d-s),s) \, {\rm d}s \right) \, r_j \, .
\end{align*}
\end{defn}

\begin{rem}\label{30ab}
\textup{1.) For $F$ as in \eqref{30aa}-\eqref{30a}, we have
\begin{align}\label{31a}
{\bf E}F =\sum^3_{i=1} \tilde F_i(x,\theta_0 +\omega_i \, \xi_d) \, r_i,\text{ where } \tilde F_i (x,\theta)
:=f^i_i(x,\theta) +g^i_{i,i}(x,\theta) \, h^i_{i,i}(x,\theta) \, .
\end{align}}

\textup{2.) Observe that $\cF$ as defined in \eqref{19} is of type $\cF$ (hence the terminology), provided the
$\sigma_i$'s have sufficiently regularity and decay in $\theta$. In that case, we obtain
\begin{align*}
{\bf E}\cF(x,\theta_0,\xi_d)=-\sum_{i=1}^3 \big( X_{\phi_i} \sigma_i +c^i_i \, \sigma_i \, \partial_{\theta_i}\sigma_i
-e^i_i \, \sigma_i \big) \, r_i \, ,\quad \text{ where }\sigma_i=\sigma_i (x,\theta_0+\omega_i \, \xi_d).
\end{align*}}
\end{rem}

\begin{rem}\label{32}
\textup{The definition of ${\bf E}$ can be extended to more general functions. For example, if
\begin{align*}
F =\sum^3_{i=1} F_i(x,\theta_0+\omega_i \, \xi_d) \, r_i,
\end{align*}
where the $F_i(x,\theta)$ are arbitrary continuous functions, the limits that define ${\bf E}F$ exist and we
have ${\bf E}F=F$. For another example, suppose $F$ is of type $\cF$ and satisfies $EF=0$. Define
\begin{align}\label{R}
{\bf R}_\infty F(x,\theta_0,\xi_d) :=\sum_{i=1}^3 \left( \int^{\xi_d}_\infty F_i(x,\theta_0+\omega_i(\xi_d-s),s) \,
{\rm d}s \right) \, r_i \, .
\end{align}
Then the limits defining ${\bf R}_\infty F$ and ${\bf E} {\bf R}_\infty F$ exist and we have ${\bf E} {\bf R}_\infty F=0$.}
\end{rem}

\begin{prop}\label{35}
Suppose $F$ is of type $\cF$ and satisfies ${\bf E}F=0$. Then ${\bf R}_\infty F$ is bounded and
\begin{align*}
\tilde \cL(\partial_{\theta_0},\partial_{\xi_d}) \, {\bf R}_\infty F ={\bf R}_\infty \, \tilde \cL(\partial_{\theta_0},\partial_{\xi_d}) F
= F =(I-{\bf E}) F.
\end{align*}
\end{prop}

\begin{proof}
It just remains to show ${\bf R}_\infty \, \tilde \cL(\partial_{\theta_0},\partial_{\xi_d})F=F$. This follows by direct
computation of the integrals defining ${\bf R}_\infty \, \cL(\partial_{\theta_0},\partial_{\xi_d})F$ and the fact that
when ${\bf E}F=0$, we have $F_i(x,\theta_0+\omega_i(\xi_d-\infty),\infty)=0$.
\end{proof}

\noindent The next Proposition summarizes what we have shown.

\begin{prop}\label{31}
Let $F(x,\theta_0,\xi_d)$ be a function of type $\cF$.

(a)\; Then the equation $\tilde \cL(\partial_{\theta_0},\partial_{\xi_d})\cU=F$ has a solution bounded in $(\theta_0,\xi_d)$
if and only if ${\bf E}F=0$.

(b)\; When ${\bf E}F=0$, every $C^1$ solution bounded in $(\theta_0,\xi_d)$  has the form
\begin{align*}
\cU=\sum_{i=1}^3 \tau_i (x,\theta_0+\omega_i \, \xi_d) \, r_i +{\bf R}_\infty F \text{ with } \tau_i(x,\theta) \in C^1
\text{ and bounded.}
\end{align*}
Here ${\bf E}\cU=\sum_{i=1}^3 \tau_i(x,\theta_0+\omega_i \, \xi_d) \, r_i$ and $(I-{\bf E})\cU={\bf R}_\infty F$.

(c)If $\cU$ is of type $\cF$ then
\begin{align*}
{\bf E} \, \tilde \cL(\partial_{\theta_0},\partial_{\xi_d}) \cU =\tilde \cL(\partial_{\theta_0},\partial_{\xi_d}) \, {\bf E}\cU=0.
\end{align*}
\end{prop}

\begin{proof}
Part (a) follows from the form of the general solution given in \eqref{22}, and the fact that when a function $F$ of type
$\cF$ satisfies ${\bf E}F=0$, the integrals
\begin{align*}
\int^{\infty}_0 F_i(x,\theta_0+\omega_i \, (\xi_d-s),s) \, {\rm d}s
\end{align*}
are absolutely convergent. Part (b) follows from Remark \ref{30ab} and ${\bf E} {\bf R}_\infty F=0$. Part (c) follows
directly from Remark \ref{30ab}.
\end{proof}

With the leading pulse profile
\begin{equation*}
\cU^0(x,\theta_0,\xi_d) =\sigma_1(x,\theta_0+\omega_1 \, \xi_d) \, r_1 +\sigma_2(x,\theta_0+\omega_2 \, \xi_d) \, r_2
+\sigma_3(x,\theta_0+\omega_3 \, \xi_d) \, r_3 \, ,
\end{equation*}
we can rewrite the profile system \eqref{28} in a form that will be useful for the error analysis as follows:
\begin{align}\label{36a}
\begin{split}
&a)\; {\bf E} \, \cU^0 =\cU^0 \, ,\\
&b)\; {\bf E} \left( \tilde{L}(\partial) \, \cU^0 +M(\cU^0,\partial_{\theta_0} \cU^0) -F(0) \, \cU^0 \right) =0\, ,\\
&c)\; B(0) \, \cU^0|_{x_d=0,\xi_d=0}=G(x',\theta_0) \, ,\\
&d)\; \cU^0 =0 \text{ in }t<0 \, .
\end{split}
\end{align}
These equations can also be obtained by applying the operator ${\bf E}$ to the equations \eqref{13}, and have
the common structure of weakly nonlinear geometric optics equations, see e.g. \cite[chapters 7 and 9]{rauch}.

\subsection{Error analysis}\label{errorintro}

We end this introduction with a sketch of the error analysis used to prove Theorem \ref{e1}, which yields Theorem
\ref{a16} as an immediate consequence. The iteration schemes for the singular system \eqref{a4} and the profile
equations \eqref{36a} are written side by side in \eqref{e5}, \eqref{e6}. For $s$ large\footnote{We take $s>1
+[M_0+\frac{d+1}{2}]$ in Theorem \ref{e1}.} and some $T_0>0$, the proof of Theorem \ref{a12} produces a
sequence of iterates $U^n_\eps(x,\theta_0)$, bounded in the space $E^s_{T_0}$ uniformly with respect to $n$
and $\eps$, and such that
\begin{align*}
\lim_{n\to\infty}U^n_\eps =U_\eps \text{ in } E^{s-1}_{T_0} \text{ uniformly  with respect to  }\eps \in (0,\eps_0] \, ,
\end{align*}
where $U_\eps$ is the solution of the singular system \eqref{a4b}. On the other hand the construction of profiles
in Proposition \ref{c37} yields a sequence of profile iterates $\cU^{0,n}(x,\theta_0,\xi_d)$ bounded in $\cE^s_{T_0}$
(see Definition \ref{d22}) and converging in $\cE^{s-1}_{T_0}$ to a solution $\cU^0$ of the leading profile equations
\eqref{36a}. By Proposition \ref{h1} this implies that the rapidly varying functions $\cU^{0,n}_\eps(x,\theta_0) :=
\cU^{0,n}(x,\theta_0,\frac{x_d}{\eps})$ satisfy
\begin{align*}
\lim_{n\to\infty}\cU^{0,n}_{\eps} =\cU^0_{\eps}\text{ in } E^{s-1}_{T_0} \text{  uniformly  with respect to } \eps \in
(0,\eps_0] \, .
\end{align*}
Thus, in order to conclude $|\cU^0_\eps(x,\theta_0)-U_\eps(x,\theta_0)|_{E^{s-3}_{T_0}} \leq C \, \eps^{\frac{1}{2M_1+5}}$
and thereby complete the proof of Theorem \ref{a16}, it would suffice to show:
\begin{align}\label{dd3}
\text{There exists $C$ such that for every $n$, } \, |\cU^{0,n}_{\eps}-U^n_\eps|_{E^{s-3}_{T_0}} \leq C \,
\eps^{\frac{1}{2M_1+5}} \, .
\end{align}

The statement \eqref{dd3} is proved by induction in section \ref{simult}. It is natural to try to apply the estimate of
Proposition \ref{e9} to the difference $\cU^{0,n+1}_\eps-U^{n+1}_\eps$, but the problem is that for any given $n$,
$\cU^{0,n+1}_\eps$ does not by itself provide a very good approximate solution to the boundary problem \eqref{e5}
that defines $U^{n+1}_\eps$. Indeed, substitution of $\cU^{0,n+1}_\eps$ into \eqref{e5}(a) yields an error, call it
$R^{n+1}_\eps(x,\theta_0)$, that is $O(1)$ in $E^{s-3}_{T_0}$. Since $\cU^{0,n+1}$ satisfies \eqref{e6}, the main
contribution to $R^{n+1}_\eps(x,\theta_0)$ is given by $\cR^{n+1}(x,\theta_0,\frac{x_d}{\eps})$ where
\begin{align}\label{e21z}
\cR^{n+1}:=(I-{\bf E}) \, \left( \tilde L(\partial_x) \, \cU^{0,n+1} +M(\cU^{0,n},\partial_{\theta_0}\cU^{0,n+1})
-F(0) \, \cU^{0,n} \right) \, .
\end{align}
One would like to solve away the main error term in \eqref{e21z} by using Proposition \ref{35} and constructing a
corrector $\cU^{1,n+1} (x,\theta_0,\xi_d)$ such that
\begin{align}\label{dd5}
\tilde{\cL}(\partial_{\theta_0},\partial_{\xi_d}) \, \cU^{1,n+1}=-\cR^{n+1} \, ,
\end{align}
and then use a corrected approximate solution of \eqref{e5}(a) of the form
\begin{align*}
\cU^{0,n+1}_\eps +\eps \, \cU^{1,n+1}_{\eps} \, .
\end{align*}
While such a corrector is given explicitly by $\cU^{1,n+1}={\bf R}_\infty (-\cR^{n+1})$, it is not suitable for the error
analysis because, although bounded, $\cU^{1,n+1}_{\eps}$ does not lie in any of the $E^s_{T_0}$ spaces.

To see the reason for this, note that $\cU^{0,n}$ has the form
\begin{align}\label{ddd5}
\cU^{0,n}(x,\theta_0,\xi_d)=\sum^3_{i=1}\sigma^n_i(x,\theta_0+\omega_i \, \xi_d) \, r_i.
\end{align}
 Observe that if $f$ is a function that decays (say like $|s|^{-2}$) as $|s|\to\infty$,  the primitive $\int_\infty^\theta f(s) \, {\rm d}s$
itself decays to zero as $|\theta|\to\infty$ if and only if $f$ has moment zero ($\int^\infty_{-\infty}f(s) \, {\rm d}s=0$).
Since neither $\cU^{0,n+1}$ nor the term $M(\cU^{0,n},\partial_{\theta_0}\cU^{0,n+1})$ in \eqref{e21z} has moment
zero\footnote{More precisely, we refer here to the moments of profiles like $\sigma^n_i(x,\theta)$ or products of profiles
that appear in these terms.}, the definition of ${\bf R}_\infty$ shows that this choice of $\cU^{1,n+1}_\eps$ generally
cannot lie in any $E^s_{T_0}$ space. We first try to remedy this problem using an idea inspired by an argument in
\cite{altermanrauch}. We replace $\cU^{0,n}$ (and similarly $\cU^{0,n+1}$) by a function $\cU^{0,n}_p$ defined by
functions  $\sigma^n_{i,p}$ with vanishing first moments, where
 \begin{align*}
\hat\sigma^n_{i,p}(x,m) :=\chi_p(m) \, \hat\sigma^n_{i}(x,m),\quad 0<p<1\, ,
\end{align*}
and $\chi_p(m)$ is a low frequency cutoff function vanishing on a neighborhood of $0$ of size $O(|p|)$ and equal to
one outside a slightly larger neighborhood\footnote{The cutoff renders harmless the small divisor that appears when
one writes the Fourier transform of the $\theta$-primitive of $\sigma^n_{i,p}$ in terms of $\hat\sigma^n_{i,p}(x,m)$.}.
We show the estimate
\begin{align*}
|\cU^{0,n}-\cU^{0,n}_p|_{\cE^{s-1}_{T_0}} \leq C \, \sqrt{p}.
\end{align*}
These ``moment zero approximations" (Definition \ref{h5}) are the pulse analogues of the trigonometric polynomial
approximations, which can be viewed as produced by high frequency cutoffs, used in the error analysis in the
wavetrain case in \cite[section 2.5]{CGW1}. With this change the contribution to $\cU^{1,n+1}_\eps$ from
\begin{align*}
-{\bf R}_\infty \, (I-{\bf E}) \, \left( \tilde L(\partial_x) \, \cU^{0,n+1}_p -F(0) \, \cU^{0,n}_p \right)
\end{align*}
lies in a suitable $E^r_{T_0}$ space, but there is a problem due to ``self-interaction terms" of the form
\begin{align*}
\sigma^n_{i,p}(x,\theta) \, \partial_\theta \sigma^{n+1}_{i,p}(x,\theta)
\end{align*}
coming from the $M$ term in \eqref{e21z}, which do not have moment zero. Thus, we replace these terms by
$(\sigma^n_{i,p} \, \partial_\theta \sigma^{n+1}_{i,p})_p$ as in \eqref{k2}. The transversal interaction terms
$\sigma^n_{i,p} \, \partial_\theta\sigma^{n+1}_{j,p}$, $i\neq j$ already yield contributions in an $E^r_{T_0}$ space (Proposition \ref{h12}).

Using moment-zero approximations introduces errors that blow up as $p\to 0$, of course, but taking $p=\eps^b$
for an appropriate $b>0$, one can hope to control these errors using the factor $\eps$ in $\eps \, \cU^{1,n+1}$.
Indeed, this works and by the process outlined above we obtain  a corrector $\cU^{1,n+1}_{p,\eps}$ which, though
it does not solve away $R^{n+1}_\eps$,  solves away ``all but $O(\sqrt{p}+\frac{\eps}{p^{M_1+2}})$" of $R^{n+1}_\eps$
in $E^{s-3}_{T_0}$ (see \eqref{e31} for more details). Setting $p=\eps^b$ and choosing the exponent $b$ so that
$\sqrt{p}=\frac{\eps}{p^{M_1+2}}$ (so $b=\frac{2}{2M_1+5}$), we are able to apply the estimate of Proposition
\eqref{e9} to conclude
\begin{align*}
|\cU^{0,n+1}_{\eps}-U^{n+1}_\eps|_{E^{s-3}_{T_0}} \leq C\, \eps^{\frac{1}{2M_1+5}}\text{ where }
M_1=\left[ \frac{d}{2}+3 \right].
\end{align*}

\begin{rem}[Uniformly stable shocks]
\textup{There is an  analogue of Theorem \ref{a16} for uniformly stable shock waves perturbed by pulses. Uniform
stability for the shock waves problem is an extension of Definition \ref{a6} and dates back to Majda \cite{M1}. The
case of shocks perturbed by highly oscillatory wavetrains was studied in \cite[section 3]{CGW1}. In that case there
is a separate expansion for the oscillating shock front (a free boundary)
\begin{align}\label{a18}
\psi_\eps(x') \thicksim \sigma \, x_0 +\eps \, \left( \chi^0(x') +\eps \, \chi^1 \left( x',\dfrac{\phi^0(x')}{\eps} \right) \right) \, ,
\end{align}
in addition to an expansion for the solution on each side of the front. An important difference in the pulse case is that
the term $\chi^0(x')$ is absent in \eqref{a18}, and of course $\chi^1(x',\theta_0)$ is now decaying instead of periodic
in $\theta_0$. The expansions of the reflected waves on either side of the front are similar to \eqref{a10}.}

\textup{The singular shock problem has the same form as in the wavetrain case (see equations (3.39) of \cite{CGW1}
and \cite{williams4}), and the profile equations again take the form of equations (3.60) in \cite{CGW1}, except that every
occurrence of $\chi^0(x')$ is replaced by $0$. The solution of the large system for the leading profiles is now considerably
simpler than equations (3.68) of \cite{CGW1}, since all the interaction integrals in that equation are now absent. This
reflects the fact that pulses of different families do not interact at the leading order while wavetrains do. However, it is
necessary to estimate interaction integrals in the error analysis. As in the pulse problem with fixed boundaries, one
can in the shock problem obtain a rate of convergence of approximate solutions to exact solutions as $\eps\to 0$.}

\end{rem}

\section{Exact solution of the singular problem}
\label{exact}

The goal of this section is to prove Theorem \ref{a12} and solve the singular system \eqref{a4b}. This is achieved,
as in \cite[section 7]{williams3}, by solving the sequence of linear problems
\begin{align}\label{linexact}
\begin{split}
&a)\; \partial_{x_d} U^{n+1}_\eps +\sum^{d-1}_{j=0} \tilde{A}_j(\eps U^n_\eps) \, \left(
\partial_{x_j} +\dfrac{\beta_j \partial_{\theta_0}}{\eps} \right) \, U^{n+1}_\eps =F(\eps U^n_\eps) \, U^n_\eps \, ,\\
&b)\; B(\eps U^{n}_\eps)\, U^{n+1}_\eps|_{x_d=0} =G(x',\theta_0) \, ,\\
&c)\; U^{n+1}_\eps=0 \text{ in } t<0.
\end{split}
\end{align}
As for the case of hyperbolic boundary value problems, that is without the singular parameter $1/\eps$ in the
differential operator, see e.g. \cite{BS,CP}, the solvability of each linear system \eqref{linexact} relies on some
a priori estimates. Our main focus here is the derivation of a priori estimates that are uniform with respect to
the wavelength $\eps$. For the reasons detailed in the introduction of \cite{williams3}, the appropriate functional
setting in which one can derive uniform estimates is provided by the spaces $E^s_T$ defined in \eqref{a5}. The
main difficulty is to obtain $L^\infty$ estimates uniform in $\eps$.  These cannot be obtained simply from uniform $L^2(x_d,H^2(x',\theta_0))$ estimates since the estimate of $\partial_{x_d}U_\eps$ in terms of tangential derivatives provided by the system \eqref{linexact} blows up as $\eps\to 0$. Much of the analysis in this section is similar to \cite[sections 5 and 7]{williams3},
except that we use here the singular pseudodifferential calculus
of Appendix \ref{append}. This introduces some
modifications for the regularity assumptions in the results below.  In the proofs of this section we shall often refer to \cite{williams3} in order to keep the exposition as short as possible.

\subsection{Main estimate for the linearized singular problem}
\label{mainest}

We consider a linearized problem of the form
\begin{align}\label{mainlin}
\begin{split}
&a)\; \partial_{x_d} U_\eps +\sum^{d-1}_{j=0} \tilde{A}_j(\eps V_\eps) \, \left(
\partial_{x_j} +\dfrac{\beta_j \partial_{\theta_0}}{\eps} \right) \, U_\eps =f_\eps \, ,\\
&b)\; B(\eps V_\eps)\, U_\eps|_{x_d=0} =g_\eps \, ,\\
&c)\; U_\eps=0 \text{ in } t<0,
\end{split}
\end{align}
where $(V_\eps)_{\eps \in (0,1]}$ is a given family of functions, and $(f_\eps,g_\eps)$ represent some source
terms. Our first main result is the analogue of \cite[Theorems 5.1 and 5.2]{williams3} and proves unique solvability
with a uniform $L^2$ energy estimate for \eqref{mainlin}. The main point is to keep track of the regularity
assumptions on the coefficients $V_\eps$.

\begin{theo}\label{est}
Let $s_0 :=[(d+1)/2]+1$. There exists $\delta>0$ such that, for all $K \ge 1$, there exist some constants
$\gamma_0(K) \ge 1$ and $C_0(K)>0$ such that the following property holds: if the coefficients $(V_\eps
)_{\eps \in (0,1]}$ in \eqref{mainlin} satisfy
\begin{equation}\label{e4.17}
|\eps \, V_\eps|_{L^\infty(\Omega)} \le \delta \, ,\quad
|V_\eps|_{C^{0,M_0}(\Omega)} +|V_\eps|_{C(H^{s_0}(\R^d \times \R))}
+|\eps \, \partial_{x_d}V_\eps|_{L^\infty(\Omega)} \leq K\, ,
\end{equation}
then for all $T>0$, for all source terms $f_\eps \in L^2 (\Omega_T)$, $g_\eps \in L^2(b\Omega_T)$ vanishing
for $t<0$, there exists a unique solution $U_\eps \in L^2 (\Omega_T)$ to \eqref{mainlin} vanishing for $t<0$,
and this solution satisfies
\begin{equation}
\label{e4.35}
|{\rm e}^{-\gamma \, t} U_\eps|_{0,0,T} +\frac{1}{\sqrt{\gamma}} \,
\langle {\rm e}^{-\gamma \, t} U_\eps|_{x_d=0} \rangle_{0,T} \leq C_0(K) \left( \frac{1}{\gamma} \,
|{\rm e}^{-\gamma \, t} f_\eps|_{0,0,T} +\frac{1}{\sqrt{\gamma}} \,
\langle {\rm e}^{-\gamma \, t} g_\eps \rangle_{0,T} \right) \, ,
\end{equation}
for all $\gamma \ge \gamma_0(K)$.
\end{theo}

In Theorem \ref{est}, the space $C^{0,M_0}(\Omega)$ denotes the space of functions $v(x,\theta)$ such that
for all $x_d \ge 0$, $v(\cdot,x_d,\cdot)$ is bounded on $\R^d \times \R$  with all derivatives up to the order $M_0$
bounded, and with all bounds that are uniform in $x_d$. The norm is defined by
\begin{equation*}
|v|_{C^{0,M_0}(\Omega)} := \sup_{x_d \ge 0} \, \sup_{|\alpha| \le M_0}
\| \partial_{x',\theta}^\alpha v (\cdot,x_d,\cdot) \|_{L^\infty (\R^d \times \R)} \, .
\end{equation*}
For fixed $x_d$, the $(x',\theta)$-regularity of a symbol enables us to use some of the symbolic calculus rules
listed in Appendix \ref{append}.

\begin{proof}
The first main step in the proof of Theorem \ref{est} is to show a global in time a priori estimate.
In other words, we consider a smooth function $U_\eps$ solution to \eqref{mainlin}, and wish to
show the estimate \eqref{e4.35} with $T=+\infty$. We begin with the following result.

\begin{theo}[Kreiss, M\'etivier \cite{kreiss,metivier}]
\label{theokm}
There exists $\delta>0$ such that, if $B_\delta$ denotes the closed ball of radius $\delta$ in $\R^N$,
there exists an $m\times m$ matrix-valued function
\[
R \in C^\infty(B_\delta \times \mathbb{R}^d \times (0,\infty)),
\]
homogeneous of degree zero in $(\xi',\gamma)$ and satisfying:

(a)\;$R(v,\xi',\gamma)=R(v,\xi',\gamma)^*$;

(b)\;there exist $C>0$, $c>0$ such that for all $(v,\xi',\gamma)$:
\begin{align}\label{e4.11}
R(v,\xi',\gamma) +C\; B^*(v) \, B(v)\geq c \, I ;
\end{align}

(c)\;there exist finite sets of $C^\infty$ matrices on $B_\delta \times \mathbb{R}^d \times (0,\infty)$,
denoted $T_l$, $H_l$, and $E_l$ such that
\begin{align*}
(i)\;\text{\rm Re } (R(v,\xi',\gamma) \, \mathcal{A}(v,\xi',\gamma)) =\sum_l T_l(v,\xi',\gamma) \, \begin{pmatrix}
\gamma \, H_l(v,\xi',\gamma) & 0\\
0 & E_l(v,\xi',\gamma) \end{pmatrix} \, T^*_l(v,\xi',\gamma);
\end{align*}

(ii)\; $T_l$, $H_l$ are homogeneous of degree zero in $(\xi',\gamma)$, $E_l$ is homogeneous of degree one;

(iii)\;$H_l(v,\xi',\gamma)=H^*_l(v,\xi',\gamma)$, $E_l(v,\xi',\gamma)=E^*_l(v,\xi',\gamma)$;

(iv)\;there exists $c >0$ such that
\begin{align*}
\sum_l T_l(v,\xi',\gamma) \, T^*_l(v,\xi',\gamma)\geq c\, I, \; H_l(v,\xi',\gamma)\geq c\, I, \;
E_l(v,\xi',\gamma)\geq c\, (|\xi'|+\gamma)\, I.
\end{align*}
The dimensions of $H_l$ and $E_l$ can vary with $l$.
\end{theo}

The parameter $\delta$ is fixed according to Theorem \ref{theokm}. We then define the following
singular symmetrizer for the boundary value problem \eqref{mainlin}:
\begin{equation*}
\cR_{\eps,\gamma} := \opeg (R(\eps V_\eps,\xi',\gamma)) \, ,
\end{equation*}
where singular pseudodifferential operators $\opeg (a)$ associated with a symbol $a$ are defined in
Appendix \ref{append}. We observe, as in \cite[remark 5.2]{williams3} that our symmetrizer is not self-adjoint
on $L^2(\Omega)$. However, the remainder $\cR_{\eps,\gamma}-\cR_{\eps,\gamma}^*$ is $O(1/\gamma)$
as an operator on $L^2$, uniformly in $\eps$.

Under the regularity assumptions \eqref{e4.17} of Theorem \ref{est}, the results given in Appendix
\ref{append} and the arguments in \cite[pages 164-165]{williams3} give the following properties for
the symmetrizer $\cR_{\eps,\gamma}$:
\begin{align}\label{e4.16a}
\begin{split}
&(a)\; |\cR_{\eps,\gamma} \, W|_{0,0} \leq C(K)\, |W|_{0,0} \, ,\\
&(b)\; |[\partial_{x_d},\cR_{\eps,\gamma}] \, W|_0\leq C(K)\, |W|_0 \, ,\\
&(c)\; \text{Re } ((\cR_{\eps,\gamma} \, \mathcal{A}_{\eps,\gamma} \,
+\mathcal{A}_{\eps,\gamma}^* \, \cR_{\eps,\gamma}) \, W,W) \geq c(K)\, \gamma \, |W|_{0,0}^2 \, ,\\
&(d)\; \text{Re } \langle \cR_{\eps,\gamma} \, W,W \rangle +C(K)\, \langle B(\eps V_\eps) \, W\rangle^2_0
\geq c(K)\, \langle W\rangle^2_0 \, ,
\end{split}
\end{align}
where $\mathcal{A}_{\eps,\gamma}$ denotes the operator
\begin{equation*}
-\gamma \, \tilde{A}_0(\eps V_\eps)-\sum^{d-1}_{j=0} \tilde{A}_j(\eps V_\eps) \, \left(
\partial_{x_j} +\dfrac{\beta_j \partial_{\theta_0}}{\eps} \right) \, .
\end{equation*}
Let us focus for instance on property $(d)$ in \eqref{e4.16a}. Since $\cR_{\eps,\gamma}^* =\cR_{\eps,\gamma}
+O(1/\gamma)$, G{\aa}rding's inequality (Theorem \ref{thm11}) shows that it is sufficient to prove that the symbol
$R(\eps V_\eps,\xi',\gamma)+C\, B^*(\eps V_\eps) \, B(\eps V_\eps)$ is positive definite, and this property is
given by \eqref{e4.11}. Other properties in \eqref{e4.16a} are obtained by similar arguments (applying Propositions
\ref{prop18}, \ref{prop19} or \ref{prop20}), see \cite[pages 164-165]{williams3} for more details.

We perform the change of function $U_\eps \rightarrow {\rm e}^{-\gamma \, t} U_\eps$ in \eqref{mainlin},
multiply \eqref{mainlin} $a)$ by $\cR_{\eps,\gamma}$ and take the real part of the $L^2$ scalar product
with ${\rm e}^{-\gamma \, t} U_\eps$. The estimates \eqref{e4.16a} yield\footnote{The detailed computations
can be found in \cite[corollary 5.2]{williams3}, and are the singular analogue of \cite{BS,CP}.}
\begin{equation*}
|{\rm e}^{-\gamma \, t} U_\eps|_{0,0} +\frac{1}{\sqrt{\gamma}} \,
\langle {\rm e}^{-\gamma \, t} U_\eps|_{x_d=0} \rangle_0 \leq C(K) \left( \frac{1}{\gamma} \,
|{\rm e}^{-\gamma \, t} f_\eps|_{0,0} +\frac{1}{\sqrt{\gamma}} \,
\langle {\rm e}^{-\gamma \, t} g_\eps \rangle_0 \right) \, ,
\end{equation*}
for $\gamma$ sufficiently large, that is for all $\gamma \ge \gamma_0(K)$.

A similar uniform a priori estimate is valid for the dual problem (which satisfies the backward uniform Lopatinskii
condition). Then the arguments of \cite{BS,CP}, namely existence of a weak solution and "weak=strong" by
tangential mollification, yields well-posedness of the boundary value problem \eqref{mainlin}. Localization
in time is achieved as usual by showing a causality principle ("future does not affect the past"), which holds
in our context since the constant $C(K)$ in our energy estimate is independent of $\gamma$.
\end{proof}

The uniform $L^2$ estimate \eqref{e4.35} enables us to show an estimate in the space $E^0$ defined in
\eqref{a5}. The result is similar to \cite[Corollary 7.1]{williams3} with a slight improvement with respect to the norm in
which the source term $f_\eps$ is estimated.

\begin{theo}\label{est'}
Let $s_0 :=[(d+1)/2]+1$. There exists $\delta>0$ such that, for all $K \ge 1$, there exist some constants
$\gamma_1(K) \ge 1$ and $C_1(K)>0$ such that the following property holds: if the coefficients $(V_\eps
)_{\eps \in (0,1]}$ in \eqref{mainlin} satisfy \eqref{e4.17}, then for all $T>0$, for all source terms $f_\eps
\in L^2 (H^1(b\Omega_T))$, $g_\eps \in H^1(b\Omega_T)$ vanishing for $t<0$, there exists a unique solution
$U_\eps \in H^1 (\Omega_T)$ to \eqref{mainlin} vanishing for $t<0$, and this solution satisfies
\begin{equation}
\label{estE0}
|{\rm e}^{-\gamma \, t} U_\eps|_{\infty,0,T}
+|{\rm e}^{-\gamma \, t} U_\eps|_{0,1,T} +\frac{1}{\sqrt{\gamma}} \,
\langle {\rm e}^{-\gamma \, t} U_\eps|_{x_d=0} \rangle_{1,T} \leq C_1(K) \left( \frac{1}{\gamma} \,
|{\rm e}^{-\gamma \, t} f_\eps|_{0,1,T} +\frac{1}{\sqrt{\gamma}} \,
\langle {\rm e}^{-\gamma \, t} g_\eps \rangle_{1,T} \right) \, ,
\end{equation}
for all $\gamma \ge \gamma_1(K)$.
\end{theo}

\begin{proof}
The regularity of the solution $U_\eps$ can be obtained by using the same arguments as in \cite[chapter 7]{CP},
that is by commuting the system \eqref{mainlin} with a mollified version of the Fourier multiplier of symbol
$(\gamma^2+|\xi'|^2+k^2)^{1/2}$. The argument shows that the $\partial_{x'}$ and $\partial_{\theta}$
derivatives of $U_\eps$ are in $L^2$, and \eqref{mainlin} then shows that the $\partial_{x_d}$ derivative
of $U_\eps$ also belongs to $L^2$. We thus only show the estimate \eqref{estE0}.

\textbf{1. $L^2$ estimate of tangential derivatives.} Commuting \eqref{mainlin} with a tangential derivative
$\partial_{tan} \in \{ \partial_{x_0},\dots,\partial_{x_{d-1}},\partial_{\theta_0} \}$, we need to control the
commutators
\begin{equation*}
\sum_{j=0}^{d-1} [\tA_j (\eps \, V_\eps),\partial_{tan}] \, \left(
\partial_{x_j} +\dfrac{\beta_j \partial_{\theta_0}}{\eps} \right) \, U_\eps
=\sum_{j=0}^{d-1} \big( {\rm d}\tA_j (\eps \, V_\eps) \cdot \partial_{tan} V_\eps \big) \, \left(
\eps \, \partial_{x_j} +\beta_j \, \partial_{\theta_0} \right) \, U_\eps \, .
\end{equation*}
When multiplied by ${\rm e}^{-\gamma t}$, this source term is bounded in $L^2(\Omega_T)$ by a
constant times $|{\rm e}^{-\gamma \, t} U_\eps|_{0,1,T}$ and can therefore be absorbed from right to
left by choosing $\gamma$ large. At this stage, we have
\begin{equation}
\label{estE01}
|{\rm e}^{-\gamma \, t} U_\eps|_{0,1,T} +\frac{1}{\sqrt{\gamma}} \,
\langle {\rm e}^{-\gamma \, t} U_\eps|_{x_d=0} \rangle_{1,T} \leq C_1(K) \left( \frac{1}{\gamma} \,
|{\rm e}^{-\gamma \, t} f_\eps|_{0,1,T} +\frac{1}{\sqrt{\gamma}} \,
\langle {\rm e}^{-\gamma \, t} g_\eps \rangle_{1,T} \right) \, ,
\end{equation}
for all $\gamma$ large enough.

\textbf{2. $L^\infty(L^2)$ estimate, part 1.} We extend $f_\eps$ and $g_\eps$ beyond time $T$, which does
not affect the solution $U_\eps$ up to time $T$. Doing so, we just need to prove the $L^\infty(L^2)$ estimate
\eqref{estE0} for $T=+\infty$. We consider a cut-off function $\chi^e$ in the extended singular calculus,
that is a smooth function satisfying the conditions \eqref{n31} given in Appendix \ref{append}. The $L^\infty
(L^2)$ estimate is first proved on $(1-\chi^e_s(D)) \, ({\rm e}^{-\gamma \, t} U_\eps)$, where we let from
now on $\chi^e_s(D)$ denote the Fourier multiplier whose symbol is
\begin{equation*}
\chi^e \left( \xi',\dfrac{k\, \beta}{\eps},\gamma \right) \, .
\end{equation*}
Since $|k \, \beta|/\eps$ is dominated by $(\gamma^2 +|\xi'|^2)^{1/2}$ on the support of $1-\chi^e$, the same
arguments as in \cite[page 173]{williams3} yield
\begin{align}
|(1-\chi^e_s(D)) \, ({\rm e}^{-\gamma \, t} U_\eps)|_{\infty,0} &\le C(K) \, \big( |{\rm e}^{-\gamma \, t} f_\eps|_{0,0}
+|{\rm e}^{-\gamma \, t} U_\eps|_{0,1} \big) \notag \\
&\leq C(K) \left( \frac{1}{\gamma} \, |{\rm e}^{-\gamma \, t} f_\eps|_{0,1} +\frac{1}{\sqrt{\gamma}} \,
\langle {\rm e}^{-\gamma \, t} g_\eps \rangle_1 \right) \, .\label{estE02}
\end{align}

\textbf{3. $L^\infty(L^2)$ estimate, part 2.} It remains to estimate $|\chi^e_s(D) \, ({\rm e}^{-\gamma \, t} U_\eps)
|_{\infty,0}$, which uses the fact that $\beta$ is a hyperbolic frequency. More precisely, we can fix some parameters
$\delta>0$ and $\delta_2>0$ such that for all $v$ in the ball of radius $\delta$ and for all $(z,\eta)$ that are
$\delta_2$-close to $\beta$, there holds
\begin{equation*}
Q(v,z,\eta)^{-1} \, \cA (v,z,\eta) \, Q(v,z,\eta) =\text{\rm diag } (\lambda_1(v,z,\eta),\dots,\lambda_N(v,z,\eta)) \, ,
\end{equation*}
for a suitable invertible matrix $Q$, and the $\lambda_j$'s satisfy
\begin{equation*}
\text{\rm Re } \lambda_j(v,z,\eta) \begin{cases}
\le -c\, \gamma \, ,&\text{if $j=1,\dots,p$,}\\
\ge c\, \gamma \, ,&\text{if $j=p+1,\dots,N$.}\\
\end{cases}
\end{equation*}
Moreover, Assumption \ref{a7} shows that the (square) matrix whose column vectors are
\begin{equation*}
B(v)\, Q_1(v,z,\eta),\dots, B(v)\, Q_p(v,z,\eta) \, ,
\end{equation*}
is invertible (here the $Q_j$'s denote the columns of $Q$).

With the above notation, we can follow the proof of \cite[Proposition 7.3]{williams3}, and write
$\chi^e_s(D) \, ({\rm e}^{-\gamma \, t} U_\eps)$ under the form
\begin{equation*}
\chi^e_s(D) \, ({\rm e}^{-\gamma \, t} U_\eps) =r_0 \, \cW \, ,
\end{equation*}
where, here and from now on, $r_0$ denotes a bounded operator on $L^2(\Omega)$ whose operator norm
is independent of $\eps,\gamma$, and where each component $\cW_j$ of $\cW$ satisfies a transport equation
\begin{equation}
\label{eqWj}
\partial_{x_d} \cW_j -\lambda_j(\eps \, V_\eps,D_s) \, \cW_j =r_0 ({\rm e}^{-\gamma \, t} f_\eps)
+r_0 ({\rm e}^{-\gamma \, t} U_\eps) \, .
\end{equation}
In \eqref{eqWj}, $\lambda_j(\eps \, V_\eps,D_s)$ denotes the singular pseudodifferential operator of symbol
$\lambda_j(\eps \, V_\eps,z,\eta)$ (as described in Appendix \ref{append}).

In the outgoing case ($j=p+1,\dots,N$), we multiply \eqref{eqWj} by $\overline{\cW_j}$, integrate from
$x_d$ to $+\infty$ and apply G{\aa}rding's inequality (Theorem \ref{thm11}), obtaining
\begin{multline*}
\langle \cW_j(x_d) \rangle^2_0+\gamma \, \int_{x_d}^{+\infty} \langle \cW_j(y) \rangle^2_0 \, {\rm d}y \\
\leq C\, \int_{x_d}^{+\infty} \langle \cW_j(y) \rangle_0 \, \langle {\rm e}^{-\gamma \, t} f_\eps(y) \rangle_0
\, {\rm d}y
+C\, \int_{x_d}^{+\infty} \langle \cW_j(y) \rangle_0 \, \langle {\rm e}^{-\gamma \, t} U_\eps(y) \rangle_0
\, {\rm d}y \, .
\end{multline*}
The contribution of $\cW_j$ on the right-hand side can be absorbed on the left by using Young's inequality,
and Theorem \ref{est} enables us to control the $L^2$ norm of $U_\eps$. We thus get
\begin{equation}
\label{estE03}
\sup_{j=p+1,\dots,N} |\cW_j|_{\infty,0}^2 \le C(K) \left( \frac{1}{\gamma} \,
|{\rm e}^{-\gamma \, t} f_\eps|_{0,0}^2 +\frac{1}{\gamma^2} \,
\langle {\rm e}^{-\gamma \, t} g_\eps \rangle_0^2 \right) \, .
\end{equation}

The estimates in the incoming case are similar, except that we integrate from $0$ to $x_d$. We thus get
\begin{equation*}
\sup_{j=1,\dots,p} |\cW_j|_{\infty,0}^2 \le \sup_{j=1,\dots,p} \langle \cW_j|_{x_d=0} \rangle_0^2
+C(K) \, \left( \frac{1}{\gamma} \, |{\rm e}^{-\gamma \, t} f_\eps|_{0,0}^2 +\frac{1}{\gamma^2} \,
\langle {\rm e}^{-\gamma \, t} g_\eps \rangle_0^2 \right) \, .
\end{equation*}
Using the same arguments as in \cite[page 178]{williams3}, we can use the uniform Lopatinskii condition
and write
\begin{equation*}
\begin{pmatrix}
\cW_1|_{x_d=0} \\
\vdots \\
\cW_p|_{x_d=0} \end{pmatrix} =r_0 \, \begin{pmatrix}
\cW_{p+1}|_{x_d=0} \\
\vdots \\
\cW_N|_{x_d=0} \end{pmatrix} +r_0 \, ({\rm e}^{-\gamma \, t} g_\eps)
+\dfrac{1}{\gamma} \, r_0 \, ({\rm e}^{-\gamma \, t} U_\eps|_{x_d=0})\, ,
\end{equation*}
from which we derive the estimate
\begin{equation*}
\sup_{j=1,\dots,p} \langle \cW_j|_{x_d=0} \rangle_0^2 \le C(K) \, \sup_{j=p+1,\dots,N} |\cW_j|_{\infty,0}^2
+C(K) \, \langle {\rm e}^{-\gamma \, t} g_\eps \rangle_0^2
+\dfrac{C(K)}{\gamma^3} \, |{\rm e}^{-\gamma \, t} f_\eps|_{0,0}^2 \, .
\end{equation*}
We combine the latter inequality with \eqref{estE03}, and recall $\chi^e_s(D) \, ({\rm e}^{-\gamma \, t} U_\eps)
=r_0 \, \cW$, so we get
\begin{equation*}
|\chi^e_s(D) \, ({\rm e}^{-\gamma \, t} U_\eps)|_{\infty,0} \le C(K) \, \left( \frac{1}{\sqrt{\gamma}} \,
|{\rm e}^{-\gamma \, t} f_\eps|_{0,0}^2 +\langle {\rm e}^{-\gamma \, t} g_\eps \rangle_0^2 \right) \, .
\end{equation*}
Adding with \eqref{estE02} and \eqref{estE01}, we complete the proof of Theorem \ref{est'}.
\end{proof}

\subsection{Construction of the exact solution}

We use the iteration scheme \eqref{linexact} to solve the nonlinear system \eqref{a4b}. As usual, the convergence
of the iteration scheme follows from the combination of two arguments: a uniform boundedness in a "high norm"
(here in the space $E^s$ given in \eqref{a5}), and a contraction property in a "low norm" (here in $E^0$). The
estimate of a solution in $E^0$ will be provided by Theorem \ref{est'} above, and we indicate below how we
obtain the estimate of a solution in $E^s$, $s \in \N$.

\begin{proposition}
\label{estimEk}
Let $s_0 :=[(d+1)/2]+1$ and let $k \in \N$. There exists $\delta>0$ such that, for all $K \ge 1$, there exist some
constants $\gamma_k(K) \ge 1$ and $C_k(K)>0$ such that the following property holds: if the coefficients
$(V_\eps)_{\eps \in (0,1]}$ in \eqref{mainlin} satisfy \eqref{e4.17} and belong to $L^2 (H^{k+1}(b\Omega_T))
\cap L^\infty(H^k(b\Omega_T))$, then for all $T>0$, for all source terms $f_\eps \in L^2 (H^{k+1}(b\Omega_T))$,
$g_\eps \in H^{k+1}(b\Omega_T)$ vanishing for $t<0$, there exists a unique solution $U_\eps \in L^2 (H^{k+1}
(b\Omega_T)) \cap L^\infty(H^k(b\Omega_T))$ to \eqref{mainlin} vanishing for $t<0$, and this solution satisfies
\begin{multline}
\label{estEk}
|{\rm e}^{-\gamma \, t} U_\eps|_{\infty,k,T}
+|{\rm e}^{-\gamma \, t} U_\eps|_{0,k+1,T} +\frac{1}{\sqrt{\gamma}} \,
\langle {\rm e}^{-\gamma \, t} U_\eps|_{x_d=0} \rangle_{k+1,T} \leq C_k(K) \left( \frac{1}{\gamma} \,
|{\rm e}^{-\gamma \, t} f_\eps|_{0,k+1,T} \right. \\
\left. +\frac{1}{\sqrt{\gamma}} \, \langle {\rm e}^{-\gamma \, t} g_\eps \rangle_{k+1,T}
+|U_\eps|_{L^\infty(W^{1,\infty}(b\Omega_T))} \, \left( \dfrac{|{\rm e}^{-\gamma \, t} V_\eps|_{0,k+1,T}}{\gamma}
+\dfrac{|{\rm e}^{-\gamma \, t} V_\eps|_{x_d=0}|_{k+1,T}}{\sqrt{\gamma}} \right)
\right) \, ,
\end{multline}
for all $\gamma \ge \gamma_k(K)$.
\end{proposition}

\begin{proof}
The proof is like that of \cite[Theorem 7.2]{williams3}. One commutes \eqref{mainlin}
with a tangential derivative $\partial^\alpha$ of order $1 \le |\alpha| \le k$, and applies (tangential)
Gagliardo-Nirenberg inequalities. When the additional fast variable $\theta$ lies in the torus $\R/\Z$, these
inequalities are given in \cite[Lemma 7.3]{williams3}, and we claim that the exact same inequalities are valid
when the fast variable lies in $\R$.

When commuting \eqref{mainlin} with a tangential derivative $\partial^\alpha$, one applies Theorem
\ref{est'} and needs to control the commutator
\begin{equation*}
\left[ \tilde{A}_j(\eps V_\eps) \, \left( \partial_{x_j} +\dfrac{\beta_j \partial_{\theta_0}}{\eps} \right) ;\partial^\alpha \right]
\, U_\eps \, ,
\end{equation*}
in the norm $|{\rm e}^{-\gamma \, t} \cdot|_{0,1,T}$. The $\eps$ factor in front of $V_\eps$ cancels
the singular $1/\eps$ factor and we obtain the estimate
\begin{multline*}
\left| {\rm e}^{-\gamma \, t} \left[ \tilde{A}_j(\eps V_\eps) \, \left( \partial_{x_j} +\dfrac{\beta_j \partial_{\theta_0}}{\eps}
\right) ;\partial^\alpha \right] \, U_\eps \right|_{0,1,T} \le C(K) \, |{\rm e}^{-\gamma \, t} U_\eps|_{0,k+1,T} \\
+C(K) \, |U_\eps|_{L^\infty(W^{1,\infty}(b\Omega_T))} \, |{\rm e}^{-\gamma \, t} V_\eps|_{0,k+1,T}\, .
\end{multline*}
The $|{\rm e}^{-\gamma \, t} U_\eps|_{0,k+1,T}$ term on the right hand-side is absorbed by choosing
$\gamma$ large enough, and we are left with \eqref{estEk}. (Estimates on the boundary are similar.)
\end{proof}

We can then deduce the main estimate in the space $E^k_T$ defined in \eqref{a5} (the proof is the same
as that of \cite[Corollary 7.2]{williams3} and is based on the choice $T=1/\gamma$ in Proposition \ref{estimEk}).

\begin{cor}
\label{corEk}
Let $k \ge M_0+[\frac{d+1}{2}]$ and $K_1,K_2 \ge 1$. Then there exist a constant $C(K_1,K_2)>0$, a parameter
$\eps_0(K_1,K_2) \in (0,1]$ and a time ${\mathcal T}(K_1,K_2)>0$ satisfying the following property: if $T \le
{\mathcal T}(K_1,K_2)$, if the coefficients $(V_\eps)_{\eps \in (0,1]}$ in \eqref{mainlin} belong to $E^k_T$ and satisfy
\begin{equation}\label{taillecoeff}
|V_\eps|_{\infty,k,T} +|V_\eps|_{x_d=0}|_{k+1,T} \le K_1 \, ,\quad
|\eps \, \partial_{x_d}V_\eps|_{L^\infty(\Omega_T)} \leq K_2 \, ,
\end{equation}
and if $\eps \le \eps_0(K_1,K_2)$, then for all source terms $f_\eps \in L^2 (H^{k+1}(b\Omega_T))$, $g_\eps \in H^{k+1}
(b\Omega_T)$ vanishing for $t<0$, there exists a unique solution $U_\eps \in E^k_T$ to \eqref{mainlin} vanishing
for $t<0$, and this solution satisfies
\begin{equation}
\label{mainestEk}
|U_\eps|_{\infty,k,T} +|U_\eps|_{0,k+1,T} +\sqrt{T} \, \langle U_\eps|_{x_d=0} \rangle_{k+1,T} \leq
C_k(K_1,K_2) \left( T \, |f_\eps|_{0,k+1,T} +\sqrt{T} \, \langle g_\eps \rangle_{k+1,T} \right) \, .
\end{equation}
\end{cor}

The parameter $\eps_0$ in Corollary \ref{corEk} is chosen so that \eqref{taillecoeff} implies $|\eps \, V_\eps
|_{L^\infty(\Omega_T)} \le \delta$ where $\delta$ is as in Proposition \ref{estimEk}.

We are now in a position to prove our main existence result for the singular system \eqref{a4b}. The norm
in the space $E^k_T$ is defined by
\begin{equation*}
|v|_{E^k_T} :=|v|_{\infty,k,T} +|v|_{0,k+1,T} \, .
\end{equation*}

\begin{theo}\label{estnl}
Let $K>0$ and let $k \ge M_0+[\frac{d+1}{2}]$. Then there exists a constant $K'>0$, a parameter
$\eps_0(K) \in (0,1]$ and a time ${\mathcal T}(K)>0$ satisfying the following property: the iteration
\eqref{linexact} with $U^0_\eps \equiv 0$ is well-defined for $0<T \le {\mathcal T}(K)$ and satisfies
\begin{equation*}
\forall \, n \in \N \, ,\quad \forall \eps \le \eps_0(K) \, ,\quad
|U^n_\eps|_{E^k_T} +|U^n_\eps|_{x_d=0}|_{k+1,T} \le K \, ,\quad
|\eps \, \partial_{x_d} U^n_\eps|_{L^\infty(\Omega_T)} \leq K' \, .
\end{equation*}
Moreover, the sequence $(U^n_\eps)$ converges towards a function $U_\eps$ in $E^{k-1}_T$,
uniformly with respect to $\eps \in (0,\eps_0(K)]$. The limit $U_\eps$ belongs to $E^k_T$ and is a solution
to \eqref{a4b}.
\end{theo}

\begin{proof}
The constant $K'$ is chosen such that, if $|U^n_\eps|_{E^k_T} \le K$, and if furthermore $|V_\eps|_{E^k_T}
\le K$, then one has
\begin{equation}
\label{choixK'}
\left| \eps \, F(\eps \, U^n_\eps) U_\eps^n -\sum_{j=0}^{d-1} \tA_j(\eps \, U^n_\eps) (\eps \, \partial_{x_j}
+\beta_j \, \partial_{\theta_0}) V_\eps \right|_{L^\infty(\Omega_T)} \leq K' \, ,
\end{equation}
independently of $\eps \in (0,1]$. Then the parameter $\eps_0$ is chosen as $\eps_0(K,K')$ given by
Corollary \ref{corEk}. The time ${\mathcal T}(K,K')>0$ is chosen accordingly. Assuming that the induction
assumption
\begin{equation*}
\forall \, j \le n \, ,\quad \forall \, \eps \le \eps_0(K) \, ,\quad
|U^j_\eps|_{E^k_T} +\langle U^j_\eps|_{x_d=0} \rangle_{k+1,T} \le K \, ,\quad
|\eps \, \partial_{x_d} U^j_\eps|_{L^\infty(\Omega_T)} \leq K' \, ,
\end{equation*}
holds (this is trivially true for $n=0$), we can apply the estimate \eqref{mainestEk} of Corollary \ref{corEk}
to the system \eqref{linexact} and obtain
\begin{align*}
|U^{n+1}_\eps|_{E^k_T} +\sqrt{T} \, \langle U^{n+1}_\eps|_{x_d=0} \rangle_{k+1,T}
&\leq C_k(K,K') \left( T \, |F(\eps U^n_\eps) U^n_\eps|_{0,k+1,T} +\sqrt{T} \, \langle G \rangle_{k+1,T} \right) \\
&\le C_k(K,K') \left( T \, C(K) +\sqrt{T} \, \langle G \rangle_{k+1,T} \right) \, .
\end{align*}
Since $\langle G \rangle_{k+1,T}$ tends to zero as $T$ tends to zero, we can choose the time $T$ small enough
so that the induction assumption implies
\begin{equation*}
\forall \, \eps \le \eps_0(K) \, ,\quad |U^{n+1}_\eps|_{E^k_T}
+\langle U^{n+1}_\eps|_{x_d=0} \rangle_{k+1,T} \le K \, .
\end{equation*}
Our choice of $K'$ in \eqref{choixK'} implies that the induction assumption propagates from the rank $n$ to the
rank $n+1$ because $|\eps \, \partial_{x_d} U^{n+1}_\eps|_{L^\infty(\Omega_T)} \leq K'$.

The convergence in $E^{k-1}_T$ is obtained by showing a contraction estimate in $E^0_T$, which is obtained
by applying Theorem \ref{est'}. We refer to \cite[page 184]{williams3} for the details. The limit $U_\eps$ of the
iteration scheme \eqref{linexact} is a solution to \eqref{a4b}, which yields the regularity $U_\eps \in E^k_T$ (see
\cite[chapter 9]{BS} for similar arguments).
\end{proof}

\section{Construction of the leading pulse profiles}\label{sect3}

\quad Observe that we can solve the system \eqref{28} by solving instead
\begin{align}\label{f1}
\begin{split}
&X_{\phi_i} \sigma_i +c^i_i \, \sigma_i \, \partial_{\theta}\sigma_i -e^i_i \, \sigma_i=0,\;i=1,2,3\\
&(\sigma_i(x',0,\theta), i =1,3)=\cB\left(G(x',\theta),\sigma_2(x',0,\theta)\right),\\
&\sigma_i=0\text{ in }t<0.
\end{split}
\end{align}
where all occurrences of $\theta_i$ or $\theta_0$ are now replaced by $\theta$. To solve \eqref{f1} we
use the iteration scheme
\begin{align}\label{f2}
\begin{split}
&(a)\;X_{\phi_i} \sigma_i^{n+1} +c^i_i \, \sigma_i^n\partial_{\theta}\sigma^{n+1}_i =e^i_i \, \sigma^{n}_i,\;i=1,2,3\\
&(b)\;(\sigma_i^{n+1}(x',0,\theta), i =1,3)=\cB\left(G(x',\theta),\sigma_2^{n+1}(x',0,\theta)\right),\\
&(c)\;\sigma_i^{n+1}=0\text{ in }t<0.
\end{split}
\end{align}
We will prove estimates for \eqref{f2} in a class of Sobolev spaces weighted in $\theta$. These weights are
introduced in order to get an explicit decay rate in $\theta$ at infinity.

\begin{defn}\label{c1}
For $s\in\N$ and $\gamma\geq 1$ define the spaces
\begin{align*}
\begin{split}
\Gamma^s &:=\Big\{ a(x,\theta) \in L^2 (\R^{d+1}_+ \times\R) : (\theta,\partial_{x},\partial_{\theta})^\beta a
\in L^2 \, \text{\rm for }|\beta|\leq s, \, \text{\rm and }a=0 \, \, \text{\rm in }t<0 \Big\} \, ,\\
\text{\rm and } \Gamma^s_\gamma &:=e^{\gamma t} \, \Gamma^s \, ,
\end{split}
\end{align*}
with respective norms
\begin{equation}
\label{c1a}
|a|_s := \sum_{|\beta|=|(\beta_1,\beta_2,\beta_3)|\leq s} |\theta^{\beta_1} \, \partial_x^{\beta_2} \,
\partial_\theta^{\beta_3} \, a|_{L^2(x,\theta)} \quad \text{\rm and } |a|_{s,\gamma}:=|e^{-\gamma t}a|_s.
\end{equation}
We will let $H^s$ and $H^s_\gamma$ denote the usual Sobolev spaces with norms defined just as in \eqref{c1a}
but without the $\theta$ weights. These spaces and those below have the obvious meanings when $a$ is vector-valued.
\end{defn}

\begin{rem}\label{c2}
\textup{We have
\begin{equation*}
|a|_{s,\gamma} \sim \sum_{|\beta|\leq s} \gamma^{s-|\beta|} \, |{\rm e}^{-\gamma t} \, \theta^{\beta_1} \, \partial_x^{\beta_2}
\, \partial_\theta^{\beta_3}a|_{L^2(x,\theta)} \sim \sum_{|\beta|\leq s} |{\rm e}^{-\gamma t} \, \theta^{\beta_1} \,
\partial_x^{\beta_2} \, \partial_\theta^{\beta_3}a|_{L^2(x,\theta)},
\end{equation*}
where ``$\sim$" denotes an equivalence of norms with constants independent of $\gamma \ge 1$. The second
equivalence follows from
\begin{equation*}
(\partial_t+\gamma) \, ({\rm e}^{-\gamma t} \, a)={\rm e}^{-\gamma t} \, \partial_t a.
\end{equation*}}
\end{rem}


The next proposition is helpful for estimating the commutators that arise when deriving $\Gamma^s$ estimates of
solutions to the linearization of the profile system \eqref{28}. Define
\begin{align*}
\Lambda^s :=\Big\{ a \in L^2 (\R^{d+1}_+ \times \R) : \quad &\theta^{\beta_1} \, a \in L^2 \, \, \text{\rm for }
|\beta_1|\leq s, \quad \partial_x^{\beta_2} \, a\in L^2 \, \, \text{\rm for } |\beta_2| \leq s,\\
&\partial_{\theta}^{\beta_3} \, a \in L^2 \, \, \text{\rm for }|\beta_3|\leq s,\quad a=0 \, \, \text{\rm in }t<0 \Big\},
\end{align*}
with
\begin{equation*}
|a|_{\Lambda^s} :=\sum_{|\beta_1| \leq s} |\theta^{\beta_1} \, a|_{L^2}
+\sum_{|\beta_2|\leq s} |\partial_x^{\beta_2} \, a|_{L^2}
+\sum_{|\beta_3|\leq s} |\partial_{\theta}^{\beta_3} \, a|_{L^2},
\end{equation*}
and let define $\Lambda^s_\gamma :={\rm e}^{\gamma t} \, \Lambda^s$ with the norm $|a|_{\Lambda^s_\gamma}
:=|{\rm e}^{-\gamma t}a|_{\Lambda^s}$ accordingly.

\begin{prop}\label{c4}
The spaces $\Gamma^s$ and $\Lambda^s$ are equal, and the norms $|a|_{s}$ and $|a|_{\Lambda^s}$ are equivalent:
there exists a constant $C_s$ such that
\begin{align}\label{c4a}
|a|_{\Lambda^s} \leq |a|_s \leq C_s \, |a|_{\Lambda^s} \, .
\end{align}
\end{prop}

\begin{proof}
Clearly, $|a|_{\Lambda^s} \leq |a|_s$. The remaining inequality is proved by induction on $s$.
The case $s=0$ is clear. The square $|a|^2_s$ is a sum of terms
\begin{align}\label{c5}
\int (\theta^{\beta_1} \, \partial_x^{\beta_2} \, \partial_\theta^{\beta_3} \, a) \,
(\theta^{\beta_1} \, \partial_x^{\beta_2} \, \partial_\theta^{\beta_3} a) \, {\rm d}x \, {\rm d}\theta \, ,
\end{align}
where $|\beta|=|\beta_1|+|\beta_2|+|\beta_3|\leq s$. The terms with $|\beta|<s$ are dominated by
$C|a|_{\Lambda^s}^2$ by the induction assumption.

Consider now a term like \eqref{c5} with $|\beta|=s>0$. Either $2\, |\beta_1|\geq s$ or $2\, |\beta_2,\beta_3| \geq s$.
Suppose $2\, |\beta_2,\beta_3|\geq s$. Perform integrations by parts to obtain terms of the form
\begin{equation}\label{c6a}
C\, \int\partial_{(x,\theta)}^{\alpha_1} \, a \cdot (\theta,\partial_x,\partial_\theta)^{\alpha_2} \, a \; {\rm d}x \, {\rm d}\theta
\leq C_\delta \, |\partial_{(x,\theta)}^{\alpha_1} \, a|^2_{L^2} +\delta \, |(\theta,\partial_x,\partial_\theta)^{\alpha_2} \,
a|^2_{L^2}, \;\;|\alpha_i|=s, \;i=1,2 \, ,
\end{equation}
and other terms (where powers of $\theta$ are differentiated) that can be estimated using the induction assumption.
The second term on the right in \eqref{c6a} can be absorbed by $|a|_{s}^2$. Integrations by parts show that the first
term on the right is dominated by the sum of a multiple of $|a|_{\Lambda^s}^2$ and a term that can be absorbed by
$|a|_{s}^2$.

The remaining case $2|\beta_1|\geq s$ is handled similarly.
\end{proof}

\begin{rem}\label{f3}
\textup{1) The spaces $\Gamma^s$, $\Gamma^s_\gamma$, $\Lambda^s$, and $\Lambda^s_\gamma$ all have
obvious analogues when $L^2(\R^{d+1}_+ \times\R)$ is replaced by $L^2(\Omega_T)$ in the definitions, where
we recall the notation
\begin{equation*}
\Omega_T=\big\{ (x,\theta) \in \R^{d+1}_+ \times\R: t<T \big\}.
\end{equation*}
The corresponding norms are denoted by adding the subscript $T$: $|a|_{s,T}$, $|a|_{s,\gamma,T}$,
$|a|_{\Lambda^s_T}$, etc. The equivalence \eqref{c4a} continues to hold for the norms restricted to $\Omega_T$,
as can be seen by a standard argument using Seeley extensions \cite{CP}.}

\textup{2) The analogous norms of functions of $b(x',\theta)$ defined on $\R^d\times \R$ or on $b\Omega_T$ are
denoted with brackets: $\langle b\rangle_{s,T}$, $\langle b\rangle_{s,\gamma,T}$, etc. We denote the corresponding
spaces by $b\Gamma^s_T$, $b\Lambda^s_T$, etc.}

\textup{3) By Sobolev embedding it follows that if the functions $\sigma_j$ appearing in \eqref{25} lie in $\Gamma^s_T$
for $s>\frac{d+2}{2}+3$, then $\cF$ as in \eqref{19} is of type $\cF$. Indeed, we then have, for example, $\theta^2_j \,
\partial_{\theta_j} \sigma_j \in H^t_T$ for an index $t>\frac{d+2}{2}$.}

\textup{4) More generally, if the functions $f^i_k$, $g^i_{l,m}$, $h^i_{l,m}$ appearing in \eqref{30a} lie in $\Gamma^s_T$
for $s>\frac{d+2}{2}+2$, then $F$ as in \eqref{30aa} is of type $\cF$.}
\end{rem}

We set $|a|_\infty :=|a|_{L^\infty(\Omega_{T})}$ when the domain $\Omega_T$ makes no possible confusion, and
we define
\begin{equation*}
W^{1,\infty}_T := \big\{
a(x,\theta) \, : \, |a|_{1,\infty}:=\sum_{|\alpha|\leq 1}|\partial^\alpha_{x,\theta}a|_{\infty}<\infty \big\} \, .
\end{equation*}

\textbf{Estimates for the coupled systems.} We can now state the main existence result for solutions
\begin{align}\label{c23}
\cV^{0,n+1}(x,\theta)=(\sigma_1^{n+1},\sigma_2^{n+1},\sigma_3^{n+1})
\end{align}
to the sequence of linear systems \eqref{f2}.

\begin{prop}\label{c24}
Let $T>0$, $m>\frac{d+2}{2}+1$ and suppose that $G \in b\Gamma^m_T$ and $\cV^{0,n}\in\Gamma^m_T$
both vanish in $t\leq 0$. Then the system \eqref{f2} has a unique solution $\cV^{0,n+1} \in \Gamma^m_T$
vanishing in $t\leq 0$ with $\sigma^{n+1}_2 \equiv 0$. Moreover, there exist increasing functions, $\gamma_0(K)$
and $C(K)$ of $K:=|\cV^{0,n}|_{m,T}$ such that for $\gamma \ge \gamma_0(K)$ we have
\begin{align}\label{c25}
|\cV^{0,n+1}|_{m,\gamma,T} +\dfrac{\langle \cV^{0,n+1} \rangle_{m,\gamma,T}}{\sqrt{\gamma}} \leq C(K) \,
\left( \dfrac{\langle G\rangle_{m,\gamma,T}}{\sqrt{\gamma}} +\dfrac{|\cV^{0,n}|_{m,\gamma,T}}{\gamma} \right).
\end{align}
\end{prop}

\begin{proof}
\textbf{1. $L^2$ estimate.} Anticipating the extra forcing terms that arise in the higher derivative estimates,
we first prove an $L^2$ a priori estimate in the case where a forcing term $f_{i}(x,\theta)$ vanishing in $t\leq 0$
is added to the right side of each interior equation in \eqref{f2}. Setting $F(x,\theta):=(f_1,f_2,f_3)$ we claim
\begin{align}\label{c26a}
|\cV^{0,n+1}|_{0,\gamma,T}+\dfrac{\langle \cV^{0,n+1} \rangle_{0,\gamma,T}}{\sqrt{\gamma}} \leq C(K') \,
\left( \dfrac{|F|_{0,\gamma,T}}{\gamma} +\dfrac{\langle G\rangle_{0,\gamma,T}}{\sqrt{\gamma}}
+\dfrac{|\cV^{0,n}|_{0,\gamma,T}}{\gamma} \right),
\end{align}
where $K' :=|\cV^{0,n}|_{1,\infty}$.
The latter estimate is obtained by considering the weighted function ${\rm e}^{-\gamma t} \, \cV^{0,n+1}$,
and by performing straightforward energy estimates on \eqref{f2}(a). The traces of $\sigma_i^{n+1}$, $i=1,3$,
are directly estimated by using \eqref{f2}(b).


\textbf{2. Higher order estimates.} We use again the system \eqref{f2} in its original form. Using the equivalence
of norms established in Proposition \ref{c4}, we first apply the $L^2$ estimate to the problems satisfied by
$\theta^k \, \sigma_i$, where $k\leq m$. The forcing term in this case is
\begin{align*}
-c^i_i \, [\sigma^n_i \, \partial_\theta,\theta^k] \, \sigma^{n+1}_i =-c^i_i \, k \, \sigma^n_i \, \theta^{k-1} \, \sigma^{n+1}_i.
\end{align*}
Clearly we may assume $|\theta|\geq 1$. Applying \eqref{c26a} we can absorb the terms on the right involving
$\sigma_i^{n+1}$ to obtain
\begin{align}\label{f5}
|\theta^k \, \cV^{0,n+1}|_{0,\gamma,T} +\dfrac{\langle \theta^k \, \cV^{0,n+1} \rangle_{0,\gamma,T}}{\sqrt{\gamma}}
\leq C(K') \, \left( \dfrac{\langle G \rangle_{m,\gamma,T}}{\sqrt{\gamma}}
+\dfrac{|\cV^{0,n}|_{m,\gamma,T}}{\gamma} \right) \, \text{ for } \gamma \geq \gamma_0(K').
\end{align}

Next we estimate $\partial_{x'}^\alpha \sigma^{n+1}_i$ for $|\alpha|\leq m$. The forcing term in the problem
satisfied by $\partial_{x'}^\alpha\sigma^{n+1}_i$ is now
\begin{align*}
-c^i_i \, [\sigma^n_i \, \partial_\theta,\partial_{x'}^\alpha] \, \sigma^{n+1}_i.
\end{align*}
The commutator is a finite linear combination of terms of the form
\begin{align}\label{f7}
(\partial_{x'}^{\alpha_1}\sigma^n_i) \, (\partial^{\alpha_2}_{x'}\partial_\theta\sigma^{n+1}_i), \; \;
|\alpha_1|+|\alpha_2|=|\alpha|, \;\;|\alpha_1|\geq 1.
\end{align}
We estimate these terms using the following two observations:

\textbf{A}. Suppose $m_1+m_2>\frac{d+2}{2}$, $m_i\geq 0$. Then the product $(a(x,\theta),b(x,\theta)) \to a\cdot b$
is continuous from $H^{m_1}_T\times H^{m_2}_T\to H^0_T$.

\textbf{B}. Suppose $|\alpha_1|+|\alpha_2|\leq |\alpha|\leq m$. Then
\begin{align*}
|{\rm e}^{-\gamma t} \, \partial^{\alpha_1}_{x,\theta}u(x,\theta)|_{|\alpha_2|,T} \leq |u|_{m,\gamma,T}.
\end{align*}
By \textbf{A} (with $m_1=m-|\alpha_1|$, $m_2=|\alpha_1|-1$) and \textbf{B} we have
\begin{align*}
|\eqref{f7}|_{0,\gamma,T}\leq C \, |\sigma_i^n|_{m,T} \, |\sigma_{i}^{n+1}|_{m,\gamma,T} \leq C \, K \,
|\sigma_{i}^{n+1}|_{m,\gamma,T}.
\end{align*}
Applying \eqref{c26a} and absorbing terms from the right, we obtain an estimate like \eqref{f5} for
$\partial_{x'}^\alpha\cV^{0,n+1}$ with $C(K')$ replaced by $C(K)$.

The $\theta$ derivatives $\partial^k_\theta \sigma_i$, $k\leq m$, are estimated similarly. Derivatives involving
$\partial_{x_d}$ are estimated in the customary way using the tangential estimates and the fact that $x_d=0$
is noncharacteristic for $X_{\phi_i}$.

\textbf{3. Existence and uniqueness}. This follows easily from the above estimates since the principal part of
the system \eqref{f2} is given by three decoupled vector fields. One can therefore integrate along characteristics.
Equations \eqref{f2}(a),(c) and the fact that $X_{\phi_2}$ is outgoing imply $\sigma^{n+1}_2=0$.
\end{proof}

Next we show convergence of the iterates $\cV^{0,n}$ to a short time solution of the nonlinear profile equations
\eqref{f1}.

\begin{prop}\label{c37}
Consider the profile equations \eqref{f1}, where $G\in b\Gamma^m_T$, $m>\frac{d+2}{2}+1$, and vanishes in
$t\leq 0$. For some $0<T_0\leq T$ the system has a unique solution $\cV^0\in \Gamma^m_{T_0}$ with $\sigma_2=0$.
\end{prop}

\begin{proof}
\textbf{1.} The iteration scheme \eqref{f2} defines a sequence $(\cV^{0,n})$ in $\Gamma^m_T$. Fixing $K>0$
we claim that for $T^*>0$ small enough,
\begin{align}\label{c38}
|\cV^{0,n}|_{m,T^*} +\langle\cV^{0,n}\rangle_{m,T^*} <K \text{ for all } n.
\end{align}
Indeed, first observe that
\begin{align*}
|u|_{m,\gamma,T} \leq C_1 \, |u|_{m,T} \leq C_2 \, {\rm e}^{\gamma T} \, |u|_{m,\gamma,T},
\end{align*}
and fix $\gamma>\gamma_0(K)$ such that $\sqrt{\gamma} \ge 2\, C(K) \, C_1$ for $\gamma_0(K)$ and $C(K)$
as in Proposition \ref{c24}. Assuming \eqref{c38} holds for $n\leq n_0$, we find that it holds for $n_0+1$ after
shrinking $T^*$ if necessary, using the estimate \eqref{c25} and the fact that $G$ vanishes in $t\leq 0$. This
new choice of $T^*$ works for all $n$.

\textbf{2.} Convergence of the iterates in $\Gamma^0_{T_0}$  to some $\cV^0$ for a possibly smaller $T_0>0$
now follows from \eqref{c38} by applying \eqref{c25} when $m=0$ to the problem satisfied by $(\cV^{0,n+1}-\cV^{0,n})$.
In view of \eqref{c38} and a classical argument involving weak convergence and interpolation, we thereby obtain
a solution $\cV^0\in \Gamma^m_{T_0}$ with, in fact, a trace that lies in $b\Gamma^m_{T_0}$. This argument
shows that the iterates $\cV^{0,n}$ converge to $\cV^0$ in $\Gamma^{m-1}_{T_0}$.
\end{proof}

\section{Error analysis}\label{simult}

\emph{\quad} Next we carry out the error analysis sketched in section \ref{errorintro}. In section \ref{mz}
we define and derive estimates for moment-zero approximations. In section \ref{estii} we estimate interaction
integrals involving both transversal and nontransversal interactions of pulses; these estimates are used later
to estimate the first corrector $\cU^1_{p,\eps}$. Finally, in section \ref{proofmain} we complete the proof of
Theorem \ref{a16} by proving the stronger result, Theorem \ref{e1}.

\subsection{Moment-zero approximations to $\cU^0$}\label{mz}

\quad  When constructing a corrector to the leading term in the approximate solution we must take primitives
in $\theta$ of functions $\sigma(x,\theta)$ that decay to zero as $|\theta|\to\infty$. A difficulty is that such
primitives do not necessarily decay to zero as $|\theta|\to \infty$, and this prevents us from using those
primitives directly in the error analysis. The failure of the primitive to decay manifests itself on the Fourier
transform side as a small divisor problem. To get around this difficulty we work with the primitive of a
\emph{moment-zero approximation} to $\sigma$, because such a primitive does have the desired decay.

We will use  the following spaces:

\begin{defn}\label{d22}
1.)  For $s\geq 0$, we recall the notation \eqref{a5}, that is $E^s_T:= \{ U \in C(x_d,H^s_T(x',\theta_0))
\cap L^2(x_d,H^{s+1}_T(x',\theta_0))\}$. This space is equipped with the norm
\begin{equation*}
|U(x,\theta_0)|_{E^s_T}:=|U|_{\infty,s,T}+|U|_{0,s+1,T}.
\end{equation*}

2.)  Let $\cE^s_T:=\{\cU(x,\theta_0,\xi_d):|\cU|_{\cE^s_T}:=\sup_{\xi_d\geq 0} \, |\cU(\cdot,\cdot,\xi_d)|_{E^s_T} < \infty\}$.
\end{defn}


\begin{prop}\label{d23a}
For $s>(d+1)/2$ the spaces $E^s_T$ and $\cE^s_T$ are Banach algebras.

\end{prop}

\begin{proof}
This is a consequence of the Sobolev embedding Theorem and the fact that $L^\infty(b\Omega_T)\cap
H^s(b\Omega_T)$ is a Banach algebra for $s\geq 0$.
\end{proof}

The proofs of the following two propositions follow directly from the definitions.

\begin{prop}\label{h1}
(a)\; For $s\geq 0$, let $\sigma(x,\theta)\in E^s_T$ 
and set $\tilde \sigma(x,\theta_0,\xi_d) :=\sigma(x,\theta_0+\omega\, \xi_d)$, $\omega\in\R$. Then $\tilde \sigma
\in \cE^s_T$ and
\begin{align*}
|\tilde\sigma|_{\cE^s_T}\leq C|\sigma|_{H^{s+1}_T}.
\end{align*}

(b)\; For $\tilde \sigma \in \cE^s_T$, set $\tilde \sigma_\eps(x,\theta_0) :=\tilde\sigma(x,\theta_0,\frac{x_d}{\eps})$. Then
\begin{align*}
|\tilde\sigma_\eps|_{E^s_T}\leq |\tilde\sigma|_{\cE^s_T}.
\end{align*}
\end{prop}

\begin{defn}[Moment-zero approximations]\label{h5}
Let $0<p<1$, and let $\phi\in C^\infty(\R)$ have $\mathrm{supp}\;\phi\subset \{m:|m|\leq 2\}$ with $\phi=1$ on $\{
|m| \leq 1 \}$. Set $\phi_p(m):=\phi(\frac{m}{p})$ and $\chi_p:=1-\phi_p$. For $\sigma(x,\theta)\in L^2(\Omega_T)$,
define the \emph{moment zero approximation} to $\sigma$, $\sigma_p(x,\theta)$ by
\begin{align}\label{h5a}
\hat\sigma_p(x,m):=\chi_p(m) \, \hat\sigma(x,m),
\end{align}
where the hat denotes the Fourier transform in $\theta$.
\end{defn}

\begin{prop}\label{h6}
For $s\geq 1$ suppose $\sigma(x,\theta)\in \Gamma^{s+2}_T$, and define $\tilde \sigma(x,\theta_0,\xi_d)
:=\sigma(x,\theta_0+\omega\xi_d)$. Then
\begin{align*}
\begin{split}
&a)\; |\tilde\sigma-\tilde \sigma_p|_{\cE^s_T} \leq C \, |\sigma|_{\Gamma^{s+2}_T} \, \sqrt{p},\\
&b)\; |\partial_{x_d}\tilde\sigma-\partial_{x_d}\tilde \sigma_p|_{\cE^{s-1}_T} \leq C \, |\sigma|_{\Gamma^{s+2}_T}
\, \sqrt{p}.
\end{split}
\end{align*}
\end{prop}

\begin{proof}
\textbf{1. } Recall that $\sigma \in \Gamma^s_T \Leftrightarrow \theta^{\beta_1} \, \partial_x^{\beta_2} \,
\partial_\theta^{\beta_3} \sigma(x,\theta) \in L^2(x,\theta)$ for $|\beta|\leq s$, which is also equivalent to
$\partial_m^{\beta_1} \, \partial_x^{\beta_2} \, m^{\beta_3} \, \hat\sigma(x,m) \in L^2(x,m)$ for $|\beta|\leq s$.
It follows that
\begin{align}\label{h7}
\sigma\in\Gamma^{s+2}_T \Rightarrow \hat \sigma(x,m)\in H^{s+2}_T(x,m) \subset H^1(m,H^{s+1}(x))
\subset L^\infty(m,H^{s+1}(x)).
\end{align}


\textbf{2. } We have
\begin{align*}
|\sigma-\sigma_p|^2_{H^{s+1}_T} &\sim \sum_{|\alpha|+k\leq s+1} |\partial_x^\alpha \, m^k \, \hat\sigma(x,m)
\, (1-\chi_p(m))|^2_{L^2(x,m)} \\
&=\sum_{|\alpha|+k\leq s+1} \int_{|m|\leq 2p} \int |\partial_x^\alpha \, m^k \, \hat\sigma(x,m) \, \phi_p(m)|^2 \,
{\rm d}x \, {\rm d}m \\
&\leq C\, \int_{|m|\leq 2p} |\hat\sigma(x,m)|^2_{H^{s+1}(x)} \, {\rm d}m \leq  C\, |\sigma|^2_{\Gamma^{s+2}_T} \, (2p),
\end{align*}
where the last inequality uses \eqref{h7}.  The conclusion now follows from Propostion \ref{h1}.

\textbf{3. } The proof of inequality $b)$ in Proposition \ref{h6} is essentially the same.
\end{proof}

\begin{prop}\label{h8}
Let $\sigma(x,\theta)\in H^s_T$, $s\geq 0$, and let $\sigma_p$ be a moment-zero approximation to $\sigma$.
We have
\begin{align*}
\begin{split}
&(a)\; |\sigma_p|_{H^s_T} \leq C \, |\sigma|_{H^s_T} \, ,\\
&(b)\; \text{ If } \sigma \in \Gamma^s_T, \text{ then } |\sigma_p|_{\Gamma^s_T} \leq \dfrac{C}{p^s} \,
|\sigma|_{\Gamma^s_T}.
\end{split}
\end{align*}
\end{prop}

\begin{proof}
Part $b)$ follows from \eqref{h5a}. Indeed, for $|\beta|\leq s$,
\begin{align*}
|\partial_m^{\beta_1} \, \partial_x^{\beta_2} \, m^{\beta_3} \, \hat\sigma_p(x,m)|_{L^2_T} \leq \dfrac{C}{p^{\beta_1}}
\, |\sigma|_{\Gamma^s_T} \, ,
\end{align*}
since $|\partial_m^{\beta_1} \chi_p|\leq C/p^{\beta_1}$. Taking $\beta_1=0$ we similarly obtain part $a)$.
\end{proof}

Next we consider primitives of moment-zero approximations.

\begin{prop}\label{h9}
Let $\sigma(x,\theta)\in\Gamma^s_T$, $s>\frac{d}{2}+3$. Let $\sigma_p^*(x,\theta)$ be the unique primitive of
$\sigma_p$ in $\theta$ that decays to zero as $|\theta|\to \infty$. Then $\sigma^*_p\in\Gamma^s_T$ with moment
zero, and
\begin{align}\label{h10}
\begin{split}
&(a)\; |\sigma^*_p|_{H^s_T} \leq C \, \dfrac{|\sigma_p|_{H^s_T}}{p} \, ,\\
&(b)\; |\sigma^*_p|_{\Gamma^s_T} \leq C \, \dfrac{|\sigma_p|_{\Gamma^s_T}}{p^{s+1}} \, .
\end{split}
\end{align}
\end{prop}

\begin{proof}
\textbf{1. } Since $\sigma_p(x,\theta)\in\Gamma^s_T$, $s>\frac{d}{2}+3$, we have $|\sigma_p(x,\theta)| \leq C \,
\langle \theta \rangle^{-2}$ for all $(x,\theta)$. The unique $\theta$-primitive of $\sigma_p$ decaying to zero as
$|\theta|\to \infty$ is thus
\begin{align*}
\sigma^*_p(x,\theta) =-\int^\infty_\theta \sigma_p(x,s) \, {\rm d}s =\int^\theta_{-\infty} \sigma_p(x,s) \, {\rm d}s.
\end{align*}
Moreover, we have
\begin{align}\label{h11}
\partial_\theta\sigma^*_p =\sigma_p \Rightarrow im \, \widehat{\sigma^*_p} =\widehat{\sigma_p} =\chi_p \,
\hat\sigma, \quad \text{ so }\widehat{\sigma^*_p}=\dfrac{\chi_p}{im} \, \hat\sigma.
\end{align}
Since $|m|\geq p$ on the support of $\chi_p$, this gives
\begin{align}
|\widehat{\sigma^*_p}(x,m)| \leq C \, \dfrac{|\hat\sigma(x,m)|}{p}
\end{align}
and \eqref{h10}(a) follows directly from this. From \eqref{h11} we also obtain $\widehat{\sigma^*_p}(x,0)=0$.

\textbf{2. } The proof of \eqref{h10} $(b)$ is almost the same, except that now one uses
\begin{align*}
\left|\partial_m^s\left( \dfrac{\chi_p}{m} \right) \right| \leq \dfrac{C}{p^{s+1}}.
\end{align*}
\end{proof}

\begin{prop}\label{h11a}
Let $\sigma(x,\theta)$ and $\tau(x,\theta)$ belong to $H^s_T$, $s>\frac{d+2}{2}$. Then
\begin{align}\label{h11b}
|\sigma \, \tau -(\sigma \, \tau)_p|_{H^s_T}\leq C \, |\sigma|_{H^s_T} \, |\tau|_{H^s_T} \, \sqrt{p}.
\end{align}
\end{prop}

\begin{proof}
With $*$ denoting convolution in $m$ we have
\begin{align*}
|\sigma \, \tau -(\sigma \, \tau)_p|^2_{H^s_T} &\sim \sum_{|\alpha|+k\leq s+1} |\partial_x^\alpha \, m^k \,
(\hat\sigma *\hat\tau)(x,m) \, (1-\chi_p(m))|^2_{L^2(x,m)} \\
&\le C\, \int_{|m|\leq 2p} |(\hat\sigma *\hat\tau)(x,m)|^2_{H^s(x)} \, {\rm d}m \\
&\leq C \, \int_{|m|\leq 2p}
\left( \int |\hat\sigma(x,m-m_1)|_{H^s(x)} \, |\hat\tau(x,m_1)|_{H^s(x)} \, {\rm d}m_1\right)^2 \, {\rm d}m \\
&\le C\, p \, |\hat\sigma(x,m)|^2_{L^2(m,H^s(x))} \, |\hat\tau(x,m)|^2_{L^2(m,H^s(x))} \leq
C\, p \, |\sigma|^2_{H^s_T} \, |\tau|^2_{H^s_T}.
\end{align*}
\end{proof}

\begin{prop}\label{h11c}
Let $\sigma(x,\theta)$ and $\tau(x,\theta)$ belong to $\Gamma^s_T$, $s>\frac{d}{2}+3$ and let $(\sigma\tau)_p^*$
denote the unique primitive of $(\sigma\tau)_p$ that decays to zero as $|\theta|\to \infty$.  Then
\begin{align*}
|(\sigma \, \tau)_p^*|_{H^s_T} \leq C\, \dfrac{|\sigma|_{H^s_T} \, |\tau|_{H^s_T}}{p}.
\end{align*}
\end{prop}

\begin{proof}
Since $\Gamma^s_T$ is a Banach algebra, Proposition \ref{h9} implies $(\sigma\tau)_p^*\in\Gamma^s_T$ with
moment zero and
\begin{align*}
|(\sigma \, \tau)_p^*|_{H^s_T} \leq C \, \frac{|(\sigma \, \tau)_p|_{H^s_T}}{p}.
\end{align*}
Since $H^s_T$ is a Banach algebra, the result now follows from Proposition \ref{h8}(a).
\end{proof}

\subsection{Estimates of interaction integrals}\label{estii}

\emph{\quad}Pulses do not interact to produce resonances that affect the leading order profiles as in the wavetrain
case. However, interaction integrals must be estimated carefully in order to do the error analysis.

The following propositions will be used in the error analysis for estimating terms related to $\cU^1$ as in \eqref{dd5},
where the $\cF_i$ appearing there are given by \eqref{26}; in particular, we must estimate primitives of products of
pulses. In some of the estimates below we must introduce moment-zero approximations to avoid errors that are too
large to be useful  in the error analysis. We begin with an estimate of ``transversal interactions".

\begin{prop}\label{h12}
Let $t$ be the smallest integer greater than $\frac{d}{2}+3$ and let $s\geq 0$.  Let $\sigma_1(x,\theta)$,
$\sigma_2(x,\theta)$ belong to $\Gamma^t_T\cap H^{s+1}_T$ and define
\begin{align}\label{h12a}
u(x,\theta_0,\xi_d) :=\int_\infty^{\xi_d}\sigma_1(x,\theta_0+\omega \, \xi_d+\alpha \, s) \,
\sigma_2(x,\theta_0+\omega \, \xi_d+s) \, {\rm d}s,
\end{align}
where $\omega$, $\alpha$ are real and $\alpha\notin\{0,1\}$. With $u_\eps(x,\theta_0):=u(x,\theta_0,\frac{x_d}{\eps})$
we have
\begin{align*}
|u_\eps|_{E^s_T}\leq C \, (|\sigma_1|_{H^{s+1}_T} \, |\sigma_2|_{\Gamma^t_T}
+|\sigma_2|_{H^{s+1}_T} \, |\sigma_1|_{\Gamma^t_T}).
\end{align*}
uniformly for $\eps\in (0,1]$.
\end{prop}

\begin{proof}
\textbf{1. } For fixed $x_d$ and $\eps$ we first estimate
\begin{multline}\label{h12b}
\left| \int^{x_d/\eps}_\infty \sigma_1 \, \sigma_2 \, {\rm d}s \right|_{H^s_T(x',\theta_0)} \sim \sum_{|\alpha|\leq s}
\left| \int^{x_d/\eps}_\infty \partial_{x'}^\alpha \, (\sigma_1 \, \sigma_2) \, {\rm d}s \right|_{L^2(x',\theta_0)}
+\sum_{k\leq s} \left| \int^{x_d/\eps}_\infty \partial_{\theta_0}^k \, (\sigma_1 \, \sigma_2) \, {\rm d}s \right|_{L^2(x',\theta_0)}\\
:=A_s(x_d)+B_s(x_d).
\end{multline}
Here and below $\sigma_1$, $\sigma_2$ and their derivatives are evaluated at the points indicated in \eqref{h12a}
with $\xi_d=\frac{x_d}{\eps}$, unless explicitly stated otherwise.

\textbf{2. } To estimate $A_s(x_d)$ we consider for $|\alpha_1|+|\alpha_2|=|\alpha|$:
\begin{multline*}
\left| \int^{x_d/\eps}_\infty \partial_{x'}^{\alpha_1}\sigma_1 \, \partial_{x'}^{\alpha_2}\sigma_2 \, {\rm d}s \right|_{L^2(x',\theta_0)}
\leq \left| \int^{x_d/\eps}_\infty \left| \partial_{x'}^{\alpha_1}\sigma_1 \, \partial_{x'}^{\alpha_2}\sigma_2 \right|_{L^2(x')}
\, {\rm d}s \right|_{L^2(\theta_0)} \\
\le \left| \int^\infty_{-\infty} (|\sigma_1|_{L^\infty(x')} \, |\sigma_2|_{H^s(x')} +|\sigma_1|_{H^s(x')} \,
|\sigma_2|_{L^\infty(x')}) \, {\rm d}s \right|_{L^2(\theta_0)} \leq A_{1,s}(x_d)+A_{2,s}(x_d),
\end{multline*}
where we have used a Moser estimate in the $x'$ variable. Setting $z=\theta_0+\omega \, \frac{x_d}{\eps}+\alpha \, s$,
we obtain
\begin{align}
A_{1,s}(x_d) &= C \, \left| \int^\infty_{-\infty}|\sigma_1(x,z)|_{L^\infty(x')} \, \left| \sigma_2 \left( x,
(\theta_0+\omega\frac{x_d}{\eps})(1-\frac{1}{\alpha})+\frac{z}{\alpha} \right) \right|_{H^s(x')} \, {\rm d}z \right|_{L^2(\theta_0)}
\notag \\
&\le C\, \int^\infty_{-\infty}|\sigma_1(x,z)|_{L^\infty(x')} \, |\sigma_2|_{L^2(\theta,H^s(x'))} \, {\rm d}z \notag \\
&=C\, |\sigma_2|_{L^2(\theta,H^s(x'))} \, \int^\infty_{-\infty} |\sigma_1(x,z)|_{L^\infty(x')} \, \langle z \rangle^2
\dfrac{{\rm d}z}{\langle z\rangle^{-2}} \notag \\
&\leq C\, |\sigma_2|_{L^2(\theta,H^s(x'))} \, |\sigma_1(x,z) \langle z\rangle^2|_{L^\infty(x,z)} \, \leq C \,
|\sigma_2|_{H^{s}_T(x',\theta)} \, |\sigma_1|_{\Gamma^t_T},\label{h13a}
\end{align}
where the last inequality uses Remark \ref{f3}. The estimate of $A_{2,s}(x_d)$ is similar.

\textbf{3. } Recalling the definition of the $E^s_T$ norm and using
\begin{align*}
|\sigma_2|_{C(x_d,H^{s}_T(x',\theta))} \leq C \, |\sigma_2|_{H^1(x_d,H^s_T(x',\theta))}
\leq C \, |\sigma_2|_{H^{s+1}_T(x,\theta)} \, ,
\end{align*}
we obtain from \eqref{h13a}:
\begin{align}\label{h13}
|A_{1,s}(x_d) \, A_{2,s}(x_d)|_{C(x_d)} +|A_{1,s+1}(x_d) \, A_{2,s+1}(x_d)|_{L^2(x_d)} \leq C \,
(|\sigma_1|_{H^{s+1}_T} \, |\sigma_2|_{\Gamma^t_T} +|\sigma_2|_{H^{s+1}_T} \, |\sigma_1|_{\Gamma^t_T}).
\end{align}

\textbf{4. } To estimate $B_s(x_d)$ in \eqref{h12b} we consider for $k_1+k_2=k$:
\begin{align*}
\int^{x_d/\eps}_\infty \partial_{\theta_0}^{k_1}\sigma_1 \, \partial_{\theta_0}^{k_2}\sigma_2 \, {\rm d}s
=\pm \int^{x_d/\eps}_\infty (\partial_{\theta_0}^{k}\sigma_1) \, \sigma_2 ds + \mathrm{ (boundary \;terms) }.
\end{align*}
Each boundary term has the form $\partial_{\theta_0}^{m_1}\sigma_1 \, \partial_{\theta_0}^{m_2}\sigma_2$,
$m_1+m_2<k$, where $s$ is evaluated at $x_d/\eps$. We estimate such terms using Moser estimates as follows:
\begin{align}\label{h14a}
|\partial_{\theta_0}^{m_1}\sigma_1 \, \partial_{\theta_0}^{m_2}\sigma_2|_{L^2(x',\theta_0)}
\leq C \, (|\sigma_1|_{L^\infty(x,\theta)} \, |\sigma_2|_{H^{s-1}(x',\theta_0)} +|\sigma_1|_{H^{s-1}(x',\theta_0)}
\, |\sigma_2|_{L^\infty(x,\theta_0)}).
\end{align}
For the integral term setting $z=\theta_0+\omega\frac{x_d}{\eps}+s$, we have
\begin{align}\label{h14}
\begin{split}
&\left| \int^{x_d/\eps}_\infty (\partial_{\theta_0}^{k}\sigma_1) \, \sigma_2 \, {\rm d}s \right|_{L^2(x',\theta_0)}
\leq C \, \left| \int^\infty_{-\infty} |\partial_{\theta_0}^k \sigma_1 \left(x,(\theta_0+\omega\frac{x_d}{\eps})(1-\alpha)
+\alpha z\right) \, \sigma_2(x,z) |_{L^2(x')} \, {\rm d}z \right|_{L^2(\theta_0)} \\
&\qquad \leq C \, \left| \int^\infty_{-\infty} |\partial_{\theta_0}^k \sigma_1 \left(x,(\theta_0+\omega\frac{x_d}{\eps})(1-\alpha)
+\alpha z\right)|_{L^2(x')} \, |\sigma_2(x,z)|_{L^\infty(x)} \, {\rm d}z \right|_{L^2(\theta_0)} \\
&\qquad \leq C \, |\sigma_1|_{L^2(x',H^s(\theta))} \, |\sigma_2(x,z)\langle z\rangle^2|_{L^\infty(x,z)}
\leq C \, |\sigma_1|_{H^{s}_T(x',\theta)} \, |\sigma_2|_{\Gamma^t_T}.
\end{split}
\end{align}
From \eqref{h14a} and \eqref{h14}, we obtain parallel to \eqref{h13}:
\begin{align*}
|B_{s}(x_d)|_{C(x_d)} +|B_{s+1}(x_d)|_{L^2(x_d)} \leq C \, (|\sigma_1|_{H^{s+1}_T} \, |\sigma_2|_{\Gamma^t_T}
+|\sigma_2|_{H^{s+1}_T} \, |\sigma_1|_{\Gamma^t_T}),
\end{align*}
completing the proof.
\end{proof}

The previous estimate of transversal interactions did not require the use of moment-zero approximations.
However, nontransversal interactions of pulses can produce errors that are too big to be helpful in the error
analysis. Thus, we are forced to use a moment-zero approximation in the next proposition.

\begin{prop}\label{h16}
Let $\sigma(x,\theta)$ and $\tau(x,\theta)$ belong to $\Gamma^s_T$, $s>\frac{d}{2}+3$. For $\alpha,\omega
\in \R$, $\alpha\neq 0$ set
\begin{align*}
f(x,\theta_0,\xi_d) :=\int^{\xi_d}_\infty (\sigma \, \tau)_p(x,\theta_0+\omega \, \xi_d+\alpha \, s) \, {\rm d}s.
\end{align*}
Then
\begin{equation*}
\left| f(x,\theta_0,\frac{x_d}{\eps}) \right|_{E^{s-1}_T} \leq C \, \dfrac{|\sigma|_{H^s_T} \, |\tau|_{H^s_T}}{p} \, .
\end{equation*}
\end{prop}

\begin{proof}
The integral equals $\alpha^{-1} \, (\sigma \, \tau)_p^*(x,\theta_0+\xi_d(\omega+\alpha))$ so the estimate
follows first by applying Proposition \ref{h11c} and then by applying Proposition \ref{h1}.
\end{proof}

\begin{cor}\label{h17}
Let $\sigma(x,\theta)$, $\tau(x,\theta)$, and $\omega,\alpha$ be as in Proposition \ref{h16} and set
\begin{align*}
g(x,\theta_0,\xi_d) :=\int^{\xi_d}_\infty (\sigma_p \, \tau_p)_p(x,\theta_0+\omega \, \xi_d+\alpha \, s) \, {\rm d}s.
\end{align*}
Then
\begin{equation*}
\left| g(x,\theta_0,\frac{x_d}{\eps}) \right|_{E^{s-1}_T} \leq C \, \dfrac{|\sigma|_{H^s_T} \, |\tau|_{H^s_T}}{p} \, .
\end{equation*}
\end{cor}

\begin{proof}
First apply Proposition \ref{h16} and then Proposition \ref{h8}(a).
\end{proof}

\begin{prop}\label{h18}
For $s>\frac{d}{2}+3$ let $\sigma(x,\theta)\in H^s_T$, $\tau(x,\theta)\in\Gamma^{s+1}_T$. With $\omega,\alpha
\in \R$, $\alpha\neq 0$ set
\begin{align*}
h(x,\theta_0,\xi_d) :=\sigma(x,\theta_0+\omega \, \xi_d) \, \int^{\xi_d}_\infty \partial_{\theta_0}\tau (x,\theta_0
+\omega \, \xi_d+\alpha \, s) \, {\rm d}s.
\end{align*}
Then
\begin{equation*}
\left| h(x,\theta_0,\frac{x_d}{\eps}) \right|_{E^{s-1}_T} \leq C \, |\sigma|_{H^s_T} \, |\tau|_{H^s_T} \, .
\end{equation*}
\end{prop}

\begin{proof}
The integral is equal to $\alpha^{-1} \, \tau(x,\theta_0+\xi_d(\omega+\alpha))$ so the estimate follows from the
fact that $E^{s-1}_T$ is a Banach algebra together with Proposition \ref{h1}.
\end{proof}

In the next Proposition we must use a moment-zero approximation since $\tau(x,\theta)$ may not have moment zero.

\begin{prop}\label{h19}
For $s> \frac{d}{2}+3$ let $\sigma(x,\theta)\in H^s_T$, $\tau(x,\theta)\in\Gamma^s_T$. With $\omega,\alpha \in \R$,
$\alpha\neq 0$ set
\begin{align*}
j(x,\theta_0,\xi_d) :=\partial_{\theta_0} \sigma (x,\theta_0+\omega \, \xi_d) \, \int^{\xi_d}_\infty \tau_p
(x,\theta_0+\omega \, \xi_d+\alpha \, s) \, {\rm d}s.
\end{align*}
Then
\begin{equation*}
\left| j(x,\theta_0,\frac{x_d}{\eps}) \right|_{E^{s-2}_T} \leq C \, \dfrac{|\sigma|_{H^s_T} \, |\tau|_{H^{s-1}_T}}{p} \, .
\end{equation*}
\end{prop}

\begin{proof}
The integral is equal to $\alpha^{-1} \, \tau_p^*(x,\theta_0+\xi_d(\omega+\alpha))$. The estimate follows by the
argument of Proposition \ref{h18}, except that now we also need Proposition \ref{h9}(a) and Proposition \ref{h8}(a).
\end{proof}

The proof of the next Proposition is evident from the proof of Proposition \ref{h19}.

\begin{prop}\label{h20}
For $s>\frac{d}{2}+3$ and $\omega,\alpha\in\R$, $\alpha\neq 0$, let $\sigma\in\Gamma^s_T$ and set
\begin{align*}
k(x,\theta_0,\xi_d)=\int^{\xi_d}_\infty \sigma_p(x,\theta_0+\omega\xi_d+\alpha s)ds.
\end{align*}
Then
\begin{equation*}
\left| k(x,\theta_0,\frac{x_d}{\eps}) \right|_{E^{s-1}_T} \leq C \, \dfrac{|\sigma|_{H^s_T}}{p} \, .
\end{equation*}
\end{prop}

\subsection{Proof of Theorem \ref{a16}}\label{proofmain}


\emph{\quad} Now we are ready to prove Theorem \ref{a16}, which shows that the approximate solution
$u^a_\eps(x)$ converges in $L^\infty$ to the exact solution $u_\eps$ of Theorem \ref{a12} as $\eps\to 0$. In
this section we prove the following more precise Theorem, which implies Theorem \ref{a16} as an immediate
corollary. As before we focus on the $3\times 3$ strictly hyperbolic case to ease the exposition. The mostly
minor changes needed to treat $N\times N$ systems satisfying Assumptions \ref{assumption1}, \ref{assumption2},
and \ref{a7} are described in section \ref{extension}.

\begin{theo}\label{e1}
For $M_0=3d+5$ and $s\geq 1+[M_0+\frac{d+1}{2}]$, let $G(x',\theta_0)\in b\Gamma^{s+1}_T$ and suppose
$G=0$ in $t\leq 0$. Let $U_\eps(x,\theta_0)\in E^{s}_{T_0}$ be the exact solution to the singular system \eqref{a4b}
for $0<\eps\leq \eps_0$ given by Theorem \ref{a12}, let $\cV^0=(\sigma_1,\sigma_2,\sigma_3)\in \Gamma^{s+1}_{T_0}$
be the profile given by Proposition \ref{c37}, and let $\cU^0\in\cE^{s}_{T_0}$ be defined by
\begin{align*}
\cU^0(x,\theta_0,\xi_d) :=\sum_{j=1}^3 \sigma_j(x,\theta_0+\omega_j \, \xi_d) \, r_j.
\end{align*}
Here $0 < T_0\leq T$ is the minimum of the existence times for the quasilinear problems \eqref{a4b} and \eqref{36a}.
Define
\begin{align*}
\cU^0_\eps(x,\theta_0) := \cU^0(x,\theta_0,\frac{x_d}{\eps}).
\end{align*}
The family $\cU^0_\eps$ is uniformly bounded in $E^{s}_{T_0}$ for $0<\eps\leq \eps_0$; moreover, there exists
$0<T_1\leq T_0$ and $C>0$ such that
\begin{align}\label{e4}
|U_\eps-\cU^0_\eps|_{E^{s-3}_{T_1}}\leq C\eps^{\frac{1}{2M_1+5}},
\end{align}
where $M_1$ is the smallest integer $>\frac{d}{2}+3$.
\end{theo}

The proof of Theorem \ref{e1} will use the strategy of \emph{simultaneous Picard iteration} first used by Joly, M\'etivier,
and Rauch in \cite{jmr} to justify leading term expansions for initial value problems on domains without boundary.
Consider the iteration schemes for the quasilinear problems \eqref{a4b} and \eqref{36a}:
\begin{align}\label{e5}
\begin{split}
&a)\; \partial_{x_d}U^{n+1}_\eps+\sum^{d-1}_{j=0}\tilde{A}_j(\eps U^n_\eps) \left(
\partial_{x_j}+\frac{\beta_j \partial_{\theta_0}}{\eps}\right)U^{n+1}_\eps =F(\eps U^n_\eps)U^n_\eps,\\
&b)\;B(\eps U^{n}_\eps) \, U^{n+1}_\eps|_{x_d=0}=G(x',\theta_0),\\
&c)\;U^{n+1}_\eps=0 \text{ in } t<0,
\end{split}
\end{align}
and
\begin{align}\label{e6}
\begin{split}
&a)\; {\bf E} \, \cU^{0,n+1}=\cU^{0,n+1}\\
&b)\; {\bf E} \left(\tilde{L}(\partial)\cU^{0,n+1}+M(\cU^{0,n},\partial_{\theta_0}\cU^{0,n+1})\right) ={\bf E} \, (F(0)\cU^{0,n})\\
&c)\; B(0) \, \cU^{0,n+1}|_{x_d=0,\xi_d=0}=G(x',\theta_0)\\
&d)\; \cU^{0,n+1}=0\text{ in }t<0,
\end{split}
\end{align}
where $\cU^{0,n}(x,\theta_0,\xi_d):=\sum_{j=1}^3\sigma^n_j(x,\theta_0+\omega_j \, \xi_d) \, r_j$ for $\sigma^n_j$
as constructed in Proposition \ref{c24}. Setting
\begin{equation*}
\cU^{0,n}_{\eps}(x,\theta_0):=\cU^{0,n}(x,\theta_0,\frac{x_d}{\eps}),
\end{equation*}
we observe that to prove the theorem it suffices to prove boundedness of the family $\cU^0_{\eps}$ in $E^{s}_{T_0}$
along with the following three statements:
\begin{align}\label{e8}
\begin{split}
&(a)\;\lim_{n\to\infty}U^n_\eps= U_\eps\text{ in }   E^{s-1}_{T_0} \text{ uniformly  with respect to }\eps\in (0,\eps_0]\\
&(b) \;\lim_{n\to\infty}\cU^{0,n}_{\eps}= \cU^0_{\eps}\text{ in }   E^{s-1}_{T_0} \text{  uniformly  with respect to }\eps\in (0,\eps_0]\\
&(c) \;\text{There exist positive constants }T_1\leq T_0\text{ and }C_1, \text{ independent of }n,
\text{ such that for every  }n\\ &\qquad |U^n_\eps-\cU^{0,n}_{\eps}|_{E^{s-3}_{T_1}}\leq C_1 \, \eps^{\frac{1}{2M_1+5}}.
\end{split}
\end{align}
The first statement, together with uniform boundedness of the families $U^n_\eps$, $U_\eps$ in $E^s_{T_0}$, is proved
in Theorem \ref{estnl} by showing convergence of the scheme \eqref{e5} using the following linear estimate (which is a
consequence of Proposition \ref{estimEk}).

\begin{prop}\label{e9}
Let $s\geq [M_0+\frac{d+1}{2}]$ and consider the  problem \eqref{e5}, where $G\in H^{s+1}$ vanishes in $t\leq 0$,
and where the right side of \eqref{e5}(a) is replaced by $\cF\in E^s_T$. Suppose $U^n_\eps \in E^s_T$
and that for some $K>0$, $\eps_1>0$, we have
\begin{equation*}
|U^n_\eps|_{E^s_T}+|\eps \, \partial_{x_d}U^n_\eps|_{L^\infty}\leq K \text{ for }\eps\in (0,\eps_1].
\end{equation*}
Then there exist  constants $T_0(K)$ and $\eps_0(K)\leq \eps_1$ such that for $0<\eps\leq\eps_0$ and $T\leq T_0$
we have
\begin{equation*}
|U^{n+1}_\eps|_{E^s_T} +\sqrt{T} \, \langle U^{n+1}_\eps|_{x_d=0}\rangle_{s+1,T} \leq C(K) \, \left( T \, |\cF|_{E^s_T}
+\sqrt{T} \, \langle G\rangle_{s+1,T}\right).
\end{equation*}
\end{prop}


\begin{proof}[Proof of Theorem \ref{e1}]

\textbf{1. }The boundedness of $\cU^0_\eps$ in $E^s_{T_0}$ and \eqref{e8}(b) follow  directly from Proposition \ref{h1},
together with the fact that $\cV^0\in H^{s+1}_{T_0}$ and $\cV^{0,n}\to\cV^0$ in $H^s_{T_0}$. (In fact, the proof of
Proposition \ref{c37} shows that the $\cV^{0,n}$ are bounded in $\Gamma^{s+1}_{T_0}$ and $\cV^{0,n}\to\cV^0$ in
$\Gamma^s_{T_0}$.)

\textbf{2. }The approximate solution $\cU^{0,n}_\eps$ is by itself too crude; in order to prove \eqref{e4} using Proposition
\ref{e9} we must construct a  corrector $\eps \, \cU^1_{p,\eps}$ that lies  in some $E^r_{T_0}$ space. To achieve this we
first approximate $\cU^{0,n}$ and $\cU^{0,n+1}$ by moment-zero approximations $\cU^{0,n}_p$ and $\cU^{0,n+1}_p$.
For now we fix $0<p<1$ and define for each $n$
\begin{equation*}
\cU^{0,n}_p(x,\theta_0,\xi_d)=\sum^3_{j=1}\sigma^n_{j,p}(x,\theta_0+\omega_j\xi_d) \, r_j,
\end{equation*}
where $\sigma^n_{j,p}$ is the moment-zero approximation to $\sigma^n_j$ defined by \eqref{h5a}. Thus we have
$\cU^{0,n}_p(x,\theta_0,\xi_d) ={\bf E} \, \cU^{0,n}_p(x,\theta_0,\xi_d)$ and by Proposition \ref{h6}
\begin{equation}\label{e17}
|\cU^{0,n}-\cU^{0,n}_p|_{\cE^{s-1}_{T_0}} \le C \, \sqrt{p} \, ,\quad \text{ and } \quad
|\partial_{x_d}\cU^{0,n+1}-\partial_{x_d}\cU^{0,n+1}_p|_{\cE^{s-2}_{T_0}} \le C \, \sqrt{p} \, ,
\end{equation}
for $C$ independent of $n$\footnote{Constants $C,C_1...$ appearing in this proof are all independent of $n$ and
$\eps$.}.

\textbf{3. }For now we express the induction assumption as: there exists $0<a<1$ (to be determined) and positive
constants $C_1$, $T_1\leq T_0$ such that
\begin{align}\label{e18}
|U^n_\eps-\cU^{0,n}_\eps|_{E^{s-3}_{T_1}}\leq C_1 \, \eps^a.
\end{align}
The boundedness of the family $U^n_\eps$ in $E^s_{T_0}$ together with \eqref{e18} imply
\begin{equation*}
|F(\eps U^n_\eps) \, U^n_\eps -F(0) \, \cU^{0,n}_\eps|_{E^{s-3}_{T_0}}\leq C \, \eps^a.
\end{equation*}
In view of \eqref{e17} and Proposition \ref{h1} this implies
\begin{align}\label{e20}
|F(\eps U^n_\eps) \, U^n_\eps -F(0) \, \cU^{0,n}_{p,\eps}|_{E^{s-3}_{T_0}}\leq C \, (\sqrt{p}+\eps^a).
\end{align}

\textbf{4. }Define
\begin{equation*}
\cG_p :=\tilde L(\partial_x)\cU^{0,n+1}_p +M(\cU^{0,n}_p,\partial_{\theta_0}\cU^{0,n+1}_p).
\end{equation*}
We claim that
\begin{align}\label{e22}
|{\bf E}\cG_p -{\bf E}(F(0) \, \cU^{0,n}_p)|_{\cE^{s-2}_{T_0}}\leq C\sqrt{p}.
\end{align}
Indeed, from \eqref{e17} and the explicit formula \eqref{31a} for the action of ${\bf E}$ on functions of type $\cF$,
we have
\begin{equation*}
|{\bf E}\left( F(0) \, \cU^{0,n} -F(0) \, \cU^{0,n}_p \right)|_{\cE^{s-1}_{T_0}}\leq C\sqrt{p}.
\end{equation*}
But ${\bf E} (F(0)\cU^{0,n})$ is given by the left side of \eqref{e6}(b), so \eqref{e22} follows by observing that
\eqref{e17} and Proposition \ref{d23a} imply
\begin{align}\label{e24}
\begin{split}
&\left| {\bf E} \left( \tilde L(\partial_x)\left(\cU^{0,n+1}-\cU^{0,n+1}_p\right)\right) \right|_{\cE^{s-2}_{T_0}}\leq C\sqrt{p}\\
&|{\bf E} \left(M(\cU^{0,n},\partial_{\theta_0}\cU^{0,n+1})-M(\cU^{0,n}_p,\partial_{\theta_0}\cU^{0,n+1}_p)\right)
|_{\cE^{s-2}_{T_0}}\leq C\sqrt{p}.
\end{split}
\end{align}
Here we have used the fact that the arguments of ${\bf E}$ in \eqref{e24} are functions of type $\cF$, so the
formula \eqref{31a} can be applied.

\textbf{5. }Next define the operator
\begin{equation*}
\LL_0 :=\tilde L(\partial_x)+\frac{1}{\eps}\tilde L({\rm d}\phi_0)\partial_{\theta_0} + M(\cU^{0,n}_{p,\eps},\partial_{\theta_0}),
\end{equation*}
which is an approximation to the operator appearing on the left side of \eqref{e5}(a) that will allow us to use
Proposition \ref{35} to construct a useful corrector $\cU^1_p$. Indeed, we claim
\begin{align}\label{e26}
|\LL_0 U^{n+1}_\eps-F(\eps U^n_\eps)U^n_\eps|_{E^{s-3}_{T_0}} \leq C(\sqrt{p}+\eps^a).
\end{align}
This follows from \eqref{e5}(a) and the estimates
\begin{align}\label{e27}
\begin{split}
&|\tilde A_j(\eps U^n_\eps)\partial_{x_j}U^{n+1}_\eps -\tilde A_j(0)\partial_{x_j}U^{n+1}_\eps|_{E^{s-1}_{T_0}}
\leq C\eps \, ,\\
&\left| \dfrac{1}{\eps} \, \tilde A_j(\eps U^n_\eps) \, \beta_j \, \partial_{\theta_0}U^{n+1}_\eps -\left(
\dfrac{1}{\eps} \, \tilde A_j(0) \, \beta_j \, \partial_{\theta_0}U^{n+1}_\eps +{\rm d}\tilde A_j(0) \cdot U^n_\eps
\, \beta_j \, \partial_{\theta_0}U^{n+1}_\eps \right) \right|_{E^{s-1}_{T_0}} \leq C \, \eps \, ,\\
&\left|{\rm d}\tilde A_j(0) \cdot (U^n_\eps-\cU^{0,n}_{p,\eps}) \, \beta_j \, \partial_{\theta_0}U^{n+1}_\eps
\right|_{E^{s-3}_{T_0}} \leq C \, |U^n_\eps-\cU^{0,n}_{p,\eps}|_{E^{s-3}_{T_0}} \leq C \, (\sqrt{p}+\eps^a).
\end{split}
\end{align}

\textbf{6. Construction of the corrector.} First observe that since $\tilde \cL(\partial_{\theta_0},\partial_{\xi_d})
\cU^{0,n+1}_p=0$, we have
\begin{equation*}
\LL_0 \, \cU^{0,n+1}_{p,\eps}=\cG_{p,\eps} \, ,
\end{equation*}
and thus
\begin{align}\label{e29}
\begin{split}
&\LL_0\cU^{0,n+1}_{p,\eps}-F(0)\cU^{0,n}_{p,\eps}=\cG_{p,\eps}-F(0)\cU^{0,n}_{p,\eps}=\\
&\quad \left({\bf E}(\cG_{p}-F(0)\cU^{0,n}_{p})\right)_\eps+\left((I-{\bf E})(\cG_{p}-F(0)\cU^{0,n}_{p})\right)_\eps.
\end{split}
\end{align}
We have
\begin{align}\label{e29a}
|\left({\bf E}(\cG_{p}-F(0)\cU^{0,n}_{p})\right)_\eps|_{E^{s-2}_{T_0}}\leq C \, \sqrt{p} \, ,
\end{align}
by \eqref{e22}, the formula \eqref{31a} for ${\bf E}$, and Proposition \ref{h1}. The second term on the right in
\eqref{e29} is not small, so we construct $\cU^1_p$ to solve (most of) it away. By Proposition \ref{35}  the function
$\tilde{\cU}^1_p:=-{\bf R}_\infty\left((I-{\bf E})(\cG_{p}-F(0)\cU^{0,n}_{p})\right)$ satisfies
\begin{align}\label{e30}
\tilde{\cL}(\partial_{\theta_0},\partial_{\xi_d})\tilde{\cU}^1_p=-(I-{\bf E})(\cG_{p}-F(0)\cU^{0,n}_{p}).
\end{align}
However, this choice of $\tilde{\cU}^1_p$ is too large to be useful in the error analysis.

\textbf{7. }To remedy this problem we replace $(I-{\bf E})\cG_p$ by a modification $[(I-{\bf E})\cG_p]_{mod}$ defined
as follows. First, using \eqref{25} and Remark \ref{30ab} we have
\begin{align}\label{k1}
(I-{\bf E})\cG_p=\sum^3_{i=1}\left(-\sum_{k\neq i}V^i_k\sigma^{n+1}_{k,p}+\sum_{k\neq i} c^i_k\sigma^n_{k,p}
\partial_{\theta_0}\sigma^{n+1}_{k,p}+\sum_{l\neq m}d^i_{l,m}\sigma^n_{l,p}\partial_{\theta_0}\sigma^{n+1}_{m,p}
\right)r_i,
\end{align}
where $\sigma^n_{q,p}=\sigma^n_{q,p}(x,\theta_0+\omega_q\xi_d)$. The problem is caused by the nontransversal
interaction terms given by the middle sum over $k\neq i$, so we define
\begin{align}\label{k2}
[(I-{\bf E})\cG_p]_{mod}=\sum^3_{i=1}\left(-\sum_{k\neq i}V^i_k\sigma^{n+1}_{k,p}+\sum_{k\neq i} c^i_k(\sigma^n_{k,p}
\partial_{\theta_0}\sigma^{n+1}_{k,p})_p+\sum_{l\neq m}d^i_{l,m}\sigma^n_{l,p}\partial_{\theta_0}\sigma^{n+1}_{m,p}
\right)r_i,
\end{align}
and we set
\begin{align}\label{k3}
\cU^1_p:=-{\bf R}_\infty \left( [(I-{\bf E})\cG_{p}]_{mod}-(I-{\bf E})F(0)\cU^{0,n}_{p})\right).
\end{align}
Instead of \eqref{e30} we have
\begin{align}\label{k4}
\tilde{\cL}(\partial_{\theta_0},\partial_{\xi_d})\cU^1_p=-[(I-{\bf E})\cG_{p}]_{mod}+(I-{\bf E})F(0)\cU^{0,n}_{p}.
\end{align}
For later use we set
\begin{equation*}
D(x,\theta_0,\xi_d):=(I-{\bf E})\cG_p-[(I-{\bf E})\cG_p]_{mod}
\end{equation*}
and estimate
\begin{align}\label{k6}
|D(x,\theta_0,\frac{x_d}{\eps})|_{E^{s-3}_T}\leq C \, \sqrt{p}.
\end{align}
Indeed, using Propositions \ref{h11a} and \ref{h8}(a) we have
\begin{multline*}
|\left(\sigma^n_{k,p}\partial_{\theta_0}\sigma^{n+1}_{k,p}-(\sigma^n_{k,p}\partial_{\theta_0}\sigma^{n+1}_{k,p})_p
\right)(x,\theta_0+\omega_k\frac{x_d}{\eps})|_{E^{s-3}_T} \\
\le |\sigma^n_{k,p}\partial_{\theta_0}\sigma^{n+1}_{k,p}-\sigma^n_{k,p}\partial_{\theta_0}\sigma^{n+1}_{k,p})_p
|_{H^{s-2}_T}\leq|\sigma^n_k|_{H^{s-2}_T}|\sigma^{n+1}_k|_{H^{s-1}_T}\sqrt{p}.
\end{multline*}

\textbf{8. Estimate of $|\cU^1_{p,\eps}|_{E^{s-2}_T}$.}  By \eqref{k3},\eqref{k2} and the formula \eqref{R} for
${\bf R}_\infty$, for $i=1,2,3$ we must estimate $|b_i(x,\theta_0,\frac{x_d}{\eps})|_{E^{s-2}_T}$, where
$b_i(x,\theta_0,\xi_d)=$
\begin{align}\label{k8}
\begin{split}
&\quad\sum_{k\neq i} c^i_k\int^{\xi_d}_\infty\left(\sigma^n_{k,p}\partial_{\theta_0}\sigma^{n+1}_{k,p} \right)_p
(x,\theta_0+\omega_i\xi_d+s(\omega_k-\omega_i)) \;{\rm d}s+\\
&\quad \sum_{m\neq i}d^i_{i,m}\int^{\xi_d}_\infty\sigma^n_{i,p}(x,\theta_0+\omega_i\xi_d)
\partial_{\theta_0}\sigma^{n+1}_{m,p}(x,\theta_0+\omega_i\xi_d+s(\omega_m-\omega_i)) \; {\rm d}s+\\
&\quad \sum_{l\neq i} d^i_{l,i}\int^{\xi_d}_\infty\sigma^n_{l,p}(x,\theta_0+\omega_i\xi_d+s(\omega_l-\omega_i))
\partial_{\theta_0}\sigma^{n+1}_{i,p}(x,\theta_0+\omega_i\xi_d) \; {\rm d}s+\\
&\quad \quad \sum_{l\neq m,l\neq i,m\neq i}d^i_{l,m}\int^{\xi_d}_\infty\sigma^n_{l,p}(x,\theta_0+\omega_i\xi_d
+s(\omega_l-\omega_i))\partial_{\theta_0}\sigma^{n+1}_{m,p}(x,\theta_0+\omega_i\xi_d+s(\omega_m-\omega_i))
\; {\rm d}s+\\
&\quad\quad\quad \quad\quad \sum_{k\neq i}e^i_k\int^{\xi_d}_\infty\sigma^n_{k,p}(x,\theta_0+\omega_i\xi_d
+s(\omega_k-\omega_i)) \;{\rm d}s -\\
&\quad\quad\quad \quad\quad \sum_{k\neq i}\int^{\xi_d}_\infty V^i_k\sigma^n_{k,p}(x,\theta_0+\omega_i\xi_d
+s(\omega_k-\omega_i)) \; {\rm d}s =\sum^6_{r=1}b_{i,r}(x,\theta_0,\xi_d),
\end{split}
\end{align}
where $b_{i,r}$, $r=1,\dots,6$ are defined by the respective lines of \eqref{k8}. Since $\cV^{0,n}$ is bounded
in $H^{s+1}_T$, using Corollary \ref{h17} we find
\begin{equation*}
|b_{i,1}(x,\theta_0,\frac{x_d}{\eps})|_{E^{s-2}_T} \leq C \sum_{k\neq i}\frac{|\sigma^n_k|_{H^{s-1}_T}
|\partial_\theta\sigma^{n+1}_k|_{H^{s-1}_T}}{p} \leq C/p \, .
\end{equation*}
Similarly, from Propositions \ref{h18} and \ref{h19} we get respectively
\begin{equation*}
|b_{i,2}(x,\theta_0,\frac{x_d}{\eps})|_{E^{s-2}_T}\leq C \, ,\quad
|b_{i,3}(x,\theta_0,\frac{x_d}{\eps})|_{E^{s-2}_T}\leq C/p \, .
\end{equation*}
Since $\cV^{0,n}$ is actually bounded in $\Gamma^{s+1}_T$, Proposition \ref{h12} on transversal interactions implies
\begin{equation*}
|b_{i,4}(x,\theta_0,\frac{x_d}{\eps})|_{E^{s-2}_T}\leq  \frac{C}{p^{t+1}},
\end{equation*}
where we have used Proposition \ref{h8}(b) to estimate
\begin{equation*}
|\sigma^{n+1}_{m,p}|_{\Gamma^{t+1}_T}\leq \dfrac{C}{p^{t+1}} \, |\sigma^{n+1}_{m}|_{\Gamma^{t+1}_T} \, .
\end{equation*}

By Proposition \ref{h20} we have $|b_{i,5}(x,\theta_0,\frac{x_d}{\eps})|_{E^{s-2}_T}\leq  C$, and the estimate of
$b_{i,6}$ is the same, so adding up we obtain
\begin{align}\label{k12}
|\cU^1_{p,\eps}|_{E^{s-2}_T}\leq \dfrac{C}{p^{t+1}} \, .
\end{align}
To estimate $(\partial_{x_d}\cU^1_p)_\eps$ we differentiate \eqref{k8} and estimate as above to find
\begin{align}\label{k13}
|(\partial_{x_d}\cU^1_p)_\eps|_{E^{s-3}_T}\leq \dfrac{C}{p^{t+2}} \, .
\end{align}

\textbf{9. }We claim
\begin{align}\label{e31}
|\LL_0 \left(\cU^{0,n+1}_{p,\eps} +\eps \, \cU^1_{p,\eps}\right) -F(0) \, \cU^{0,n}_{p,\eps}|_{E^{s-3}_{T_0}} \leq
C \, \left(\sqrt{p}+\dfrac{\eps}{p^{t+2}}\right).
\end{align}
Indeed, we have
\begin{equation*}
\LL_0(\eps \, \cU^1_{p,\eps})=(\tilde{\cL}(\partial_{\theta_0},\partial_{\xi_d})\cU^1_p)_\eps
+(\tilde L(\partial)\eps \, \cU^1_p)_\eps +M(\cU^{0,n}_{p,\eps},\partial_{\theta_0})(\eps \, \cU^1_{p,\eps}),
\end{equation*}
so by \eqref{e29} and \eqref{k4} we find
\begin{multline*}
\LL_0\left(\cU^{0,n+1}_{p,\eps}+\eps \, \cU^1_{p,\eps}\right)-F(0) \, \cU^{0,n}_{p,\eps}=\\
\left({\bf E}(\cG_{p}-F(0)\cU^{0,n}_{p})\right)_\eps+D(x,\theta_0,\frac{x_d}{\eps})
+(\tilde L(\partial) \eps \, \cU^1_p)_\eps +M(\cU^{0,n}_{p,\eps},\partial_{\theta_0})(\eps\cU^1_{p,\eps}).
\end{multline*}
The estimate \eqref{31} now follows from \eqref{e29a}, \eqref{k6}, \eqref{k12}, and \eqref{k13}.

Using \eqref{e20}, \eqref{e26}, and \eqref{e31}, we obtain
\begin{equation}\label{e33}
\left|\LL_0\left(U^{n+1}_\eps-(\cU^{0,n+1}_{p,\eps}+\eps \, \cU^1_{p,\eps})\right)\right|_{E^{s-3}_{T_0}}
\leq C \, (\sqrt{p}+\eps^a+\frac{\eps}{p^{t+2}}) \, .
\end{equation}

\textbf{10. }Next we claim that the following estimates hold:
\begin{align}\label{e34}
&\begin{split}
&(a)\; \left|\left( \partial_{x_d} +\bA(\eps \, \cU^{0,n}_{p,\eps},\partial_{x'}+\dfrac{\beta \, \partial_{\theta_0}}{\eps})
\right) \left(U^{n+1}_\eps -(\cU^{0,n+1}_{p,\eps}+\eps \, \cU^1_{p,\eps}) \right) \right|_{E^{s-3}_{T_0}} \leq
C \, (\sqrt{p}+\eps^a+\frac{\eps}{p^{t+2}}) \, ,\\
&(b)\; \left|B(\eps \, \cU^{0,n}_{p,\eps}) \left(U^{n+1}_\eps-(\cU^{0,n+1}_{p,\eps}+\eps \, \cU^1_{p,\eps}) \right)
\right|_{H^{s-2}_{T_0}} \leq C \, (\sqrt{p}+\eps^a+\frac{\eps}{p^{t+2}}).
\end{split}
\end{align}
Indeed, \eqref{e34}(a) follows from \eqref{e33} by estimates similar to \eqref{e27}, while \eqref{e34}(b) is a simple
consequence of \eqref{e5}(b) and \eqref{e6}(c). Applying Proposition \ref{e9} we find
\begin{equation*}
|U^{n+1}_\eps -(\cU^{0,n+1}_{p,\eps} +\eps \, \cU^1_{p,\eps})|_{E^{s-3}_{T_0}} \leq C \, \sqrt{T_0} \,
(\sqrt{p}+\eps^a+\frac{\eps}{p^{t+2}}) \, ,
\end{equation*}
and thus
\begin{equation*}
|U^{n+1}_\eps -\cU^{0,n+1}_{\eps}|_{E^{s-3}_{T_0}} \leq C \, \sqrt{T_0} \, (\sqrt{p}+\eps^a+\frac{\eps}{p^{t+2}}) \, .
\end{equation*}
Recall that $t$ is fixed and equals $M_1$ in the notation of Theorem \ref{e1}. Setting $p=\eps^b$ we compute
$\sqrt{p}=\frac{\eps}{p^{t+2}}$ when $b=\frac{2}{2t+5}$, so we take $a=\frac{b}{2}=\frac{1}{2t+5}$ and complete
the induction step by shrinking $T_0$ to a small enough $T_1$ if necessary. This completes the proof of Theorem
\ref{e1}.
\end{proof}

\section{Extension to the general $N\times N$ case}
\label{extension}

\emph{\quad} Here we describe the relatively minor changes needed to treat $N\times N$ systems satisfying
Assumptions \ref{assumption1}, \ref{assumption2}, and \ref{a7}. We first describe the construction of profiles
in the general $N\times N$ case. For each $m\in \{1,\dots,M\}$, let
\begin{align*}
\ell_{m,k},\;k=1,\dots,\nu_{k_m} \, ,
\end{align*}
denote a basis of real vectors for the left eigenspace of the real matrix $i\cA(\beta)$ associated to the real
eigenvalue $-\omega_m$ and chosen to satisfy
\begin{align*}
\ell_{m,k}\cdot r_{m',k'}=\begin{cases}1, \;\text{ if }m=m'\text{ and }k=k' \, ,\\
0,\;\text{ otherwise.}
\end{cases}
\end{align*}
For $v\in\C^N$ we set
\begin{align*}
P_{m,k} \, v:=(\ell_{m,k}\cdot v) \, r_{m,k}\;  \text{ (no complex conjugation here)}.
\end{align*}

Functions of type $\cF$ (see Definition \ref{30}) have the form
\begin{align}\label{r4}
F(x,\theta_0,\xi_d)=\sum_{m=1}^M\sum_{k=1}^{\nu_{k_m}}F_{m,k}(x,\theta_0,\xi_d) \, r_{m,k}
\end{align}
where each scalar function $F_{m,k}$ is decomposed as
\begin{multline}
\label{r4a}
F_{m,k}=\sum_{m'} f^{m,k}_{m'}(x,\theta_0+\omega_{m'} \, \xi_d) \\
+\sum_{m',k',m'',k''} g^{m,k}_{m',k',m'',k''}(x,\theta_0+\omega_{m'} \, \xi_d) \,
h^{m,k}_{m',k',m'',k''}(x,\theta_0+\omega_{m''} \, \xi_d) \, .
\end{multline}
In \eqref{r4a}, $m'\in \{1,\dots,M\}$, $k'\in\{1,\dots,\nu_{k_{m'}}\}$, and similarly for $(m'',k'')$; moreover, the
functions $f^{m,k}_{m',k'}$ etc. have the same properties as the corresponding functions in Definition \ref{30}.
The averaging operator ${\bf E}$ is given by
\begin{align*}
{\bf E}F :=\sum_{m,k} \left( \lim_{T\to\infty} \dfrac{1}{T} \, \int^T_0
F_{m,k}(x,\theta_0+\omega_m \, (\xi_d-s),s) \, {\rm d}s \, \right) \, r_{m,k} \, ,
\end{align*}
and for $F$ as in \eqref{r4}, it follows that ${\bf E}F=\sum_{m,k}\tilde F_{m,k} \, r_{m,k}$ where
\begin{align*}
\tilde F_{m,k} :=f^{m,k}_m(x,\theta_0+\omega_m \, \xi_d) +\sum_{k',k''}
g^{m,k}_{m,k',m,k''}(x,\theta_0+\omega_m \, \xi_d) \, h^{m,k}_{m,k',m,k''}(x,\theta_0+\omega_m \, \xi_d).
\end{align*}
On functions of type $\cF$ such that ${\bf E}F=0$, the action of the operator ${\bf R}_\infty$ is given by
\begin{align*}
{\bf R}_\infty F :=\sum_{m,k} \left( \int^{\xi_d}_\infty F_{m,k}(x,\theta_0+\omega_m(\xi_d-s),s) \, {\rm d}s \right)
\, r_{m,k} \, .
\end{align*}

The general form of the profile equations \eqref{36a} still applies. With
\begin{align*}
W(x,\theta_0,\xi_d)=\sum_{m,k}w_{m,k}(x,\theta_0,\xi_d) \, r_{m,k},
\end{align*}
the decomposition of Proposition \ref{23a} now has the form
\begin{align*}
\tilde L(\partial)W =\sum_{m,k} (X_{\phi_m}w_{m,k}) \, r_{m,k}
+\sum_{m,k} \left( \sum_{m'\neq m,k'} V^{m,k}_{m',k'} \, w_{m',k'} \right) \, r_{m,k},
\end{align*}
where $V^{m,k}_{m',k'}$ is the tangential vector field
\begin{align*}
V^{m,k}_{m',k'} :=\sum^{d-1}_{j=0} (\ell_{m,k} \, \tilde A_j(0) \, r_{m',k'}) \, \partial_{x_j}.
\end{align*}

In place of \eqref{ddd5} and \eqref{c23} we now have
\begin{align*}
\cU^{0,n}(x,\theta_0,\xi_d) &=\sum_{m=1}^M \sum_{k=1}^{\nu_{k_m}} \sigma_{m,k}^n(x,\theta_0+\omega_m \, \xi_d)
\, r_{m,k} \, ,\\
\cV^{0,n+1}(x,\theta) &=\Big( \sigma_{m,k}^{n+1}(x,\theta) \Big)_{m=1,\dots,M; \, k=1,\dots,\nu_{k_m}} \, .
\end{align*}
The argument that led to the profile system \eqref{f2} now gives\footnote{The nonlinear equations for the
functions $\sigma_{m,l}$ are, of course, obtained from \eqref{c20} by removing the superscripts $n$ and $n+1$.}
\begin{align}\label{c20}
\begin{split}
&(a)\; X_{\phi_m} \sigma_{m,l}^{n+1} +\sum^{d-1}_{j=0} \sum^{\nu_{k_m}}_{k,k'=1} b_{m,l,j}^{k,k'} \,
\sigma^n_{m,k} \, \partial_{\theta} \sigma^{n+1}_{m,k'} =\sum^{\nu_{k_m}}_{k=1} e^k_{m,l} \, \sigma^n_{m,k} \, ,\\
&(b)\;\left( \sigma^{n+1}_{m,k} (x',0,\theta), m\in\cI, k=1,\dots\nu_{k_m} \right) =\cB \left( G(x',\theta),
\sigma^{n+1}_{m,k}(x',0,\theta), m\in\cO, k=1,\dots\nu_{k_m} \right) \, ,\\
&(c)\;\sigma^{n+1}_{m,k}=0 \quad \text{ in }t\leq 0\text{ for all } m,k,
\end{split}
\end{align}
where the coefficients $b_{m,l,j}^{k,k'}$ are defined by
\begin{align}\label{c20b}
b_{m,l,j}^{k,k'} :=\ell_{m,l} \cdot \beta_j \, ({\rm d}\tilde{A}_j(0) \, r_{m,k}) \, r_{m,k'}.
\end{align}

\begin{rem}\label{c20c}
\textup{There is a potentially serious obstacle to proving estimates for the system \eqref{c20}. If one takes the
$L^2$ pairing of \eqref{c20}(a) with $\sigma_{m,l}^{n+1}(x,\theta)$, it is not clear how to use integration by parts
in $\theta$ to move the $\theta-$derivative in the sum on the left onto the $n$-th iterate. This problem does not
arise in the estimate for \eqref{f2}. The next Proposition, which is \cite[Proposition 2.18]{CGW1}, removes this
difficulty by showing  that there is a symmetry in the coefficients that appears after regrouping.}
\end{rem}

\begin{defn}\label{b25a}
For $u$ near $0$ let $-\omega_m(u)$, $m=1,\dots,M$, be the eigenvalues of
\begin{align*}
i \, \cA(u,\beta) :=A_d^{-1}(u) \, \left( \utau \, I +\sum_{j=1}^{d-1} \ueta_j \, A_j(u)\right),
\end{align*}
and $P_m(u)$ the corresponding projectors.
\end{defn}

The functions $\omega_m(u)$ and $P_m(u)$ are $C^\infty$ for $u$ near $0$ since $\beta$ then belongs to the
hyperbolic region of $\cA(u,\xi')$.

\begin{prop}\label{c20d}
Let $w\in\R^N$ be expanded as $w=\sum_{m,k}w_{m,k}r_{m,k}=\sum_m w_m$ and define
\begin{equation}\label{c20e}
B^m_{l,k'}(w) := \sum_{j=0}^{d-1} \sum_{k=1}^{\nu_{k_m}} b_{m,l,j}^{k,k'} \, w_{m,k} \, ,
\end{equation}
where the $b_{m,l,j}^{k,k'}$ are defined in \eqref{c20b}. Then there holds
\begin{equation}
\label{c20ee}
B^m_{l,k'}(w) =\begin{cases}
-{\rm d}\omega_m(0) \cdot w_m &\text{if $k'=l$,}\\
0 &\text{otherwise.}
\end{cases}
\end{equation}
\end{prop}

\begin{proof}
We differentiate the equation
\begin{equation*}
\left( \omega_m(u) \, I +\sum^{d-1}_{j=0} \beta_j \, \tilde A_j(u) \right) \, P_m(u)=0 \, ,
\end{equation*}
with respect to $u$ in the direction $w_m$, evaluate at $u=0$, and apply $P_m:=P_m(0)$ on the left to obtain
\begin{align}\label{c20i}
P_m \, \sum^{d-1}_{j=0}\beta_j \, \left({\rm d}\tilde A_j(0) \cdot w_m \right) \, P_m
=(-{\rm d}\omega_m(0) \cdot w_m) \, P_m \, .
\end{align}
The second equality in \eqref{c20e} and \eqref{c20i} imply \eqref{c20ee}.
\end{proof}

Proposition \ref{c20d} allows us to write
\begin{align}
\label{c20j}
\sum^{d-1}_{j=0} \sum^{\nu_{k_m}}_{k,k'=1} b_{m,l,j}^{k,k'} \, \sigma^n_{m,k} \, \partial_{\theta}\sigma^{n+1}_{m,k'}
=B^m_{l,l} (\cW^{0,n}) \, \partial_\theta \sigma^{n+1}_{m,l},
\end{align}
where $\cW^{0,n} :=\sum_{m,k} \sigma^n_{m,k} \, r_{m,k}$; hence we can shift the $\theta-$derivative and
integrate by parts as discussed in Remark \ref{c20c}. Using \eqref{c20j}, we deduce from \eqref{c20} that
$\sigma^{n+1}_{m,k}=0$ when $m\in\cO$. Otherwise the proof of Proposition \ref{c24} goes through as
before. The statement of Proposition \ref{c37} is thus unchanged, except that in the second sentence we
have $\sigma_{m,k}=0$ when $m\in\cO$ now.

The formulation of Theorem \ref{e1} is exactly as before except now
\begin{equation*}
\cU^0(x,\theta_0,\xi_d)=\sum_{m=1}^M \sum_{k=1}^{\nu_{k_m}} \sigma_{m,k}(x,\theta_0+\omega_m\xi_d) \,
r_{m,k} \, ,\quad \cV^0(x,\theta) =\Big( \sigma_{m,k}(x,\theta) \Big)_{m=1,\dots,M;k=1,\dots,\nu_{k_m}} \, .
\end{equation*}
The error analysis in the proof of Theorem \ref{e1} goes through with the obvious minor changes. For example,
the troublesome self-interaction terms $c^i_k \, \sigma^n_{k,p} \, \partial_{\theta_0}\sigma^{n+1}_{k,p}$, $k\neq i$,
in \eqref{k1} are now replaced by terms of the form $c^i_{m,k,k'} \, \sigma^n_{m,k,p} \, \partial_{\theta_0}
\sigma^{n+1}_{m,k',p}$, $m\neq i$, where the index $p$ as before denotes a moment-zero approximation.
These terms are handled just as before by introducing $[(I-{\bf E}) \, \cG]_{mod}$, see \eqref{k2}, in which
they are replaced by $c^i_{m,k,k'} \, (\sigma^n_{m,k,p} \, \partial_{\theta_0}\sigma^{n+1}_{m,k',p})_p$. The
contribution of these terms to the corrector $\cU^1_{p,\eps}$ is estimated as before using Corollary \ref{h17}.

\appendix
\section{Singular pseudodifferential calculus for pulses}
\label{append}

\emph{\quad} Here we summarize the parts of the singular pulse calculus constructed in \cite{CGW2} that
are needed in this article. First we define the singular Sobolev spaces used to describe mapping properties.

The variable in $\R^{d+1}$ is denoted $(x,\theta)$, $x \in \R^d$, $\theta \in \R$, and the associated frequency
is denoted $(\xi,k)$. In this context, the singular Sobolev spaces are defined as follows. We consider a vector
$\beta \in \R^d \setminus \{ 0\}$. Then for $s \in \R$ and $\eps \in \, ]0,1]$, the anisotropic Sobolev space
$H^{s,\eps} (\R^{d+1})$ is defined by
\begin{multline*}
H^{s,\eps}(\R^{d+1}) := \Big\{ u \in {\mathcal S}'(\R^{d+1}) \, / \, \widehat{u} \in L^2_{\rm loc}(\R^{d+1}) \\
\text{\rm and} \quad \int_{\R^{d+1}} \left( 1+\left| \xi+\dfrac{k \, \beta}{\eps} \right|^2 \right)^s
\, \big| \widehat{u}(\xi,k) \big|^2 \, {\rm d}\xi \, {\rm d}k <+\infty \Big\} \, .
\end{multline*}
Here $\widehat{u}$ denotes the Fourier transform of $u$ on $\R^{d+1}$. The space $H^{s,\eps}(\R^{d+1})$ is
equipped with the family of norms
\begin{equation*}
\forall \, \gamma \ge 1 \, ,\quad \forall \, u \in H^{s,\eps}(\R^{d+1}) \, ,\quad
\| u \|_{H^{s,\eps},\gamma}^2 := \dfrac{1}{(2\, \pi)^{d+1}} \, \int_{\R^{d+1}}
\left( \gamma^2 +\left| \xi+\dfrac{k \, \beta}{\eps} \right|^2 \right)^s
\, \big| \widehat{u}(\xi,k) \big|^2 \, {\rm d}\xi \, {\rm d}k \, .
\end{equation*}
When $m$ is an integer, the space $H^{m,\eps} (\R^{d+1})$ coincides with the space of functions $u \in L^2
(\R^{d+1})$ such that the derivatives, in the sense of distributions,
\begin{equation*}
\left( \partial_{x_1} +\dfrac{\beta_1}{\eps} \, \partial_\theta \right)^{\alpha_1} \dots
\left( \partial_{x_d} +\dfrac{\beta_d}{\eps} \, \partial_\theta \right)^{\alpha_d} \, u \, ,\quad
\alpha_1+\dots+\alpha_d \le m \, ,
\end{equation*}
belong to $L^2 (\R^{d+1})$. In the definition of the norm $\| \cdot \|_{H^{m,\eps},\gamma}$, one power of
$\gamma$ counts as much as one derivative.

\subsection{Symbols}

\emph{\quad} In this Appendix, $\cO$ denotes an open set and nolonger denotes the set of outgoing
phases. Our singular symbols are built from the following sets of classical symbols.

\begin{defn}\label{n1}
Let $\cO\subset \R^N$ be an open subset that contains the origin. For $m\in\R$, we let $\bfS^m(\cO)$ denote
the class of all functions $\sigma:\cO\times \R^d\times [1,\infty)\to \C^M$, $M \ge 1$, such that $\sigma$ is
$C^\infty$ on $\cO \times \R^d$ and for all compact sets $K\subset \cO$:
\begin{equation*}
\sup_{v\in K} \, \sup_{\xi \in\R^d} \, \sup_{\gamma\geq 1} \, (\gamma^2+|\xi|^2)^{-(m-|\nu|)/2} \,
|\partial^\alpha_v\partial_\xi^\nu \sigma(v,\xi,\gamma)| \leq C_{\alpha,\nu,K}.
\end{equation*}
\end{defn}

Let ${\mathcal C}^k_b(\R^{d+1})$, $k \in \N$, denote the space of continuous and bounded functions on $\R^{d+1}$,
whose derivatives up to order $k$ are continuous and bounded. Let us next define the singular symbols.

\begin{definition}[Singular symbols]
\label{def4}
Fix $\beta\in\R^d\setminus \{ 0\}$, let $m \in \R$ and let $n \in \N$. Then we let $S^m_n$ denote the set of families
of functions $(a_{\eps,\gamma})_{\eps \in ]0,1],\gamma \ge 1}$ that are constructed as follows:
\begin{equation}
\label{singularsymbolp}
\forall \, (x,\theta,\xi,k) \in \R^{d+1} \times \R^{d+1} \, ,\quad a_{\eps,\gamma} (x,\theta,\xi,k) =
\sigma \left( \eps \, V(x,\theta),\xi+\dfrac{k \, \beta}{\eps},\gamma \right) \, ,
\end{equation}
where $\sigma \in \bfS^m({\mathcal O})$, $V$ belongs to the space ${\mathcal C}^n_b (\R^{d+1})$ and where
furthermore $V$ takes its values in a convex compact subset $K$ of ${\mathcal O}$ that contains the origin (for
instance $K$ can be a closed ball centered round the origin).
\end{definition}

All results below extend to the case where in place of a function $V$ that is independent of $\eps$, the
representation \eqref{singularsymbolp} is considered with a function $V_\eps$ that is indexed by $\eps$,
provided that we assume that all functions $\eps \, V_\eps$ take values in a {\it fixed} convex compact
subset $K$ of ${\mathcal O}$ that contains the origin, and $(V_\eps)_{\eps \in (0,1]}$ is a bounded family
of ${\mathcal C}^n_b (\R^{d+1})$. Such singular symbols with a function $V_\eps$ are exactly the kind
of symbols that we manipulated in the construction of exact solutions to the singular system \eqref{a3}.

\subsection{Definition of operators and action on Sobolev spaces}
\label{sect8}

To each symbol $a = (a_{\eps,\gamma})_{\eps \in ]0,1],\gamma \ge 1} \in S^m_n$ given by the formula
\eqref{singularsymbolp} and with values in $\C^{N \times N}$, we associate a singular pseudodifferential
operator $\opeg (a)$, with $\eps \in \, ]0,1]$ and $\gamma \ge 1$, whose action on a function $u \in
{\mathcal S} (\R^{d+1} ; \C^N)$ is defined by
\begin{equation}
\label{singularpseudop}
\opeg (a) \, u \, (x,\theta) := \dfrac{1}{(2\, \pi)^{d+1}} \, \int_{\R^{d+1}} {\rm e}^{i\, (\xi \cdot x +k \, \theta)} \,
\sigma \left( \eps \, V(x,\theta),\xi+\dfrac{k \, \beta}{\eps},\gamma \right) \, \widehat{u} (\xi,k)
\, {\rm d}\xi \, {\rm d}k \, .
\end{equation}
Let us briefly note that for the Fourier multiplier $\sigma (v,\xi,\gamma) =i\, \xi_1$, the corresponding
singular operator is $\partial_{x_1} +(\beta_1/\eps) \, \partial_\theta$. We now describe the action of
singular pseudodifferential operators on Sobolev spaces.

\begin{proposition}
\label{prop13}
Let $n \ge d+1$, and let $a \in S^m_n$ with $m \le 0$. Then $\opeg (a)$ in \eqref{singularpseudop} defines
a bounded operator on $L^2 (\R^{d+1})$: there exists a constant $C>0$, that only depends on $\sigma$
and $V$ in the representation \eqref{singularsymbolp}, such that for all $\eps \in \, ]0,1]$ and for all
$\gamma \ge 1$, there holds
\begin{equation*}
\forall \, u \in {\mathcal S} (\R^{d+1}) \, ,\quad \left\| \opeg (a) \, u \right\|_0 \le \dfrac{C}{\gamma^{|m|}} \, \| u \|_0 \, .
\end{equation*}
\end{proposition}

The constant $C$ in Proposition \ref{prop13} depends uniformly on the compact set in which $V$ takes its
values and on the norm of $V$ in ${\mathcal C}^{d+1}_b$. For operators defined by symbols of order $m>0$
we have:

\begin{proposition}
\label{prop14}
Let $n \ge d+1$, and let $a \in S^m_n$ with $m>0$. Then $\opeg (a)$ in \eqref{singularpseudop} defines
a bounded operator from $H^{m,\eps}(\R^{d+1})$ to $L^2 (\R^{d+1})$: there exists a constant $C>0$, that
only depends on $\sigma$ and $V$ in the representation \eqref{singularsymbolp}, such that for all $\eps \in
\, ]0,1]$ and for all $\gamma \ge 1$, there holds
\begin{equation*}
\forall \, u \in {\mathcal S} (\R^{d+1}) \, ,\quad \left\| \opeg (a) \, u \right\|_0 \le C \, \| u \|_{H^{m,\eps},\gamma} \, .
\end{equation*}
\end{proposition}

The next proposition describes the smoothing effect of operators of order $-1$.

\begin{proposition}
\label{prop15}
Let $n \ge d+2$, and let $a \in S^{-1}_n$. Then $\opeg (a)$ in \eqref{singularpseudop} defines a bounded
operator from $L^2 (\R^{d+1})$ to $H^{1,\eps}(\R^{d+1})$: there exists a constant $C>0$, that only depends
on $\sigma$ and $V$ in the representation \eqref{singularsymbolp}, such that for all $\eps \in \, ]0,1]$ and for
all $\gamma \ge 1$, there holds
\begin{equation*}
\forall \, u \in {\mathcal S} (\R^{d+1}) \, ,\quad \left\| \opeg (a) \, u \right\|_{H^{1,\eps},\gamma} \le C \, \| u \|_0 \, .
\end{equation*}
\end{proposition}

\begin{remark}\label{a4}
\textup{In applications of the pulse calculus, we verify the hypothesis that for $V$ as in \eqref{singularsymbolp},
$V \in \mathcal{C}^n_b(\R^{d+1})$, by showing $V\in H^s(\R^{d+1})$ for some $s>\frac{d+1}{2}+n$.}
\end{remark}

\subsection{Adjoints and products}
\label{sect9}

For proofs of the following results we refer to \cite{CGW2}. The two first results deal with adjoints of singular
pseudodifferential operators while the last two deal with products.

\begin{proposition}
\label{prop18}
Let $a=\sigma(\eps V,\X,\gamma) \in S_n^0$, $n \ge 2\, (d+1)$, where $V\in H^{s_0}(\R^{d+1})$ for some
$s_0>\frac{d+1}{2}+1$, and let $a^*$ denote the conjugate transpose of the symbol $a$. Then $\opeg (a)$
and $\opeg (a^*)$ act boundedly on $L^2$ and there exists a constant $C \ge 0$ such that for all $\eps \in
\, ]0,1]$ and for all $\gamma \ge 1$, there holds
\begin{equation*}
\forall \, u \in {\mathcal S} (\R^{d+1}) \, ,\quad
\left\| \opeg (a)^* \, u -\opeg (a^*) \, u \right\|_0 \le \dfrac{C}{\gamma} \, \| u \|_0 \, .
\end{equation*}

If $n \ge 3\, d +3$, then for another constant $C$, there holds
\begin{equation*}
\forall \, u \in {\mathcal S} (\R^{d+1}) \, ,\quad
\left\| \opeg (a)^* \, u -\opeg (a^*) \, u \right\|_{H^{1,\eps},\gamma} \le C \, \| u \|_0 \, ,
\end{equation*}
uniformly in $\eps$ and $\gamma$.
\end{proposition}

\begin{proposition}
\label{prop19}
Let $a=\sigma(\eps V,\X,\gamma) \in S_n^1$, $n \ge 3\, d +4$, where $V\in H^{s_0}(\R^{d+1})$ for some
$s_0>\frac{d+1}{2}+1$, and let $a^*$ denote the conjugate transpose of the symbol $a$. Then $\opeg (a)$
and $\opeg (a^*)$ map $H^{1,\eps}$ into $L^2$ and there exists a family of operators $R^{\eps,\gamma}$
that satisfies
\begin{itemize}
 \item there exists a constant $C \ge 0$ such that for all $\eps \in \, ]0,1]$ and for all $\gamma \ge 1$, there holds
\begin{equation*}
\forall \, u \in {\mathcal S} (\R^{d+1}) \, ,\quad \left\| R^{\eps,\gamma} \, u \right\|_0 \le C \, \| u \|_0 \, ,
\end{equation*}

 \item the following duality property holds
\begin{equation*}
\forall \, u,v \in {\mathcal S} (\R^{d+1}) \, ,\quad
\langle \opeg (a) \, u,v \rangle_{L^2} -\langle u,\opeg (a^*) \, v \rangle_{L^2} =\langle
R^{\eps,\gamma} \, u,v \rangle_{L^2} \, .
\end{equation*}
In particular, the adjoint $\opeg (a)^*$ for the $L^2$ scalar product maps $H^{1,\eps}$ into $L^2$.
\end{itemize}
\end{proposition}

\begin{proposition}
\label{prop20}
(a)\; Let $a,b \in S_n^0$, $n \ge 2\, (d+1)$, and suppose $b=\sigma(\eps V,\X,\gamma)$ where $V \in
H^{s_0}(\R^{d+1})$ for some $s_0>\frac{d+1}{2}+1$.  Then there exists a constant $C \ge 0$ such that
for all $\eps \in \, ]0,1]$ and for all $\gamma \ge 1$, there holds
\begin{equation*}
\forall \, u \in {\mathcal S} (\R^{d+1}) \, ,\quad
\left\| \opeg (a) \, \opeg (b) \, u -\opeg (a \, b) \, u \right\|_0 \le \dfrac{C}{\gamma} \, \| u \|_0 \, .
\end{equation*}
If $n \ge 3\, d +3$, then for another constant $C$, there holds
\begin{equation*}
\forall \, u \in {\mathcal S} (\R^{d+1}) \, ,\quad
\left\| \opeg (a) \, \opeg (b) \, u -\opeg (a \, b) \, u \right\|_{H^{1,\eps},\gamma} \le C \, \| u \|_0 \, ,
\end{equation*}
uniformly in $\eps$ and $\gamma$.

(b)\;Let $a \in S_n^1,b \in S_n^0$ or $a \in S_n^0,b \in S_n^1$, $n \ge 3\, d +4$, and in each case suppose
$b=\sigma(\eps V,\X,\gamma)$ where  $V\in H^{s_0}(\R^{d+1})$ for some $s_0>\frac{d+1}{2}+1$. Then there
exists a constant $C \ge 0$ such that for all $\eps \in \, ]0,1]$ and for all $\gamma \ge 1$, there holds
\begin{equation*}
\forall \, u \in {\mathcal S} (\R^{d+1}) \, ,\quad
\left\| \opeg (a) \, \opeg (b) \, u -\opeg (a \, b) \, u \right\|_0 \le C \, \| u \|_0 \, .
\end{equation*}
\end{proposition}

\noindent Our final result is G{\aa}rding's inequality.

\begin{theorem}
\label{thm11}
Let $\sigma \in {\bf S}^0$ satisfy $\text{\rm Re} \, \sigma (v,\xi,\gamma) \ge C_K>0$ for all $v$ in a compact
subset $K$ of ${\mathcal O}$. Let now $a \in S_0^n$, $n \ge 2\, d+2$ be given by \eqref{singularsymbolp}, where
$V\in H^{s_0}(\R^{d+1})$ for some $s_0>\frac{d+1}{2}+1$ and is valued in a convex compact subset $K$. Then
for all $\delta >0$, there exists $\gamma_0$ which depends uniformly on $V$, the constant $C_K$ and $\delta$,
such that for all $\gamma \ge \gamma_0$ and all $u \in {\mathcal S}(\R^{d+1})$, there holds
\begin{equation*}
\text{\rm Re } \langle \opeg (a) \, u ;u \rangle_{L^2} \ge (C_K-\delta) \, \| u \|_0^2 \, .
\end{equation*}
\end{theorem}

\subsection{Extended calculus}
\label{extended}

\emph{\quad} In our proof of $L^\infty(x_d;L^2(x',\theta_0))$ estimates for the linearized singular system (Theorem
\ref{est'}), we use a slight extension of the singular calculus. For given parameters $0<\delta_1<\delta_2<1$, we
choose a cutoff $\chi^e (\xi',\frac{k\, \beta}{\eps},\gamma)$ such that
\begin{align}\label{n31}
\begin{split}
&0\leq \chi^e \leq 1\, ,\\
&\chi^e \left( \xi',\dfrac{k\, \beta}{\eps},\gamma \right) =1 \text{ on } \left\{
(\gamma^2 +|\xi'|^2)^{1/2} \leq \delta_1 \, \left| \dfrac{k\, \beta}{\eps} \right| \right\} \, ,\\
&\mathrm{supp } \, \chi^e \subset \left\{ (\gamma^2 +|\xi'|^2)^{1/2} \leq \delta_2 \, \left| \dfrac{k\, \beta}{\eps} \right|
\right\} \, ,
\end{split}
\end{align}
and define a corresponding Fourier multiplier $\chi^e_D$ in the extended calculus by the formula \eqref{singularpseudop}
with $\chi^e (\xi',\frac{k\, \beta}{\eps},\gamma)$ in place of $\sigma(\eps V,X,\gamma)$. Composition laws involving
such operators are proved in \cite{CGW2}, but here we need only the fact that part {\it (a)} of Proposition \ref{prop20}
holds when either $a$ or $b$ is replaced by an extended cutoff $\chi^e$.

\bibliographystyle{plain}
\bibliography{unistapulses}
\end{document}